\documentclass[12pt]{amsart}
\usepackage{a4wide,enumerate,color,graphicx}
\usepackage{amsmath}
\usepackage{scrextend} 
\usepackage{amssymb}
\usepackage{esint}
\usepackage{amsfonts}
\usepackage{amsthm}
\usepackage{mathrsfs}
\usepackage{dsfont}
\usepackage{mhchem}     
\newtheorem{lemma}{Lemma}[section]
\newtheorem{prop}[lemma]{Proposition}
\newtheorem{theo}[lemma]{Theorem}

\newtheorem{rem}[lemma]{Remark}

\DeclareMathOperator{\divv }{div}

\DeclareMathOperator{\poule }{span}

\allowdisplaybreaks

\begin{document}

\title[Multicomponent incompressible flow models]{Well-posedness analysis of multicomponent incompressible flow models}

\author[D. Bothe]{Dieter Bothe}
\address{Mathematische Modellierung und Analysis, Technische Universit\"at Darmstadt, Alarich-Weiss-Str. 10, 64287 Darmstadt, Germany}
\email{bothe@mma.tu-darmstadt.de}

\author[P.-E. Druet]{Pierre-Etienne Druet}
\address{Weierstrass Institute, Mohrenstr. 39, 10117 Berlin, Germany}
\email{pierre-etienne.druet@wias-berlin.de}


\date{\today}


\thanks{This research was supported by the grant DR1117/1-1 of the German Science Foundation}
\maketitle

\begin{abstract} 
In this paper, we extend our study of mass transport in multicomponent isothermal fluids to the incompressible case. For a mixture, incompressibility is defined as the independence of average volume on pressure, and a weighted sum of the partial mass densities stays constant. In this type of models, the velocity field in the Navier-Stokes equations is not solenoidal and, due to different specific volumes of the species, the pressure remains connected to the densities by algebraic formula. By means of a change of variables in the transport problem, we equivalently reformulate the PDE system as to eliminate positivity and incompressibility constraints affecting the density, and prove two type of results: the local--in--time well--posedness in classes of strong solutions, and the global--in--time existence of solutions for initial data sufficiently close to a smooth equilibrium solution. 
\end{abstract}

\section{Multicomponent diffusion in an incompressible fluid}

In this paper we study the well-posedness analysis in classes of strong solutions of \emph{class-one models}\footnote{\emph{Class-one} is a terminology that we adopt from the paper \cite{bothedreyer} to describe the class of multicomponent flow models with single common velocity and temperature. The concept goes back to the work of C.~Hutter.} of mass transport in isothermal, \emph{incompressible} multicomponent fluids.
This investigation is a direct continuation of results obtained recently concerning the compressible case in \cite{bothedruet}, and the weak solvability of the incompressible model in \cite{druetmixtureincompweak}. Performing the incompressible limit (the low-Mach number limit) in models for fluid mixtures and for multicomponent fluids is desirable both from the practical and the theoretical viewpoint. On the one hand, fluid mixtures occurring in applications are often incompressible, and the limit passage reduces the stiffness of the models by eliminating the parameter which is practically infinite. On the other hand, the low-Mach number limit leads to a type of incompressibility condition which has not yet been studied in the context of mathematical analysis for fluid dynamical equations.

We are interested in the second type of issue, that is, the theoretical issues of unique solvability and continuous dependence in classes of strong solutions for the underlying PDEs. The model class for multicomponent transport in fluids here under study is the one proposed in \cite{bothedreyer}, also applied to mixtures with charged constituents in \cite{dreyerguhlkemueller}, \cite{dreyerguhlkemueller19}. Concerning the fundamentals of thermodynamics for fluid mixtures, the reader is referred to these papers, or to the book \cite{giovan}. The model for Mach-number zero (incompressibility constraint) is based on I.~M\"uller's definition of incompressibility as invariance of the volume under pressure variations \cite{mueller}, \cite{gouin}. More directly, we follow the recent example of \cite{dreyerguhlkemueller} (formal limit), and the more general road map proposed in the Section 16 of \cite{bothedreyer}. In \cite{bothedruetFE} we propose a derivation of the incompressible limit starting from a few postulates of mathematical nature about the structure of the Helmholtz free energy. Similar concepts have been exposed and discussed in a few research papers like \cite{mills,josef,donevetal}. Incompressible mixtures are also conceptualised in the book \cite{pekarsamohyl}. The corner stone of these works is that incompressibility for a multicomponent system means the invariance of \emph{average} volume under pressure variations. For a fluid mixture of $N\geq 2$ chemical species $\ce{A}_1,\ldots, \ce{A}_N$, it assumes the form of a \emph{volume constraint}
\begin{align}\label{VOLUME}
\sum_{i=1}^N \rho_i \, \bar{V}_i = 1 \, ,
\end{align}
where $\bar{V}_1,\ldots,\bar{V}_N > 0$ are partial specific volumes of the molecules at reference temperature and pressure. The relation generalises the assumption of a constant mass density considered in other analytical investigations, a. o. \cite{chenjuengel}, \cite{mariontemam}, \cite{herbergpruess}, \cite{bothepruess}.
In the present paper we are interested only in the general case that at least two indices exist such that $\bar{V}_{i_1} \neq \bar{V}_{i_2}$ or, in vectorial notation, that $ \bar{V} \neq \lambda \, 1^N$ for all $\lambda \in \mathbb{R}$, where $\bar{V} =(\bar{V}_1,\, \bar{V}_2,\ldots, \bar{V}_N)$ and $1^N = (1, \, 1,\ldots,1) \in \mathbb{R}^N$. 

\textbf{Bulk.} The convective and diffusive mass transport of these species and the momentum balance are described by the partial differential equations
\begin{alignat}{2}
\label{mass}\partial_t \rho_i + \divv( \rho_i \, v + J^i) & = r_i & &  \text{ for } i = 1,\ldots,N \, ,\\
\label{momentum}\partial_t (\varrho \, v) + \divv( \varrho \, v\otimes v - \mathbb{S}(\nabla v)) + \nabla p & = \sum_{i=1}^N \rho_i \, b^i(x, \, t)  & & \, .
\end{alignat}
The physical system is assumed isothermal with absolute temperature $\theta > 0$. The partial mass densities of the species are denoted $\rho_1,\ldots,\rho_N$. Throughout the paper we shall use the abbreviation $\varrho := \sum_{i=1}^N \rho_i$ for the total mass density. The barycentric velocity of the fluid is called $v$ and the thermodynamic pressure $p$.
In the Navier-Stokes equations, $\mathbb{S}(\nabla v)$ denotes the viscous stress tensor, which we assume for simplicity of Newtonian form. The vector fields $b^1,\ldots,b^N$ are the external body forces. The \emph{diffusions fluxes} $J^1,\ldots,J^N$, that are defined to be the non-convective part of the mass fluxes, must satisfy by definition the necessary side-condition $\sum_{i=1}^N J^i = 0$. A thermodynamic consistent Fick--Onsager closure respecting this constraint is assumed. This approach is described in great generality among others by \cite{bothedreyer}, \cite{dreyerguhlkemueller19} following older ideas by \cite{MR59}, \cite{dGM63}. The diffusions fluxes $J^1,\ldots,J^N$ obey
\begin{align}\label{DIFFUSFLUX}
J^i = - \sum_{j=1}^N M_{i,j}(\rho_1, \ldots, \rho_N) \, (\nabla \mu_j - b^j) \, \text{ for } i =1,\ldots,N \, .
\end{align}
The \emph{Onsager matrix} $M(\rho_1, \ldots,\rho_N)$ is a symmetric, positive semi-definite $N\times N$ matrix for every $(\rho_1,\ldots,\rho_N) \in \mathbb{R}^N_+$. In all known linear closure approaches, this matrix satisfies
\begin{align}\label{CONSTRAINT}
\sum_{i=1}^N M_{i,j}(\rho_1,\ldots,\rho_N) = 0 \text{ for all } (\rho_1,\ldots,\rho_N) \in \mathbb{R}^N_+ \, .
\end{align}
One possibility to compute the special form of $M$ is for instance to invert the Maxwell-Stefan balance equations. For the mathematical treatment of this algebraic system, the reader can consult \cite{giovan}, \cite{bothe11}, \cite{justel13}, \cite{herbergpruess}, \cite{mariontemam}. Or $M$ is constructed directly in the form $P^{\sf T} \, M_0 \, P$, where $M_0$ is a given matrix of full rank, and $P$ is a projector guaranteeing that \eqref{CONSTRAINT} is valid. The paper \cite{bothedruetMS} establishes equivalence relations between the Fick--Onsager and the Maxwell-Stefan constitutive approaches, proposing moreover a novel unifying approach to close the diffusion model.

The quantities $\mu_1,\ldots,\mu_N$ are the \emph{chemical potentials} from which the thermodynamic driving forces for the diffusion phenomena are inferred. For an incompressible system, they are related to the mass densities $\rho_1, \ldots,\rho_N$ and to the pressure via
\begin{align}\label{CHEMPOT}
\mu_i = \bar{V}_i \, p + \partial_{\rho_i}k(\theta, \, \rho_1, \ldots, \rho_N) \, .
\end{align}
Here the function $k$ denotes the the positively homogeneous part of the free energy, which is independent of thermodynamical pressure. A typical choice discussed in \cite{bothedreyer} is
\begin{align}\label{kfunktion}
k(\theta, \, \rho) = & \sum_{i=1}^N \mu_{i}^{\text{ref}} \, \rho_i+ k_B \, \theta \, \sum_{i=1}^N n_i \, \ln y_i   \, ,
\end{align}
where $n_i := \rho_i/m_i$ are the number densities with the molecular masses $m_1,\ldots,m_N > 0$, $y_i = n_i/\sum_{j=1}^N n_j$ are the number fractions, and $\mu_{i}^{\text{ref}} $ are reference values of the chemical potentials. For the mathematical theory in this paper, more general structures in \eqref{kfunktion} will however be admitted. The isothermal Gibbs-Duhem equation: $dp = \sum_{i=1}^N \rho_i \, d\mu_i$ defines the intrinsic relationship between \eqref{VOLUME}, \eqref{CHEMPOT}, and the pressure field. The paper \cite{bothedruetFE} shows that the relation \eqref{CHEMPOT} indeed occurs in the limit case when the bulk free energy density of the system adopts the singular form
\begin{align}\label{hinfty}
\varrho\psi = h^{\infty}(\theta, \, \rho) := \begin{cases}
                     k(\theta, \, \rho) & \text{ if } \sum_{i=1}^N \rho_i \, \bar{V}_i = 1 \, ,\\
                     + \infty & \text{ otherwise.}
                    \end{cases} \, .
\end{align}
The relation \eqref{CHEMPOT} is an equivalent expression of $\mu \in \partial h^{\infty}(\theta, \, \rho)$, where $\partial$ denote the subdifferential of the convex function $h^{\infty}(\theta, \cdot)$, and the function $p = - h^{\infty}(\theta, \, \rho) + \sum_{i=1}^N \rho_i \, \mu_i$ can be understood as a 'Lagrange multiplier' associated with the constraint \eqref{VOLUME}. 

We notice that, multiplying the equations \eqref{mass} with the constants $\bar{V}_i$ and summing up, the local change of volume is described by the equation
\begin{align}\label{isochoric}
 \divv v = - \divv( \sum_{i=1}^N \bar{V}_i \, J^i) + \sum_{i=1}^N \bar{V}_i \, r_i \, .
\end{align}
Effects like diffusion and chemical reactions will induce a local change in the molecular composition, implying a net local change of the volume, independent of a mechanical compression or expansion.

Concerning the presence of reaction terms in \eqref{mass}, we have to mention in respect with the compressible systems considered in \cite{bothedruet} a subtle difference of the incompressible models. For the compressible case, the reactions densities $r_i$ in \eqref{mass} are allowed to be general functions $r_i = r_i(\rho_1, \ldots,\rho_N)$, without influencing \emph{qualitatively} the well--posedness results or the mathematical methods. This is different in the incompressible case. At first, the restriction \eqref{VOLUME} implies that $\mu$ does not depend on $\rho$ only, so that the structure $r = r(\rho)$ does not comply with standard thermodynamically consistent reaction terms. At second, the 'elliptic equation' \eqref{isochoric} defines a differential operator acting on a certain relative chemical potential (variable $\zeta$, details below). This elliptic operator is linear for the pure diffusion case, but turns to non-linear in the presence of reactions of the general form $r = r(\mu)$. 
In this paper, we treat incompressible multicomponent diffusion in itself. We shall address the specific problems raised by chemical reactions in further research. Thus, allowing -- as we shall do -- for certain source terms $r = r(\rho)$ in \eqref{mass} means a bit more mathematical generality, but it remains clear that realistic models of \emph{chemical reactions} require non trivial modifications of the methods used here.

As to the stress tensor $\mathbb{S}$ we shall restrict for simplicity to the standard Newtonian form with constant coefficients. We, however, present methods which are sufficient to extend the results to the case of density and composition dependent viscosity coefficients.

\textbf{Boundary and initial conditions.} We investigate the problem \eqref{mass}, \eqref{momentum} in a cylindrical domain $Q_T := \Omega \times ]0,T[$ where $T$ is a finite time and $\Omega \subset \mathbb{R}^3$ a bounded domain. It is possible to treat the case $\Omega \subset \mathbb{R}^d$ for general $d\geq 2$ with similar methods. We consider initial conditions
\begin{alignat}{2}\label{initial0rho}
\rho_i(x,\, 0) & = \rho^0_i(x) & & \text{ for } x \in \Omega, \, i = 1,\ldots, N \, ,\\
\label{initial0v} v_j(x, \, 0) & = v^0_j(x) & & \text{ for } x \in \Omega, \, j = 1,2,3\, .
\end{alignat}
For simplicity, we consider the linear homogeneous boundary conditions
\begin{alignat}{2}
\label{lateral0v} v & = 0 & &\text{ on } S_T := \partial \Omega \times ]0,T[ \, ,\\
\label{lateral0q} \nu \cdot J^i & = 0 & &\text{ on } S_T \text{ for } i = 1,\ldots,N \, .
\end{alignat}
As a matter of fact, these simplifying choices oblige us to make a further restriction. To see this, we recall the relation \eqref{isochoric}, that we integrate over $\Omega$. If there is no mass flux through the boundary, we see that $\int_{\Omega} \sum_{i=1}^N \bar{V}_i \, r_i(x, \, t) \, dx = 0$. This condition cannot be enforced for a general $r = r(\rho)$, unless we assume that $r$ takes values in $\{\bar{V}\}^{\perp}$. Recalling that realistic models for chemical reactions are to be treated in an upcoming paper, we here restrict to the case that $r(\rho) \cdot \bar{V} = 0$ for all $\rho$.

\section{State of the art and our main result}

\subsection{A review of prior investigations and our method}

Up to few exceptions, models for incompressible multicomponent fluids have not been investigated in mathematical analysis. For a mathematical treatment in the case of the constraint $\varrho = {\rm const}$, which corresponds to choosing $\bar{V}_1 = \ldots = \bar{V}_N$ in \eqref{VOLUME}, the reader might consult the papers \cite{chenjuengel} and  \cite{mariontemam} (global weak solution analysis) and \cite{herbergpruess}, and \cite{bothepruess} (local--in--time well-posedness).\footnote{In the latter paper the phase change liquid/gas is actually in the focus. All references are based on the equivalent Maxwell-Stefan structure for the diffusion fluxes, rather than the Fick-Onsager one.} 
From the viewpoint of the mathematical structure, the case $\varrho = {\rm const}$ exhibits profoundly different features than the general relation \eqref{VOLUME}. The principal difference is that \eqref{DIFFUSFLUX}, \eqref{CONSTRAINT} and \eqref{CHEMPOT} imply the decoupling of the pressure and of the diffusion fluxes. The Navier-Stokes equations reduce to their single component solenoidal variant and can be solved independently. Of course, this does not mean that $\varrho = {\rm const}$ cannot be a good approximation under special circumstances. In \cite{bothesoga} for instance, a class of multicomponent mixtures has been introduced for which the use of the incompressible Navier-Stokes equation is realistic: Incompressibility is assumed for the solvent only, and diffusion is considered against the solvent velocity. See also the discussion in the paragraph 4.8 of \cite{pekarsamohyl} on incompressible mixtures.

In the case that $\bar{V}$ is not parallel to $1^N$, \eqref{DIFFUSFLUX} implies that the pressure affects the diffusion fluxes via the chemical potentials. A corollary of this fact is that if we multiply the equations \eqref{mass} with the constants $\bar{V}_i$ and sum up, we obtain \eqref{isochoric} for the local change of volume. Moreover,
\begin{enumerate}[(a)]
 \item \label{featurea} the viscous stress tensor does not simplify to the symmetric velocity gradient;
 \item \label{featureb} the total mass density is calculated from the continuity equation $\partial_t \varrho + \divv(\varrho \, v) = 0$;
 \item \label{featurec} the pressure remains partly connected to the other variables by an algebraic formula.
\end{enumerate}
Our main method to approach the PDE problem is a switch of variables in the transport problem as already applied in \cite{bothedruet}. Instead of the original variables $(\rho_1,\ldots,\rho_N)$ and $(p, \, v_1,\,v_2,\, v_3)$, we regard $N-1$ linear combinations of the chemical potentials $(\mu_1, \ldots,\mu_N)$, the mass density $\varrho$ and the velocity field as main variables. After the transformation we obtain for the new free variables $(\varrho, \, q_1, \ldots, q_{N-2}, \, \zeta, \, v_1, \, v_2, \, v_3)$ -- instead of \eqref{mass}, \eqref{momentum} -- the equations (here without external forcing and chemical reactions)
\begin{alignat*}{2}
\partial_t R_{k}(\varrho, \, q) + \divv( R_{k}(\varrho, \, q) \, v - \widetilde{M}_{k,\ell}(\varrho, \, q) \, \nabla q_{\ell} - A_k(\varrho, \, q)  \, \nabla \zeta) & = 0 & & \text{ for } k = 1,\ldots,N-2 \, ,\\
\divv(v -\, A(\varrho, \, q) \cdot \nabla q - d(\varrho, \, q) \, \nabla \zeta)  & = 0 & & \, ,\\
\partial_t \varrho + \divv(\varrho \, v) & = 0 & &  \, ,\\
\partial_t (\varrho \, v) + \divv( \varrho \, v\otimes v - \mathbb{S}(\nabla v)) + \nabla P(\varrho, \, q) + \nabla \zeta & = 0  & & \, .
\end{alignat*}
The nonlinear field $R$ and the function $P$, the vector field $A$, the positive matrix $\widetilde{M}$, and the positive coefficient function $d$ will be constructed below, combining certain linear projection operators with the inverse map for the algebraic equations $\mu = \bar{V} \, p + \nabla_{\rho} k(\rho)$. 
We are then faced with a nonlinear PDE system of mixed parabolic--elliptic--hyperbolic type. All variables are \emph{unconstrained}, but for the restriction $\varrho_{\min} < \varrho < \varrho_{\max}$ on the total mass density. Here, the constants $0< \varrho_{\min} < \varrho_{\max} < + \infty$ are the thresholds of the total mass for states $\rho_1,\ldots,\rho_N$ that satisfy the constraint \eqref{VOLUME}: 
\begin{align*}
 \varrho_{\min} := \min \{ \sum_{i=1}^N \rho_i \, : \, \rho_i \geq 0, \, \sum_{i=1}^N \rho_i \, \bar{V}_i = 1\} = & \frac{1}{\max \bar{V}}, \\
 \varrho_{\max} := \max \{ \sum_{i=1}^N \rho_i \, : \, \rho_i \geq 0, \,  \sum_{i=1}^N \rho_i \, \bar{V}_i = 1\} = & \frac{1}{\min \bar{V}} \, .
\end{align*}
Comparing with the paper \cite{bothedruet} on compressible class-one models based on a similar reformulation, we see that the incompressible limit corresponds structurally to the case that one of the relative chemical potentials is subject to an elliptic -- instead of a parabolic -- equation, and the total mass density is confined to a bounded interval.

For an overview of possible methods to study the transformed PDE system, we refer to our study \cite{bothedruet}. We shall follow here the same principal road map, but profound transformations are necessary to deal with the constraint on $\varrho$, since it implies that the nonlinear functions occurring in the transformed system are singular for ${\rm dist}(\varrho, \, \{\varrho_{\min}, \, \varrho_{\max}\}) \rightarrow 0$. The solution operator to the continuity equation, however, does not 'see' these thresholds, which is the source of additional problems when we attempt to linearise. Moreover, we must construct a solution operator for the parabolic--elliptic subsystem of general form for $(q, \, \zeta)$, while the reduced transport problem in \cite{bothedruet} was purely parabolic. Nontrivial extensions of the method are therefore necessary to deal with the incompressible case.

We shall study the problem in the class proposed in the paper \cite{solocompress} for Navier-Stokes: $W^{2,1}_p$ with $p$ larger than the space dimension for the velocity and $W^{1,1}_{p,\infty}$ for the density. For the variable $q_1, \ldots, q_{N-2}$, we also choose the parabolic setting of $W^{2,1}_p$. For the elliptic component $\zeta$, we choose the state space $W^{2,0}_p$. In these classes, we are able to prove the local existence for strong solutions. In general, we obtain only a short--time well--posedness result, and boundedness in the state space is not sufficient to guarantee that the solution can be extended to a larger time interval. This is due to the constraint $\varrho_{\min} < \varrho < \varrho_{\max}$: A strong solution with bounded state space norm might break down if the density reaches the thresholds. 
However, it is to note that for choices of the tensor $M$ reflecting the physically expected behaviour that, in the dilute limit, a diffusion flux is linearly proportional with the mass density of the vanishing species, we are able to show that a sufficiently smooth solution ($p > 5$) bounded in the state space cannot reach the critical values in finite time. Thus, a kind of maximum principle is available for the system.

We shall also prove the global existence under the condition that the initial data are sufficiently near to an equilibrium (stationary) solution. However, since this result relies on stability estimates in the state space, we need to assume higher regularity of the initial data in order to obtain some stability from the continuity equation. Therefore, these solutions exist on arbitrary large time intervals, but do not enjoy the extension property. 
We shall not make use of the Lagrangian coordinates but employ the approach of controlled growth in time of the solution by means of \emph{a priori} estimates.

Let us finally mention also the paper \cite{feima16}, devoted to binary mixtures. Starting from different modelling principles in the spirit of \cite{josef}, the authors derive for $N = 2$ a similar PDE system. The variable $q$ does not occur, and the coefficient $d$ is assumed constant. The authors prove for this system the global existence of weak solutions if the singularity of $P(\varrho)$ at the thresholds is sufficiently strong. 

The weak solution analysis for the general system is considered in the paper \cite{druetmixtureincompweak}.

\subsection{Main results}\label{MATHEMATICS}

We denote $Q= Q_T = \Omega \times ]0,T[$ with a bounded domain $\Omega \subset \mathbb{R}^3$ and $T >0$ a finite time. We use the standard Sobolev spaces $W^{m,p}(\Omega)$ for $m \in \mathbb{N}$ and $1\leq p\leq +\infty$, and the Sobolev-Slobodecki spaces $W^s_p(\Omega)$ for $s >0$ non-integer. If $\Omega$ is a domain of class $\mathcal{C}^2$, the spaces $W^{s}_p(\partial \Omega)$ are well defined for $0 \leq s \leq 2$.

With a further index $1 \leq r \leq +\infty$, we use the parabolic Lebesgue spaces $L^{p,r}(Q)$ (space index first: $L^p(Q) = L^{p,p}(Q)$). For $\ell = 1, \, 2, \ldots$ and $1 \leq p \leq +\infty$ we introduce the parabolic Sobolev spaces
\begin{align*}
W^{2\ell, \ell}_p(Q) := & \{ u \in L^p(Q) \, : \, D_t^{\beta}D^{\alpha}_x u \in L^p(Q) \, \forall\,\, 1 \leq 2 \, \beta + |\alpha| \leq 2 \, \ell \} \, ,\\
\|u\|_{W^{2\ell,\ell}_p(Q)} := & \sum_{0\leq 2 \, \beta + |\alpha| \leq 2 \, \ell} \|D_t^{\beta}D^{\alpha}_x u \|_{L^p(Q)} \, ,
\end{align*}
and, with a further index $1 \leq r < \infty$, the spaces
\begin{align*}
W^{1}_{p,r}(Q) = W^{1,1}_{p,r}(Q) := & \{u \in L^{p,r}(Q) \, : \,\sum_{0\leq \beta + |\alpha| \leq \ell} D^{\alpha}_x \, D^{\beta}_t u \in L^{p,r}(Q) \} \,  ,\\
\|u\|_{W^{\ell,\ell}_{p,r}(Q)} := & \sum_{0\leq \beta + |\alpha| \leq \ell} \|D_t^{\beta}D^{\alpha}_x u \|_{L^{p,r}(Q)} \, .
\end{align*}
In these notations, the space integrability index always comes first. For $r = + \infty$, $W^{\ell,\ell}_{p,\infty}(Q)$ denotes the closure of $C^{\ell}(\overline{Q})$ with respect to the norm above and, thus,
\begin{align*}
W^{1,1}_{p,\infty}(Q) := & \{u \in L^{p,\infty}(Q) \, : \,\sum_{0\leq \beta + |\alpha| \leq 1} D^{\alpha}_x \, D^{\beta}_t u \in C([0,T]; \, L^p(\Omega)) \} \, .
\end{align*}
We also encounter, for $\ell = 1,2$ and $1 \leq p < + \infty$,
\begin{align*}
 W^{\ell,0}_{p}(Q) := & \{u \in L^{p}(Q) \, : \,\sum_{0\leq|\alpha| \leq \ell} D^{\alpha}_x u \in L^{p}(Q) \} \, , \\
\|u\|_{W^{\ell,0}_{p}(Q)} := & \sum_{0\leq |\alpha| \leq \ell} \|D^{\alpha}_x u \|_{L^{p}(Q)}  \, .
\end{align*}
We denote by $C(\overline{Q}) = C^{0,0}(\overline{Q})$ the space of continuous functions over $\overline{Q}$ and, for $\alpha, \, \beta \in [0, \, 1]$, the H\"older spaces are defined by $C^{\alpha, \, \beta}(\overline{Q}) :=  \{ u \in C(\overline{Q}) \, : \, [ u ]_{C^{\alpha,\beta}(\overline{Q})} < + \infty\}$ with
\begin{align*}
& [u]_{ C^{\alpha, \, \beta}(\overline{Q})} =  \sup_{t \in [0, \, T], \, x,y \in \Omega} \frac{|u(t, \, x) - u(t, \, y)|}{|x-y|^{\alpha}} + \sup_{x \in \Omega, \, t,s \in [0,\, T]} \frac{|u(t, \, x) - u(s, \, x)|}{|t-s|^{\beta}} \, .
\end{align*}
Some brief remarks on notation: 
\begin{enumerate}[(1)]
\item All H\"older continuity properties are global. For the sake of notation we identify $C^{\alpha, \, \beta}(Q)$ with $C^{\alpha, \, \beta}(\overline{Q})$. 
\item Whenever confusion is impossible, we shall also employ for a function $f$ of the variables $x \in \Omega$ and $t \geq 0$ the notations $f_x = \nabla f$ for the spatial gradient, and $f_t$ for the time derivative.
\item For maps like $R$, $\widetilde{M}$ which depend on $\varrho$ and $q$, the derivatives are denoted by $R_{\varrho}$, $\widetilde{M}_q$.
\end{enumerate}


Due to \eqref{CONSTRAINT}, the matrix $M(\rho)$ possesses only $N-1$ positive eigenvalues that moreover might degenerate for vanishing species. The orthogonal projection on the $N-1$ dimensional linear space $\poule\{1^N\}^{\perp}$ in $\mathbb{R}^N$ is defined via
\begin{align*}
\mathcal{P}_{\{1^N\}^{\perp}}: \mathbb{R}^N \rightarrow \{1^N\}^{\perp}, \quad \mathcal{P}_{\{1^N\}^{\perp}} = \text{Id}_{\mathbb{R}^N} - \frac{1}{N} \, 1^N \otimes 1^N \, .
\end{align*}
The vector $\bar{V}$ occurring in \eqref{VOLUME} defines another singular direction in the model preventing parabolicity. We denote by $ \mathcal{P}_{\{1^N, \, \bar{V}\}^{\perp}}$ the orthogonal projection onto the $N-2$ dim. space $\{1^N, \, \bar{V}\}^{\perp}$. We also introduce the notations
\begin{align}
\begin{split}\label{S0}
 \mathbb{R}^N_+ := & \{\rho = (\rho_1, \ldots, \rho_N) \in \mathbb{R}^N \, : \, \rho_i > 0 \text{ for } i = 1,\ldots,N\} \, , \\
  \overline{\mathbb{R}}^N_+ := & \{\rho = (\rho_1, \ldots, \rho_N) \in \mathbb{R}^N_+ \, : \, \rho_i \geq 0 \text{ for } i = 1,\ldots,N\} \, ,\\
 S_1  := & \{\rho = (\rho_1, \ldots, \rho_N) \in \mathbb{R}^N_+ \, : \, \sum_{i=1}^N \rho_i = 1 \}  \, , \\
  S_{\bar{V}} := & \{\rho = (\rho_1, \ldots, \rho_N) \in \mathbb{R}^N_+ \, : \, \sum_{i=1}^N \bar{V}_i \, \rho_i =1\} \, .
\end{split}
  \end{align}
The surface $S_{\bar{V}}$ is the domain of existence for the incompressible state. It is readily seen that $\rho \in S_{\bar{V}}$ implies for the variable $\varrho := \sum_{i=1}^N \rho_i$ the inequalities
\begin{align}\label{thresholds}
 \varrho_{\min} = \frac{1}{\max_{j=1,\ldots,N} \bar{V}_j} < \varrho <  \varrho_{\max} = \frac{1}{\min_{j=1,\ldots,N} \bar{V}_j} \text{ for all } \rho \in S_{\bar{V}} \, .
\end{align}
Our first main result is devoted to the short-time existence of a strong solution. (In order to avoid notational confusion with the pressure field, the integrability index is called $s$ in the next statements.)
\begin{theo}\label{MAIN}
We fix $s > 3$ and $T > 0$ and assume that 
\begin{enumerate}[(a)]
             \item $\Omega \subset \mathbb{R}^3$ is a bounded domain of class $\mathcal{C}^2$;
             
             \item $M: \, \mathbb{R}^N_+ \rightarrow \mathbb{R}^{N\times N}$ is a mapping of class $C^2(\mathbb{R}^N_+; \, \mathbb{R}^{N\times N})$ into the positive semi-definite matrices of rank $N-1$ with constant kernel $\poule\{1^N\} = \{(1, \ldots, 1)\}$;

\item $k: \, \mathbb{R}^N_+ \rightarrow \mathbb{R}$ is of class $C^3(\mathbb{R}^{N}_+)$, positively homogeneous, convex in its domain $\mathbb{R}^N_+$, and $\liminf_{m \rightarrow +\infty} |\nabla_{\rho} k(y^m)| = +\infty$ for all sequences $\{y^m\} \subset S_1$ approaching the relative boundary of $S_1$;

\item $r: \, \mathbb{R}^N_+ \rightarrow \mathbb{R}^{N}$ is a mapping of class $C^1(\mathbb{R}^{N}_+)$ into $\poule\{1^N, \, \bar{V}\}^{\perp}$;

\item \label{force} The forcing $b$ satisfies $\mathcal{P}_{\{1^N\}^{\perp}}  \, b \in W^{1,0}_{s}(Q_T; \, \mathbb{R}^{N\times 3})$ and $b -  \mathcal{P}_{\{1^N\}^{\perp}} \, b \in L^s(Q_T; \, \mathbb{R}^{N\times 3})$. For simplicity, we assume $\nu(x) \cdot \mathcal{P}_{\{1^N\}^{\perp}} \,b(x, \, t) = 0$ for $x \in \partial \Omega$ and $\lambda_1-$almost all $t \in ]0, \, T[$.
\item The initial data $\rho^0_{1}, \ldots \rho^0_{N}: \, \overline{\Omega} \rightarrow S_{\bar{V}}$ are positive measurable functions satisfying the following conditions:
\begin{itemize}
 \item The initial total mass density $\varrho_0 := \sum_{i=1}^N \rho_{i}^0$ is of class $W^{1,s}(\Omega)$;
\item The vector field $e^{0} := \partial_{\rho}k(\rho^0_{1}, \ldots \rho^0_{N})$ satisfies $\mathcal{P}_{\{1^N, \, \bar{V}\}^{\perp}} \,e^{0} \in W^{2-2/s}_s(\Omega; \, \mathbb{R}^N)$;
\item The compatibility condition $\nu(x) \cdot \mathcal{P}_{\{1^N, \, \bar{V}\}^{\perp}}\nabla e^0(x) = 0$ is valid in $W^{1-3/s}_s(\partial \Omega; \, \mathbb{R}^N)$ in the sense of traces;
\end{itemize}
\item The initial velocity $v^0$ belongs to $W^{2-2/s}_{s}(\Omega; \, \mathbb{R}^3)$ with $v^0 = 0$ in $W^{2-3/s}_{s}(\partial \Omega; \, \mathbb{R}^3)$.
\end{enumerate}
Then, there exists $T^* \in (0, \,  T]$ such that the problem \eqref{mass}, \eqref{momentum} with closure relations \eqref{DIFFUSFLUX}, \eqref{CHEMPOT}, incompressibility constraint \eqref{VOLUME} and boundary conditions \eqref{initial0rho}, \eqref{initial0v},  \eqref{lateral0v}, \eqref{lateral0q} possesses a unique solution $(\rho, \, p, \, v)$ of class
\begin{align*}
 \rho \in W^{1}_{s}(Q_{T^*}; \, S_{\bar{V}}), \quad p \in W^{1,0}_s(Q_{T^*}), \, \quad v \in W^{2,1}_s(Q_{T^*}; \, \mathbb{R}^3) \, ,
\end{align*}
such that $\mu := p \, \bar{V} + \partial_{\rho}k(\rho)$ satisfies $\mathcal{P}_{\{1^N\}^{\perp}}\mu \in  W^{2,0}_s(Q_{T^*}; \, \mathbb{R}^N)$. The solution can be uniquely extended to a larger time interval whenever the two following conditions are fulfilled: 
\begin{enumerate}[(i)]
\item \label{firstcondi} $\varrho_{\min} < \inf\{\varrho(x, \, t) \, : \, x \in \overline{\Omega}, \, t \in [0, \, T^*[\}$ {\rm and} $ \sup\{\varrho(x, \, t) \, : \, x \in \overline{\Omega}, \, t \in [0, \, T^*[\} < \varrho_{\max}$;
\item \label{secondcondi} There is $\alpha > 0$ such that the quantity $$\|\mathcal{P}_{\{1^N, \, \bar{V}\}^{\perp}} \mu\|_{C^{\alpha,\frac{\alpha}{2}}(Q_{t})} + \|\nabla \mathcal{P}_{\{1^N, \, \bar{V}\}^{\perp}} \mu\|_{L^{\infty,s}(Q_{t})} + \|v\|_{L^{z\, s,s}(Q_{t})} + \int_{0}^{t} [\nabla v(\tau)]_{C^{\alpha}(\Omega)} \, d\tau < \infty $$ stays finite as $t \nearrow T^*$. Here $z = z(s)$ is defined via $z = 3/(s-2)$ for $3 < s < 5$, $z > 1$ arbitrary for $s = 5$ and $z = 1$ if $s > 5$.  
\end{enumerate}
\end{theo}
 It is to note that the possibility to extend the solution is not -- like in the compressible case -- reducible to the smoothness criterion \eqref{secondcondi}. If \eqref{firstcondi} is failing, even a smooth solution can break down if its total mass density reaches the critical values $\{\varrho_{\min}, \, \varrho_{\max}\}$. This singularity plays an important role also in the context of the weak solution analysis (see \cite{druetmixtureincompweak}). However, we provide an important complement for \emph{physically motivated choices} of the mobility matrix $M$ and of the function $k$. Here the boundedness in the natural state space norm is sufficient to guarantee the extension property.
 \begin{theo}\label{MAINPr}
 In the situation of Theorem \ref{MAIN} we assume, in addition, that $s > 5$ and that $k$ is the function defined in \eqref{kfunktion}. We define a matrix $B_{i,j}(\rho) := M_{i,j}(\rho)/\rho_j$ for $i,j=1,\ldots,N$, and we assume that there is a continuous function $C = C(|\rho|)$, bounded on compact subsets of $\overline{\mathbb{R}}^N_+ \setminus \{0\}$, such that
\begin{align*}
|B_{i,j}(\rho)| + \rho_k \, |\partial_{\rho_k} B_{i,j}(\rho)| \leq C(|\varrho|) \text{ for all } i, \, j, \, k \in \{1,\ldots,N\} \text{ and all } \rho \in \mathbb{R}^N_+\, .
\end{align*}
Then the strong solution of Theorem \ref{MAIN} can be extended beyond $T^*$ whenever $$\lim_{t \nearrow T^*} \|\mathcal{P}_{\{1^N, \, \bar{V}\}^{\perp}} \mu\|_{W^{2,1}_s(Q_t; \, \mathbb{R}^{N})} + \|\mathcal{P}_{\{1^N\}^{\perp}} \, \mu\|_{W^{2,0}_s(Q_t; \, \mathbb{R}^{N})}  + \|v\|_{W^{2,1}_s(Q_t; \, \mathbb{R}^{3})} < + \infty \, .$$
 \end{theo}

Our second main result concerns global existence under suitable restrictions on the data. An equilibrium solution for \eqref{mass}, \eqref{momentum} is defined as a vector $(\rho_1^{\text{eq}}, \ldots, \rho_N^{\text{eq}}, \, p^{\text{eq}}, \, v_1^{\text{eq}}, \, v_2^{\text{eq}}, \, v_3^{\text{eq}})$ of functions defined in $\Omega$ such that
\begin{align*}
 \rho^{\text{eq}} \in W^{1,s}(\Omega; \, S_{\bar{V}}), \quad p^{\text{eq}} \in W^{1,s}(\Omega),  \quad v^{\text{eq}} \in W^{2,s}(\Omega; \, \mathbb{R}^3) \, ,
\end{align*}
the vector $\mu^{\text{eq}} := p^{\text{eq}} \, \bar{V} + \nabla_{\rho}k(\theta, \, \rho^{\text{eq}})$ satisfies $ \mathcal{P}_{\{1^N\}^{\perp}} \,\mu^{\text{eq}} \in  W^{2,s}(\Omega; \, \mathbb{R}^N)$ and the relations
\begin{align}
\label{massstat} \divv( \rho^{\text{eq}}_i \, v^{\text{eq}} - \sum_{j=1}^N M_{i,j}(\rho^{\text{eq}}) \, (\nabla \mu_j^{\text{eq}} - b^j(x))) =  0 \text{ for } i = 1,\ldots,N 
\end{align}
and
\begin{align}
\label{momentumsstat}  \divv( \varrho^{\text{eq}} \, v^{\text{eq}}\otimes v^{\text{eq}} - \mathbb{S}(\nabla v^{\text{eq}})) + \nabla p^{\text{eq}} =  \sum_{i=1}^N \rho_i^{\text{eq}} \, b^i(x)  \, 
\end{align}
are valid in $\Omega$. The boundary conditions are $$v^{\text{eq}} = 0 \text{ and } \nu(x) \cdot M_{i,j}(\rho^{\text{eq}}) \, (\nabla \mu_j^{\text{eq}} - b^j(x)) = 0 \text{ on }\partial \Omega \, .$$ 
We show that the problem \eqref{mass}, \eqref{momentum} possesses a unique strong solution on an arbitrary large, but finite time interval if the distance of the initial data to an equilibrium solution is sufficiently small, and if both initial conditions and equilibrium solution are smooth enough.  
\begin{theo}\label{MAIN3}
We adopt the assumptions of Theorem \ref{MAIN}, but assume also that $r \equiv 0$ and that $b = b(x)$ does not depend on time with $b \in W^{1,s}(\Omega; \, \mathbb{R}^{N\times 3})$. In addition, we assume that an equilibrium solution $(\rho^{\text{eq}}, \, p^{\text{eq}}, \,  v^{\text{eq}}) \in W^{1,s}(\Omega; \, S_{\bar{V}}) \times W^{1,s}(\Omega) \times  W^{2,s}(\Omega; \, \mathbb{R}^3)$ is given. The associated total mass $\varrho^{\text{eq}} := \sum_{i=1}^N \rho^{\text{eq}}_i$ and the velcocity  possess the additional regulartiy $\varrho^{\text{eq}} \in  W^{2,s}(\Omega)$ and $v^{\text{eq}} \in W^{3,s}(\Omega; \, \mathbb{R}^3)$. Assume that the initial data satisfies $\varrho^0 \in  W^{2,s}(\Omega)$ and $v^0 \in W^{2,s}(\Omega; \, \mathbb{R}^3)$. Then, for every $0 < T < + \infty$, there exists $R_1 > 0$, depending on $T$ and all data in their respective norms, such that under the condition
\begin{align*}
 \|\mathcal{P}_{\{1^N, \, \bar{V}\}^{\perp}}\, (e^0 - \mu^{\text{eq}})\|_{W^{2-\frac{2}{s}}_s(\Omega; \, \mathbb{R}^N)} + \|\varrho^0 - \varrho^{\text{eq}}\|_{W^{1,s}(\Omega)} +  \|v^0 - v^{\text{eq}}\|_{W^{2-\frac{2}{s}}_s(\Omega; \, \mathbb{R}^3)} \leq R_1
\end{align*}
the problem \eqref{mass}, \eqref{momentum} with incompressibility constraint \eqref{VOLUME}, closure relations \eqref{DIFFUSFLUX}, \eqref{CHEMPOT} and the initial and boundary conditions \eqref{initial0rho}, \eqref{initial0v},  \eqref{lateral0v}, \eqref{lateral0q} possesses a global unique solution of the same class as in Theorem \ref{MAIN}.
\end{theo}

\subsection{Road map}

In sections \ref{PrelimFE} and \ref{changevarsec} we show how to reformulate the original system such that it becomes easier to tackle via functional analytic methods. The functional setting is discussed in section \ref{funcsetting}. In section \ref{twomaps}, we introduce two ways to linearise the PDE system and reformulate the initial--boundary--value problem as a fixed point problem in the state space. Both fixed point equations exploit the parabolic substructure for the variables $(q, \, v)$ and treat the linear equations for $(\zeta, \,\varrho)$ as side conditions. In the first method, used to prove the short-time well posedness, all lower--order nonlinearities are frozen. For the proof of Theorem \ref{MAIN3} on small perturbations, a somewhat more elaborated linearisation principle is used in order to exhibit some stability estimates. 

The estimates for the linearised principal part of the system are presented in section \ref{ESTI}. Here we can rely partly on our work in \cite{bothedruet} for the compressible system, but have to discuss the additional problems caused by the presence of an elliptic equation and of a density constraint in the continuity equation. Section \ref{contiT} shows the self mapping estimate for the first fixed point equation, which yields the well posedness result in section \ref{FixedPointIter}. The extension criteria proved for the solution in the same section \ref{FixedPointIter} deserve attention in their own right. The proof of the global well-posedness result for small data, or rather small perturbations, is given in section \ref{contiT1}. Finally, some reminder, tools, and purely technical statements are compiled in the Appendix.

\section{The singular free energy function and its conjugate} \label{PrelimFE}

In comparison to the analysis of compressible models in \cite{bothedruet}, a main specificity of the incompressible model concerns the bulk free energy density and the definition \eqref{CHEMPOT} of the chemical potentials. With $k: \, \mathbb{R}^N_+ \rightarrow \mathbb{R}$ given, we introduce a bulk free energy density defined for $\rho \in \mathbb{R}^N_+$ of the form
\begin{align*}
h^{\infty}(\rho) := \begin{cases}
k(\rho) & \text{ if } \sum_{i=1}^N \rho_i \, \bar{V}_i = 1\, ,\\
+ \infty & \text{ otherwise.}
\end{cases}
\end{align*}
The function $h^{\infty}$ is singular, but the subdifferential $\partial h^{\infty}$ is non-empty for every $\rho$ satisfying the incompressiblity constraint $\sum_{i=1}^N \rho_i \, \bar{V}_i = 1$. If the function $k$ is continuously differentiable, it can be shown that $\mu \in \partial h^{\infty}(\rho)$ if and only if there exists $p \in \mathbb{R}$ such that $\mu_i = p \, \bar{V}_i + \partial_{\rho_i}k(\rho)$ for $i=1,\ldots,N$. It can easily be verified that the number $p$ can be characterised as follows:
\begin{align*}
p & = \sup_{\rho \in \mathbb{R}^N_+} \{\mu \cdot \rho - h^{\infty}(\rho)\} = \sup_{\rho \in \mathbb{R}^N_+, \sum_{i=1}^N \rho_i \, \bar{V}_i = 1} \{\mu \cdot \rho - k(\rho)\} = (h^{\infty})^*(\mu) \, ,
\end{align*}
where $(h^{\infty})^*$ is the convex conjugate of $h^{\infty}$. For systematic discussions and a proof of these elementary statements, we refer to \cite{bothedruetFE}.

Our approach essentially relies on the properties of the dual free energy function $f := (h^{\infty})^*$ on $\mathbb{R}^N$. We shall recall three statements of the paper \cite{bothedruetFE}. Proofs are provided in the Appendix, Section \ref{FE} for the reader's convenience. In the special case that the gradient of $k$ is explicitly invertible on $S_1$ (see \eqref{S0}), the statements can also be proved by direct algebraic computations yielding in many cases explicit formulae; see the Section 4 in \cite{druetmixtureincompweak} for a complete characterisation of the example \eqref{kfunktion}.
\begin{lemma}\label{smoothness}
We assume that $k: \, \mathbb{R}^N_+ \rightarrow \mathbb{R}$ is a positively homogeneous convex function of class $C^3(\mathbb{R}^N_+)$. We moreover assume that the restriction of $k$ to the surface $S_1$ is essentially smooth, meaning that $|\nabla_{\rho} k(y^m)| \rightarrow +\infty$ for sequences $\{y^m\}_{m \in \mathbb{N}} \subset S_1$ such that $\min_{i=1,\ldots,N} y^m_i \rightarrow 0$ as $m \rightarrow +\infty$. For $\mu \in \mathbb{R}^N$, we define $f(\mu) := \sup_{\rho \in S_{\bar{V}}} \{\mu \cdot \rho - k(\rho)\}$. Then the function $f$ belongs to $C^3(\mathbb{R}^N)$, and $\nabla_{\mu} f$ maps onto $S_{\bar{V}}$.
\end{lemma}
\begin{lemma}\label{Hessianandgradient}
We adopt the same assumptions as in Lemma \ref{smoothness}. Then
\begin{enumerate}[(1)]
\item $f(\mu + s \, \bar{V}) = f(\mu) + s$ {\rm and} $\nabla_{\mu} f(\mu + s \, \bar{V}) =\nabla_{\mu} f(\mu)$ for all $\mu \in \mathbb{R}^N$ and all $s \in \mathbb{R}$; 
\item The Hessian $D^2f(\mu)$ is positive semi-definite for all $\mu \in \mathbb{R}^N$, with ${\rm ker}(D^2f(\mu)) = {\rm span}\{\bar{V}\}$;
\end{enumerate}
\end{lemma}
The next Lemma is a main tool for our reformulation of the PDE system.
\begin{lemma}\label{functionf2}
We adopt the assumptions of Lemma \ref{smoothness}. If $\mu \in \mathbb{R}^N$, $\rho \in S_{\bar{V}}$ and $p$ are related via \eqref{CHEMPOT}, then $p = f(\mu)$ and $\rho = \nabla_{\mu} f(\mu)$.
\end{lemma}

\section{Change of variables for the incompressible model}\label{changevarsec}

We propose a reformulation of the equations \eqref{mass}, \eqref{momentum} subject to the constitutive equations \eqref{DIFFUSFLUX}, \eqref{CHEMPOT} and to the volume constraint \eqref{VOLUME} in order to eliminate the positivity constraints on $\rho$, the singularity due to $M \, 1^N = 0$ (cf. \eqref{CONSTRAINT}), and the singularity direction due to the incompressibility \eqref{VOLUME} -- equivalently, the fact that the function $f$, interpreted as the dual of the free energy, is affine in the direction of $\bar{V}$ ($D^2f \, \bar{V} = 0$, Lemma \ref{Hessianandgradient}). Like in the investigations in \cite{dredrugagu17a}, \cite{bothedruet}, \cite{druetmixtureincompweak}, the idea is to invert the algebraic relations \eqref{CHEMPOT} for $\mu, \, p, \, \rho$ and to combine this procedure with appropriate linear projections.

\subsection{General ideas}\label{changevariables}

We choose a basis of $\mathbb{R}^N$: $\{\xi^1,\ldots, \xi^{N-2}, \, \xi^{N-1}, \, \xi^N\}$ with $\xi^N = 1^N$ and $\xi^{N-1} = \bar{V}$. We then choose $\eta^1,\ldots,\eta^N$ to be the dual basis, i. e. $\xi^i \cdot \eta^j = \delta^{i}_{j}$ for $i,j = 1,\ldots,N$. We define variables $q^1, \ldots, q^{N-2}$ and $\zeta$ via
\begin{align}\label{lesqs}
q_{\ell} & := \eta^{\ell} \cdot \mu := \sum_{i=1}^N \eta^{\ell}_i \, \mu_i \text{ for } \ell = 1,\ldots,N-2 \, ,\\
\label{leszetas} \zeta (=q_{N-1}) & := \eta^{N-1} \cdot \mu = \sum_{i=1}^N \eta^{N-1}_i \, \mu_i \, .
\end{align}
For $\rho \in \mathbb{R}^N_+$ such that $\sum_{i=1}^{N} \rho_i \, \bar{V}_i = 1$, we want to invert the relation $\mu_i = \bar{V}_i \, p + \partial_{\rho_i} k(\rho)$ for $i=1,\ldots,N$. We exploit the result of Lemma \ref{functionf2} saying that \eqref{CHEMPOT} implies $\rho_i = \partial_{\mu_i} f(\mu_{1},\ldots,\mu_N)$ for $i=1,\ldots,N$. The vector $\mu$ is then decomposed according to $$\mu = \sum_{\ell=1}^{N-2} q^{\ell} \, \xi^{\ell} + \zeta \, \bar{V} + \mu\cdot \eta^N \, 1^N $$ into its projection onto $\{1^N\}^{\perp}$, expressed by the variables $q$ and $\zeta$, and its projection on $\poule\{1^N\}$. 

Next, the last coordinate $\mu\cdot \eta^N$ is eliminated using the equation
\begin{align*}
\varrho = \sum_{i=1}^N \rho_i & = 1^{N} \cdot \nabla_{\mu} f(\mu_1,\ldots,\mu_N) = 1^N \cdot \nabla_{\mu}f(\sum_{\ell = 1}^{N-2} q_{\ell} \, \xi^{\ell} + \zeta \, \bar{V} + (\mu \cdot \eta^N) \, 1^N) \, .
\end{align*}
The gradient $\nabla_{\mu} f$ is invariant in the direction $\bar{V}$ (cf.\ Lemma \ref{Hessianandgradient}) and, therefore, the variable $\zeta$ decouples from the latter equation, that now reads
\begin{align*}
\varrho - 1^N \cdot \nabla_{\mu} f(\sum_{\ell = 1}^{N-2} q_{\ell} \, \xi^{\ell} + (\mu \cdot \eta^N) \, 1^N) = 0 \, .
\end{align*}
This representation is an algebraic equation $F(\mu \cdot \eta^N, \, q_1, \ldots, q_{N-2}, \, \varrho) = 0$. In view of Lemma \ref{Hessianandgradient}, note that $\partial_{\mu \cdot \eta^N} F(\mu \cdot \eta^N, \, q_1, \ldots, q_{N-2}, \, \varrho) = - D^{2}f(\mu) 1^N \cdot 1^N < 0$, due the fact that $1^N$ is not parallel to $\bar{V}$. Thus, the last component $\mu \cdot \eta^N$ is defined implicitly as a differentiable function of $\varrho$ and $q$. We call this function $\mathscr{M}$ and obtain the equivalent formulation
\begin{align}
\mu = & \sum_{\ell=1}^{N-2} q_{\ell} \, \xi^{\ell} + \zeta \, \bar{V} + \mathscr{M}(\varrho, \, q_1,\ldots,q_{N-2}) \, 1^N\, ,\nonumber\\
\rho = &  \nabla_{\mu} f( \sum_{\ell=1}^{N-2} q_{\ell} \, \xi^{\ell} + \mathscr{M}(\varrho, \, q_1,\ldots,q_{N-2}) \, 1^N) =: \mathscr{R}(\varrho, \, q) \, ,\label{rhomap}
\end{align}
where only the total mass density $\varrho$ and the \emph{relative chemical potentials} $q_1,\ldots,q_{N-2}$ and $\zeta$ occur as free variables. Note, moreover, that $\zeta$ and $\rho$ decouple. Similarly, we obtain a representation of the pressure as
\begin{align}
\begin{split}\label{pressuredef}
p = f(\mu) = & f( \sum_{\ell=1}^{N-2} q_{\ell} \, \xi^{\ell} + \zeta \, \bar{V} + \mathscr{M}(\varrho, \, q_1,\ldots,q_{N-2}) \, 1^N) \\
=  & f( \sum_{\ell=1}^{N-2} q_{\ell} \, \xi^{\ell} + \mathscr{M}(\varrho, \, q_1,\ldots,q_{N-2}) \, 1^N) + \zeta 
=: P(\varrho, \, q) + \zeta \, . 
\end{split}
\end{align}
All this is summarised in the following Lemma, the proof of which is direct in view of the Lemmas \ref{smoothness} and \ref{functionf2}.
\begin{lemma}\label{changecoor}
We adopt the assumptions of Theorem \ref{MAIN} for the function $k$. Let $I = ]\varrho_{\min}, \, \varrho_{\max}[$ with $\varrho_{\min} = \min_{i=1,\ldots,N} 1/\bar{V}_i$ and $\varrho_{\max} = \max_{i=1,\ldots,N} 1/\bar{V}_i$. Then there exist a function $\mathscr{M} \in C^2(I \times \mathbb{R}^{N-2})$ and a field $\mathscr{R} \in C^2(I \times \mathbb{R}^{N-2}; \, S_{\bar{V}})$ such that the equations $\rho = \nabla_{\mu} f(\mu)$ are valid if and only if there are $\varrho \in I$, $q \in \mathbb{R}^{N-2}$ and $\zeta \in \mathbb{R}$ such that
\begin{align*}
\sum_{i=1}^N \rho_i = \varrho, \, \rho = \mathscr{R}(\varrho, \, q), \qquad \mu = \sum_{j=1}^{N-2} q_j \, \xi^j + \zeta \, \bar{V} + \mathscr{M}(\varrho, \, q) \, 1^N =: \mu(\varrho, \, q, \, \zeta) \, .
\end{align*}
If, moreover, $\mu = \bar{V} \, p + \partial_{\rho} k(\rho)$ then $p = P(\varrho, \, q) + \zeta$ with $P \in C^2(I \times \mathbb{R}^{N-2})$ defined by \eqref{pressuredef}.
\end{lemma}
In order to deal with the right-hand side (external forcing), we define in the same spirit: 
\begin{align*}
\tilde{b}^{\ell}(x, \, t) := & \sum_{i=1}^N b^i(x, \, t) \, \eta^{\ell}_i \text{ for } \ell = 1,\ldots,N-2 \, ,\\
\hat{b}(x, \, t) := & \sum_{i=1}^N b^i(x, \, t) \, \eta^{N-1}_i \, , \quad \bar{b}(x, \, t) := \sum_{i=1}^N b^i(x, \, t) \, \eta^{N}_i \, . 
\end{align*}
This allows to express $$b^i(x, \, t) := \sum_{\ell=1}^{N-2} \tilde{b}^{\ell}(x, \, t) \, \xi^{\ell}_i + \hat{b}(x, \, t) \, \bar{V}_i + \bar{b}(x, \, t) \text{ for } i =1,\ldots,N \, .$$ For the reaction terms, we define $\tilde{r}_{\ell}(\varrho, \, q) := \sum_{i=1}^N \xi^{\ell}_i \, r_i(\mathscr{R}(\varrho, \, q))$ for $\ell = 1, \ldots, N-2$.

\subsection{Reformulation of the partial differential equations and of the main theorem}

The relation \eqref{CONSTRAINT} and the equivalence of Lemma \ref{changecoor} show that
\begin{align*}
& J^i = - \sum_{j=1}^N M_{i,j}(\rho_1,\ldots,\rho_N) \, (\nabla \mu_j - b^j)\\
& =  - \sum_{j=1}^N M_{i,j}(\rho_1,\ldots,\rho_N) \, \left[\sum_{\ell = 1}^{N-2}  \xi^{\ell}_j \, (\nabla q_{\ell} - \tilde{b}^{\ell}) + \, \bar{V}_j \, (\nabla \zeta - \hat{b}) + (\nabla \mathscr{M}(\varrho, \, q) - \bar{b})\right] \\
& = - \sum_{\ell = 1}^{N-2}  \sum_{j=1}^N M_{i,j}(\rho_1,\ldots,\rho_N) \, \xi^{\ell}_j \, (\nabla q_{\ell}- \tilde{b}^{\ell}) -  \sum_{j=1}^N M_{i,j}(\rho_1,\ldots,\rho_N) \, \bar{V}_j \, (\nabla \zeta - \hat{b}) \, .
 \end{align*}
If we introduce the rectangular projection matrix $\Pi_{j,\ell} = \xi^{\ell}_j$ for $\ell = 1,\ldots,N-2$ and $j = 1,\ldots,N$, then $J = - M \, \Pi (\nabla q - \tilde{b}) - M \, \bar{V} \, (\nabla \zeta - \hat{b})$. Thus, we consider equivalently
\begin{alignat*}{2}
\partial_t \rho + \divv( \rho \, v - M \, \Pi\, (\nabla q - \tilde{b}) - M\, \bar{V} \, (\nabla \zeta- \hat{b})) & = r\, , & & \\
\partial_t (\varrho \, v) + \divv( \varrho \, v\otimes v - \mathbb{S}(\nabla v)) + \nabla P(\varrho, \, q) + \nabla \zeta & = \rho\cdot b  \, .& &  
\end{alignat*}
In the latter system, we have $\rho = \mathscr{R}(\varrho, \, q)$ and $(\varrho, q_1,\ldots,q_{N-2}, \zeta,\, v_1,v_2,v_3)$ are the independent variables. Next, we define for $k = 1,\ldots, N-2$ the maps
\begin{align*}
R_{k}(\varrho, \, q) := & \sum_{j=1}^N \xi^{k}_j \, \rho_j = \Pi^{\sf T} \, \rho\\
= & \sum_{j=1}^N \xi^{k}_j \, f_{\mu_j}( \sum_{\ell=1}^{N-2} q_{\ell} \, \xi^{\ell} + \mathscr{M}(\varrho, \, q_1,\ldots,q_{N-2}) \, 1^N) \, .
\end{align*}
Multiplying the mass transfer equations with $\xi^{k}_i$, we obtain that
\begin{align*}
\partial_t R_k(\varrho, \, q) +\divv\Big(R_k(\varrho, \, q) \, v - [\Pi^{\sf T} \, M(\rho) \, \Pi]_{k,\ell}\,  (\nabla q_{\ell} - \tilde{b}^{\ell}) - [\Pi^{\sf T} \, M(\rho) \, \bar{V}]_k \, (\nabla \zeta - \hat{b})\Big) = \tilde{r}_k\, .
\end{align*}
It can be checked easily that the matrix $\Pi^{\sf T} \, M(\rho) \, \Pi \in \mathbb{R}^{(N-2)\times (N-2)}$ is symmetric and strictly positive definite on all states $\rho \in S_{\bar{V}}$.
The Jacobian
\begin{align*}
R_{q} = \Pi^{\sf T} \, D^{2}f \, \Pi - \frac{\Pi^{\sf T} \, D^2f 1^N \otimes \Pi^{\sf T} \, D^2f 1^N}{D^2f 1^N \cdot 1^N} \, ,
\end{align*}
of size $(N-2)\times (N-2)$ is also strictly positive definite. Indeed, vectors of the form $\Pi \, a$ in $\mathbb{R}^N$ with nonzero $a \in \mathbb{R}^{N-2}$ can by construction never belong to $\poule\{1^N, \, \bar{V}\}$. 
We next multiply the mass balance equations with $\bar{V}_i$. Making use of the constraint \eqref{VOLUME} yields
\begin{align*}
\divv(v -\, \bar{V} \cdot M(\rho)\, \Pi\, (\nabla q - \tilde{b}) - \bar{V} \cdot M(\rho)\, \bar{V} \, (\nabla \zeta - \hat{b}))  = \bar{V} \cdot r = 0 \, ,
\end{align*}
where we use the additional assumption that $r$ maps into $\{\bar{V}\}^{\perp}$. Using that $\rho = \mathscr{R}(\varrho, \, q)$, we define
\begin{align}\label{trafom}
\widetilde{M}(\varrho, \, q) := & \Pi^{\sf T} \, M(\mathscr{R}(\varrho, \, q)) \, \Pi \in \mathbb{R}^{(N-2)\times (N-2)}\, ,\\
\label{trafoa} A(\varrho, \,q) := & \Pi^{\sf T} \, M(\mathscr{R}(\varrho, \, q)) \, \bar{V} \in \mathbb{R}^{N-2}\, ,\\
\label{trafod} d(\varrho, \,q) := & \bar{V} \cdot M(\mathscr{R}(\varrho, \, q))\, \bar{V} \, .
\end{align}
Overall, we get for the variables $(\varrho, \, q_1, \ldots, q_{N-2}, \, \zeta, \, v)$ -- instead of \eqref{mass}, \eqref{momentum} -- the equations
\begin{alignat}{2}
 \label{massprime}\partial_t R(\varrho, \, q)  + \divv( R(\varrho, \, q) \, v - \widetilde{M}(\varrho, \, q) \, \nabla q -  A(\varrho, \, q)  \, \nabla \zeta)  
  = &  \nonumber\\
\tilde{r}(\varrho, \, q) -\divv(\widetilde{M}(\varrho, \, q) \, \tilde{b} +A(\varrho, \, q)  \,\hat{b}) \, , \qquad  & \\[0.1cm]
\label{massprime2} \divv(v   -\, A(\varrho, \, q) \cdot \nabla q  -  d(\varrho, \, q) \, \nabla \zeta)   = - \divv(A(\varrho, \, q) \cdot \tilde{b} + d(\varrho, \, q) \, \hat{b})  & \, , \\[0.1cm]
  \label{massprime3} \partial_t \varrho  + \divv(\varrho \, v)  =  0 & \, , \\[0.1cm]
\label{momentumprime}\partial_t (\varrho \, v)  + \divv( \varrho \, v\otimes v - \mathbb{S}(\nabla v)) + \nabla P(\varrho, \, q) +  \nabla \zeta 
   =  & \nonumber\\
   R(\varrho, \, q) \cdot \tilde{b}(x, \, t)  + \hat{b}(x, \, t) + \varrho \, \bar{b}(x, \, t)  \, .\qquad &
\end{alignat}
The problem $(P^{\prime})$ consisting of \eqref{massprime}, \eqref{massprime2}, \eqref{massprime3} and \eqref{momentumprime} for the variables $(\varrho, \, q, \, \zeta, \, v)$ might seem to exhibit more nonlinearities than the original problem for $\rho, \, p$ and $v$. However, it has the advantage that -- up to the restriction on the total mass density $\varrho_{\min} < \varrho < \varrho_{\max}$ -- it is completely free of constraints. Furthermore, the differential operator is linear in the variable $\zeta$, which occurs only under spatial differentiation.

Our first aim is now to show that, at least locally in time, the system  \eqref{massprime}, \eqref{massprime2}, \eqref{massprime3} and \eqref{momentumprime} for the variables $(\varrho, \, q_1, \ldots, q_{N-2}, \, \zeta, \, v)$ is well posed. We consider initial conditions
\begin{align}\label{initialall}
q(x, \, 0) & = q_0(x)\, ,  \varrho(x, \, 0)  = \varrho_0(x) \, , v(x, \, 0)  = v_0(x)  \quad \text{ for } x \in \Omega \, .
\end{align}
Due to the preliminary considerations in section \ref{changevariables}, prescribing these variables is completely equivalent to prescribing initial values for the mass densities $\rho_i$ and the velocity. It suffices to define $q^0_k = \eta^k \cdot \partial_{\rho}k(\rho^0)$ for $k =1,\ldots,N-2$.

For simplicity, we consider the linear homogeneous boundary conditions
\begin{alignat}{2}
\label{lateralv} v & = 0 & &\text{ on } S_T\, ,\\
\label{lateralq} \nu \cdot \nabla \zeta, \quad \nu \cdot \nabla q_{k} & = 0 & &\text{ on } S_T \text{ for } k = 1,\ldots,N-2 \, .
\end{alignat}
The conditions \eqref{lateralq} and \eqref{lateral0q} are equivalent, because we assume throughout  that the given forcing $b$ satisfies $\nu(x) \cdot \mathcal{P}_{\{1^N\}^{\perp}} \,b(x, \, t) = 0$ for $x \in \partial \Omega$ (see assumption \eqref{force} in the statement of Theorem \ref{MAIN}).

Under the assumptions of Theorem \ref{MAIN} for the function $k$, the coefficient functions $R$, $\widetilde{M}, \, A, \, d$ and $P$ are of class $C^2$ in the domain of definitions $I \times \mathbb{R}^{N-2}$ as shown in the Lemma \ref{changecoor}. We reformulate the Theorem \ref{MAIN} for the new variables. Since the thermodynamic pressure does not occur explicitly as a variable, we now switch to denoting $p >3$ the integrability exponent (denoted $s$ in the statement \ref{MAIN}).
\begin{theo}\label{MAIN2}
Assume that the coefficient functions $R$, $\widetilde{M}$, $A$, $d$ and $P$ are of class $C^2$, and $\tilde{r}$ is of class $C^1$ in the domain of definition $I \times \mathbb{R}^{N-2}$. Let $\Omega$ be a bounded domain with boundary $\partial \Omega$ of class $\mathcal{C}^2$. Suppose that, for some $p > 3$, the initial data are of class
\begin{align*}
q^0 \in W^{2-\frac{2}{p}}_p(\Omega; \, \mathbb{R}^{N-2}), \, \varrho_0 \in W^{1,p}(\Omega), \, v^0 \in W^{2-\frac{2}{p}}_p(\Omega; \, \mathbb{R}^{3}) \, ,
\end{align*}
satisfying $\varrho_{\min} < \varrho^0(x) < \varrho_{\max}$ in $\overline{\Omega}$ and the compatibility conditions $\nu(x) \cdot \nabla q^0(x) = 0$ and $v^0(x) = 0$ on $\partial \Omega$. Assume that $\tilde{b} \in W^{1,0}_p(Q_T; \, \mathbb{R}^{(N-2)\times 3}) $, $\hat{b} \in W^{1,0}_p(Q_T; \, \mathbb{R}^{3})$ and $\bar{b} \in L^p(Q_T; \,\mathbb{R}^3)$. Then there is $0< T^* \leq T$, depending only on these data, such that the problem \eqref{massprime}, \eqref{massprime2}, \eqref{massprime3} and \eqref{momentumprime} with boundary conditions \eqref{initialall}, \eqref{lateralv} and \eqref{lateralq} is uniquely solvable in the class
\begin{align*}
(q, \, \zeta, \, \varrho,\,v) \in W^{2,1}_p(Q_{T^*}; \, \mathbb{R}^{N-2}) \times W^{2,0}_p(Q_{T^*})  \times W^{1,1}_{p,\infty}(Q_{T^*}; \, S_{\bar{V}}) \times W^{2,1}_p(Q_{T^*}; \, \mathbb{R}^{3})  \, .
\end{align*}
The solution can be uniquely extended within this class to a larger time interval whenever at least one of the following holds:
\begin{enumerate}[(1)]
 \item $p>5$ and the state space norm stays finite as $t \nearrow T^*$;
 \item The two following conditions are valid as $t \nearrow T^*$
 \begin{itemize}
\item  $\varrho_{\min} < \varrho(x, \, t) < \varrho_{\max}$ for all $x \in \overline{\Omega}$;
\item $\|q\|_{C^{\alpha,\frac{\alpha}{2}}(Q_{t})} + \|\nabla q\|_{L^{\infty,p}(Q_{t})} + \|v\|_{L^{z \, p, \, p}(Q_{t})} + \int_{0}^{t} [\nabla v(\tau)]_{C^{\alpha}(\Omega)} \, d\tau < + \infty$, with $\alpha > 0$ and $z = z(p)$ defined by Theorem \ref{MAIN},  
 \end{itemize}
\end{enumerate}
\end{theo}

\section{Functional analytic approach}\label{funcsetting}

For functions $q_1,\ldots, q_{N-2}$, $\zeta$, $\varrho$ and $v_1, \, v_2, \, v_3$ defined in $\overline{\Omega} \times [0,T]$, we introduce 
\begin{align*}
\mathscr{A}(q, \, \zeta, \, \varrho, \, v) = & (\mathscr{A}^1(q, \, \zeta, \, \varrho, \, v) , \, \mathscr{A}^2(q, \, \zeta, \,  \varrho, \, v), \, \mathscr{A}^3(\varrho, \, v) , \, \mathscr{A}^4(q, \, \zeta, \,  \varrho, \, v) ) \, ,\\
\mathscr{A}^1(q, \, \zeta, \, \varrho, \, v)  := & \partial_t R(\varrho, \, q) + \divv( R(\varrho, \, q) \, v)
\\ & -\divv( \widetilde{M}(\varrho, \, q) \, (\nabla q-\tilde{b}) + A(\varrho, \, q) \, (\nabla \zeta-\hat{b})) - \tilde{r}(\varrho, \, q) \, ,\\
\mathscr{A}^2(q, \, \zeta, \, \varrho, \, v)  := & \divv(v - d(\varrho, \, q) \, (\nabla \zeta-\hat{b}) - A(\varrho, \, q) \cdot (\nabla q -\tilde{b}))\, , \\
\mathscr{A}^3(\varrho, \, v)  := & \partial_t \varrho + \divv(\varrho \, v)\, ,\\
\mathscr{A}^4(q, \, \zeta, \, \varrho, \, v)  := & \varrho \, (\partial_t v + (v\cdot\nabla) v) - \divv \mathbb{S}(\nabla v) + \nabla P(\varrho, \, q) + \nabla \zeta\\
& - R(\varrho, \, q) \cdot \tilde{b} - \hat{b} - \varrho \, \bar{b} \, .
\end{align*}
Recall that $\tilde{b}$, $\hat{b}$ and $\bar{b}$ are given coefficients.

To get rid of the highest-order coupling in the time derivative of $\varrho$, we shall employ the same approach as in \cite{bothedruet}, which is sketched below. Consider a solution $u = (q, \, \zeta,\, \varrho, \, v)$ to $\mathscr{A}(u) = 0$. Computing time derivatives in the equation $\mathscr{A}^1(u) = 0$, we obtain that
\begin{align*}
& R_{\varrho} \, (\partial_t \varrho + v \cdot \nabla \varrho) + \sum_{j=1}^{N-1} R_{q_j} \, (\partial_t q_j + v \cdot \nabla q_j) + R \, \divv  v - \divv (\widetilde{M} \, \nabla q) + A\, \nabla \zeta)  \\
& = - \divv(\widetilde{M} \, \tilde{b} + A \, \hat{b}) + \tilde{r}  \, .
\end{align*}
Here the nonlinear functions $R, \, R_{\varrho}, \, R_{q}, \, A$ and $\widetilde{M}$, $\tilde{r}$ etc. are evaluated at $(\varrho, \, q)$. Under the side-condition $\mathscr{A}^3(\varrho, \, v) = 0$, the equation $\mathscr{A}^1(u) = 0$ is equivalent to 
\begin{align}\label{a1equiv}
&  R_{q}(\varrho, \, q) \, \partial_t q - \divv (\widetilde{M}(\varrho, \, q) \, \nabla q + A(\varrho, \, q) \, \nabla \zeta)  =  (R_{\varrho}(\varrho, \, q) \, \varrho\nonumber\\
 & \qquad - R(\varrho, \, q)) \, \divv v - R_q(\varrho, \, q) \, v \cdot \nabla q  - \divv(\widetilde{M}(\varrho, \, q) \, \tilde{b} + A(\varrho, \, q) \, \hat{b})+ \tilde{r}(\varrho, \, q)   \, .
\end{align}
We introduce $\widetilde{\mathscr{A}}(q, \, \zeta, \, \varrho, \, v) := (\widetilde{\mathscr{A}}^1(q, \, \zeta, \, \varrho, \, v) , \, \mathscr{A}^2(q, \, \zeta,\, \varrho, \, v) , \, \mathscr{A}^3(\varrho, \, v), \, \mathscr{A}^4(q, \, \zeta,\, \varrho, \, v)  )$, the first component being the differential operator defined by \eqref{a1equiv}. Clearly, $\mathscr{A}(u) = 0$ if and only if $\widetilde{\mathscr{A}}(u) = 0$.

The functional setting was introduced in Section \ref{MATHEMATICS}. Similar spaces were used in \cite{bothedruet} to study the compressible system and, in order to save room, we shall refer to this paper for the trace and embedding theorems needed in the present analysis. For $p > 3$ and $\alpha := 1/2+3/(2p)$, we recall the interpolation inequality (see \cite{nirenberginterpo}, Theorem 1)
\begin{align}
\label{gagliardo}\|\nabla f\|_{L^{\infty}(\Omega)} \leq & C_1 \, \|D^2f\|_{L^p(\Omega)}^{\alpha} \, \|f\|_{L^p(\Omega)}^{1-\alpha} + C_2 \, \|f\|_{L^p(\Omega)} \, , 
\end{align}
valid for any function $f$ in $W^{2,p}(\Omega)$, with certain constants $C_1, \, C_2$ depending only on $\Omega$. We consider the operator $(q, \, \zeta, \, \varrho,\, v) \mapsto \mathscr{A}(q, \, \zeta, \, \varrho,\, v)$ acting on 
\begin{align}\label{STATESPACE}
\mathcal{X}_T := W^{2,1}_p(Q_T; \, \mathbb{R}^{N-2}) \times W^{2,0}_p(Q_T) \times W^{1,1}_{p,\infty}(Q_T) \times W^{2,1}_p(Q_T; \, \mathbb{R}^3) \, .
\end{align}
The natural trace space at time zero is denoted $\text{Tr}_{\Omega\times\{0\}} \, \mathcal{X}_T$. The functional setting does not allow to introduce traces for the variable $\zeta$. Therefore, $ u(0) \in \text{Tr}_{\Omega\times\{0\}} \, \mathcal{X}_T$ means that $(q(0), \, \varrho(0) , \, v(0)) \in W^{2-2/p}_p(\Omega; \, \mathbb{R}^{N-2}) \times W^{1,p}(\Omega) \times W^{2-2/p}_p(\Omega; \, \mathbb{R}^{3})$. We denote by $\phantom{}_{0}\mathcal{X}_T$ the space of functions fulfilling zero initial conditions. This only makes sense, of course, for the variables having traces at $\Omega \times \{0\}$. Thus
\begin{align}\label{spaceX0T}
\phantom{}_{0}\mathcal{X}_T := \{\bar{u} = (r, \, \chi, \, \sigma, \, w) \in \mathcal{X}_T \, : \, r(0) = 0, \, \sigma(0) = 0, \, w(0) = 0\} \, \, . 
\end{align}
Since the coefficients of $\mathscr{A}$ are defined only if $\varrho$ has range in $I$, the domain of the operator is contained in the subset
\begin{align}\label{SPACECONE}
 \mathcal{X}_{T,I} := W^{2,1}_p(Q_T; \, \mathbb{R}^{N-2})\times W^{2,0}_p(Q_T) \times W^{1,1}_{p,\infty}(Q_T; \, I) \times W^{2,1}_p(Q_T; \, \mathbb{R}^3) \, .
\end{align}
We shall moreover make use of a reduced state space containing only the parabolic components $(q, \, v)$, namely
\begin{align}\label{STATESPACE2}
\mathcal{Y}_T := W^{2,1}_p(Q_T; \, \mathbb{R}^{N-2}) \times W^{2,1}_p(Q_T; \, \mathbb{R}^3) \, .
\end{align}
The operator $\mathscr{A}$ is the composition of differentiation, multiplication and Nemicki operators. Therefore, the properties of the coefficients $R, \, \tilde{M}$ etc. allow to show that $\mathscr{A}$ is continuous and bounded from $\mathcal{X}_{T,I}$ into
\begin{align}\label{IMAGESPACE}
\mathcal{Z}_T = L^p(Q_T; \, \mathbb{R}^{N-2}) \times L^p(Q_T) \times L^{p,\infty}(Q_T) \times L^p(Q_T; \, \mathbb{R}^{3}) \, .
\end{align}
Since the coefficients $R$, $\widetilde{M}, \, A, \, d$ and $P$ are twice continuously differentiable in their domain of definition $I \times \mathbb{R}^{N-2}$, the operator $\mathscr{A}$ is even continuously differentiable at every point of $\mathcal{X}_{T,I}$. We spare the proof of these rather obvious statements. 

\section{Linearisation and reformulation as a fixed-point equation}\label{twomaps}

We shall present two different manners to linearise the equation $\mathscr{A}(u) = 0$ for $u \in \mathcal{X}_T$ with initial condition $u(0) = u_0$ in $\text{Tr}_{\Omega\times\{0\}} \, \mathcal{X}_T$. They correspond to the two main Theorems \ref{MAIN}, \ref{MAIN3} respectively. In both cases, we start considering the problem to find $u =(q, \, \zeta, \, \varrho ,\, v) \in \mathcal{X}_{T,I}$ such that $\widetilde{\mathscr{A}}(u) = 0$ and $u(0) = u_0$, which after permuting rows, possesses the following structure
\begin{align}\label{linearT1gen}
\partial_t \varrho + \divv (\varrho \, v) = & 0\, ,\\
\label{linearT2gen} R_{q}(\varrho, \, q) \, \partial_t q - \divv (\widetilde{M}(\varrho, \, q) \, \nabla q + A(\varrho, \, q) \, \nabla \zeta)  = & g(x, \, t, \, q, \, \varrho,\, v, \, \nabla q, \, \nabla \varrho, \, \nabla v)\,  ,\\
\label{linearT3gen} -\divv( d(\varrho, \, q) \, \nabla \zeta + A(\varrho, \, q) \, \nabla q-v) = & - \divv h(x, \,t,  \, \varrho, \, q) \, , \\
\label{linearT4gen} \varrho \, \partial_t v - \divv \mathbb{S}(\nabla v) +\nabla \zeta = &  f(x, \, t, \, q, \, \varrho,\, v, \, \nabla q, \, \nabla \varrho, \, \nabla v) \, .
\end{align}
The functions $g$, $h$ and $f$ stand for the following expressions:
\begin{align}\label{gri}
g := &  (R_{\varrho}(\varrho, \, q) \, \varrho - R(\varrho, \, q)) \, \divv v - R_q(\varrho, \, q) \, v \cdot \nabla q \nonumber \\
& - \divv(\widetilde{M}(\varrho, \, q)  \, \tilde{b} + A(\varrho, \, q)  \, \hat{b})+ \tilde{r}(\varrho, \, q) \, , \\
 \label{hri} h := &  d(\varrho, \, q) \, \hat{b}+  A(\varrho, \, q) \, \tilde{b} \, ,\\
 \label{fri} f := & - P_{\varrho}(\varrho,\, q) \, \nabla \varrho -  P_{q}(\varrho,\, q) \, \nabla q  - \varrho \, (v\cdot \nabla)v + R(\varrho, \,q) \cdot \tilde{b} + \hat{b}+ \varrho \, \bar{b} \, .
\end{align}
These expressions are independent on the component $\zeta$. We can regard $g$, $h$ and $f$ as functions of $x, \,t$ and of the vectors $u$ and $D_x u$ and write $g(x, \,t, \, u, \, D_xu)$. 

\subsection{The first fixed-point equation}

For $(q^*, \, v^*)$ given in $W^{2,1}_p(Q_T; \, \mathbb{R}^{N-2}) \times W^{2,1}_p(Q_T; \, \mathbb{R}^{3}) $ and for unknowns $u = (q, \, \zeta, \, \varrho, \, v)$, we consider the following system of equations
\begin{align}\label{linearT1}
\partial_t \varrho + \divv (\varrho \, v^*) = & 0 \, ,\\
\label{linearT2} R_{q}(\varrho, \,q^*) \, \partial_t q - \divv (\widetilde{M}(\varrho, \, q^*) \, \nabla q + A(\varrho, \, q^*) \, \nabla \zeta)  = & g(x, \, t, \, q^*, \, \varrho,\, v^*, \, \nabla q^*, \, \nabla \varrho,\, \nabla v^*)\, ,\\
\label{linearT3} - \divv(d(\varrho, \, q^*) \, \nabla \zeta + A(\varrho, \, q^*) \, \nabla q) = & - \divv (v^* + h(x, \, t, \, q^*, \, \varrho))\, ,\\
\label{linearT4} \varrho \, \partial_t v - \divv \mathbb{S}(\nabla v) + \nabla \zeta  = & f(x, \, t, \, q^*, \, \varrho,\, v^*, \, \nabla q^*, \, \nabla \varrho,\, \nabla v^*) \, ,
\end{align}
together with the initial conditions \eqref{initialall}, \eqref{initialall}, \eqref{initialall} and the homogeneous boundary conditions \eqref{lateralv}, \eqref{lateralq}. Note that the continuity equation can be solved independently for $\varrho$. Once $\varrho$ is given, we solve the linear parabolic--elliptic system \eqref{linearT2}, \eqref{linearT3} for $q$ and $\zeta$. Here we must be careful, since the coefficients of this system are only defined as long as $\varrho(x, \, t)$ takes values in $I$. Thus, the solution $(q, \, \zeta)$ might exist only on a shorter time interval. We can solve the problem \eqref{linearT4}, which is linear in $v$, under the same restriction. 

We will show that the solution map $(q^*, \, v^*) \mapsto (q, \, v)$, denoted $\mathcal{T}$, is well defined from $\mathcal{Y}_T$ into itself for $T$ fixed and suitably small. The solutions are unique in the class $\mathcal{Y}_T$. Clearly, a fixed point of $\mathcal{T}$ is a solution to $\widetilde{\mathscr{A}}(q, \, \zeta, \, \varrho, \, v) = 0$.

\subsection{The second fixed-point equation}

Here we construct the fixed-point map comparing the solutions to a given reference vector $(\hat{q}^0, \, \hat{v}^0) \in \mathcal{Y}_T$ that extends the initial data. We assume that $\hat{q}^0$ and $\hat{v}^0$ satisfy the initial compatibility conditions. 
In order to find an extension for $\varrho_0 \in W^{1,p}(\Omega)$, we solve the problem
\begin{align}\label{Extendrho0}
 \partial_t \hat{\varrho}_0 + \divv(\hat{\varrho}_0 \, \hat{v}^0) =0, \quad \hat{\varrho}_0(0) = \varrho_0 \, .
\end{align}
For this problem, Theorem 2 of \cite{solocompress} establishes unique solvability in $W^{1,1}_{p,\infty}(Q_T)$ and, in particular, the strict positivity $\hat{\varrho}_0 \geq c_0(\Omega, \, \|\hat{v}^0\|_{W^{2,1}_p(Q_T; \, \mathbb{R}^3)}) \, \inf_{x \in \Omega} \varrho_0(x)$.

We find the extension $\hat{\zeta}^0$ by solving, for all values of $t$ such that the coefficients $\tilde{b}(t)$ and $\hat{b}(t)$ are defined, the elliptic problem
\begin{align}\label{Extendzeta0}
- \divv (d(\hat{\varrho}^0, \, \hat{q}^0) \, \nabla \hat{\zeta}^0) = \divv(- \hat{v}^0 - d(\hat{\varrho}^0, \, \hat{q}^0) \, \hat{b}(t) + A(\hat{\varrho}^0, \, \hat{q}^0) \, \nabla (\hat{q}^0 - \tilde{b}(t))) \,,
\end{align}
with homogeneous Neumann boundary conditions and zero mean--value side--condition.

Consider a solution $u = (q, \, \zeta, \, \varrho, \, v) \in \mathcal{X}_T$ to $\widetilde{\mathscr{A}}(u) = 0$. We introduce the differences $r := q - \hat{q}^0$, $\chi = \zeta-\hat{\zeta}^0$, $w := v - \hat{v}^0$ and $\sigma := \varrho- \hat{\varrho}^0$, and their vector $\bar{u} := (r, \, \chi, \, \sigma, \, w)$. Clearly, $\bar{u}$ belongs to the space $\phantom{}_{0}\mathcal{X}_T$ of homogeneous initial conditions. Recall that this does not imply a trace condition for $\chi$, cp.\  \eqref{spaceX0T}. The equations $\widetilde{\mathscr{A}}(u) = 0$ mean, equivalently, that $\widetilde{\mathscr{A}}(\hat{u}^0 +\bar{u}) = 0$. The vector $\bar{u} = (r, \, \chi, \, \sigma, \, w)$ satisfies
\begin{align}\label{difference1}
  R_q \, \partial_t r - \divv (\widetilde{M} \, \nabla r + A \, \nabla \chi) =& g^1 :=  g -  R_q \, \partial_t \hat{q}^0 +\divv( \widetilde{M} \, \nabla  \hat{q}^0  + A \, \nabla \hat{\zeta}^0)\, ,\\
\label{difference2}  - \divv (d\, \nabla \chi + A \, \nabla r-w) = & -\divv h^1 := - \divv( h + \hat{v}^0 - d \, \nabla \hat{\zeta}^0 - A \, \nabla \hat{q}^0) \, ,\\ 
\label{difference3}
  \partial_t \sigma + \divv(\sigma \, v) =& - \divv(\hat{\varrho}_0 \, w) \, ,\\
\label{difference4}
  \varrho \, \partial_t w - \divv \mathbb{S}(\nabla w) +\nabla \chi = &  f^1 =:  f - \varrho \partial_t \hat{v}^0 + \divv \mathbb{S}(\nabla \hat{v}^0) - \nabla \hat{\zeta}^0 \, . 
\end{align}
Herein, all non-linear coefficients $R, \, R_q$, etc.\ are evaluated at $(\varrho, \, q)$, while $g$, $h$ and $f$ correspond to \eqref{gri}, \eqref{hri} and \eqref{fri}.

We next want to construct a fixed-point map to solve \eqref{difference1}, \eqref{difference2}, \eqref{difference3}, \eqref{difference4} by linearising $g^1$, $h^1$ and $f^1$ defined in \eqref{difference1}, \eqref{difference2} and \eqref{difference4}. First, we expand as follows:
\begin{align}
 g =  g(x, \, t,\, u^*, \, D_xu^*) + \int_{0}^1 & \{ (g_{q})^{\theta} \, (q-q^*) + (g_{\varrho})^{\theta} \, (\varrho-\varrho^*) + (g_v)^{\theta} \, (v-v^*) \nonumber \\
 & + (g_{q_x})^{\theta} \cdot (q_x-q^*_x) + (g_{\varrho_x})^{\theta} \, (\varrho_x- \varrho^*_x) + (g_{v_x})^{\theta} \cdot (v_x-v^*_x)\} \, d\theta \, .
 \end{align}
 Here, $( \cdot )^{\theta}$ applied to a function of $x, \, t$, $u$ and $D^1_x u$ stands for the evaluation at $(x, \, t,\, (1-\theta) \, u^* + \theta \, u, \, (1-\theta) \, D_xu^* + \theta \, D_xu)$. In short, in order to avoid the integral and the parameter $\theta$, we write
 \begin{align}\label{glinear1}
  g = &g(x, \, t,\, u^*, \, D_xu^*) + g_{q}(u, \, u^*) \, (q-q^*) + g_{\varrho}(u, \, u^*) \, (\varrho-\varrho^*)+ g_v(u, \, u^*)  \, (v-v^*)  \nonumber\\
  & + g_{q_x}(u, \, u^*) \cdot (q_x-q^*_x) + g_{\varrho_x}(u, \, u^*)  \, (\varrho_x- \varrho^*_x) + g_{v_x}(u, \, u^*) \cdot (v_x-v^*_x) \nonumber\\
 =: & g(x, \, t,\, u^*, \, D_xu^*) + g^{\prime}(u, \, u^*) \, (u - u^*) \, . 
\end{align}
Obviously, the latter expressions make sense only if $u, \, u^*$ both belong to $\mathcal{X}_{T,I}$, in which case the entire convex hull $\{\theta \, u + (1-\theta) \, u^* \, : \, \theta \in [0, \,1]\}$ is in $\mathcal{X}_{T,I}$. Following the same scheme as for \eqref{glinear1}, we write in short
\begin{align}\label{glinear2}
  g^1 = & g^1(x, \, t,\, \hat{q}^0, \, \hat{\varrho}^0, \, \hat{v}^0, \, \hat{q}_x^0, \, \hat{\varrho}_x^0, \, \hat{v}_x^0) + g^1_{q}(u, \, \hat{u}^0) \, r + g^1_{\varrho}(u, \, \hat{u}^0) \, \sigma + g^1_v(u, \, \hat{u}^0)\, w \nonumber\\
 &  + g^1_{q_x}(u, \, \hat{u}^0) \, r_x + g^1_{\varrho_x}(u, \, \hat{u}^0) \, \sigma_x + g^1_{v_x}(u, \, \hat{u}^0) \, w_x \, \nonumber\\
  =: &  \hat{g}^0 + (g^1)^{\prime}(u, \, \hat{u}^0) \, \bar{u} \, .
\end{align}
Similar expressions are obtained for $h^1$ and $f^1$. In the case of $h^1$, note however that $\divv \hat{h}^0 = \divv (h^1(x, \, t,\, \hat{q}^0, \, \hat{\varrho}^0, \, \hat{v}^0, \, \hat{q}_x^0, \, \hat{\varrho}^0_x, \, \hat{v}_x^0) = 0$ due to the construction \eqref{Extendzeta0} of $\hat{\zeta}^0$.

Now we construct the fixed-point map to solve \eqref{difference1}, \eqref{difference2}, \eqref{difference3} and \eqref{difference4}. For a given vector $(r^*, \, w^*) \in \phantom{}_{0}{\mathcal{Y}}_T$, we define $q^* := \hat{q}^0 + r^*$ and $v^* := \hat{v}^0 + w^*$. Then we define $\varrho^*$ to be the unique solution to 
\begin{align}\label{equadiff0}
\partial_t \varrho^* + \divv(\varrho^* \, v^*) = 0, \quad \varrho^*(x, \, 0) = \varrho^0(x) \, . 
\end{align}
We thus write $\varrho^* := \mathscr{C}(v^*)$ where $\mathscr{C}$ is the solution operator to the continuity equation with initial data $\varrho_0$. 
We employ the abbreviation
\begin{align}\label{ustar}
u^* := & (q^*, \, 1, \, \varrho^*, \, v^*) = (q^*,\, 1, \, \mathscr{C}(v^*), \, v^*) \in \mathcal{X}_T \, .
 \end{align}
For $\bar{u} := (r, \, \chi, \, \sigma, \, w)$, we next consider the linear problem
\begin{align} 
\label{equadiff1}  R_q^* \, \partial_t r - \divv \big(\widetilde{M}^* \, \nabla r + A^* \, \nabla \chi\big) = &  \hat{g}^0 + (g^{1})^{\prime}(u^*, \, \hat{u}^0) \, \bar{u} \, , \\
\label{equadiff2}  - \divv \big(d^* \, \nabla \chi + A^* \, \nabla r-w\big) = & -\divv ( (h^{1})^{\prime}(u^*, \, \hat{u}^0) \, \bar{u})\, , \\ 
 \label{equadiff3}  \partial_t \sigma + \divv(\sigma \, v^*) = &  - \divv(\hat{\varrho}_0 \, w) \, , \\
\label{equadiff4} \mathscr{C}(v^*) \, \partial_t w - \divv \mathbb{S}(\nabla w)  + \nabla \chi = & \hat{f}^0 + (f^1)^{\prime}(u^*, \, \hat{u}^0) \, \bar{u} \, ,  
\end{align}
with the boundary conditions $\nu \cdot \nabla r = 0 = \nu \cdot \nabla \chi$ on $S_T$ and $w = 0$ on $S_T$, and with zero initial conditions for $r, \, \sigma$ and $w$. The superscript $*$ on a coefficient means evaluation at $(\mathscr{C}(v^*), \, q^* )$.

We will show that the solution map $\mathcal{T}^1:\, (r^*, \, w^*) \mapsto (r, \, w)$ is well defined from $\phantom{}_{0}{\mathcal{Y}}_T$ into itself for $T > 0$ arbitrary, provided that the distance of the initial data to an equilibrium solution is sufficiently small. As to the latter restriction, note that the expressions $(g^{1})^{\prime}(u^*, \, \hat{u}^0)$ make sense only if the density components in both $u^*$ and $\hat{u}^0$ map into the interior of the critical interval, which cannot be expected globally for the solutions to \eqref{Extendrho0} and \eqref{equadiff0}. If $\bar{u} = (r, \, w)$ is a fixed point of $\mathcal{T}^1$, then we can show that $u := \hat{u}^0 + \bar{u}$ is a solution to $\widetilde{\mathscr{A}}(u) = 0$. This is verified exactly as in \cite{bothedruet}, Remark 6.1. 

\subsection{The self-mapping property}

Assume that the map $\mathcal{T}: \, (q^*, \, v^*) \mapsto (q, \, v)$ via the solution to \eqref{linearT1}, \eqref{linearT2}, \eqref{linearT3}, \eqref{linearT4} is well defined in $\mathcal{Y}_{T}$, with image in $\mathcal{Y}_{\tilde{T}}$ for some $\tilde{T} = \tilde{T}(q^*, \, v^*) >0$. Then, we want to show that $\mathcal{T}$ maps some closed bounded set of $\mathcal{Y}_{T_0}$ into itself for a fixed $T_0 > 0$. 
Here, a major change occurs in comparison to the compressible case, since we do not expect that the linearised map $\mathcal{T}$ produces a solution defined globally up to $T$. This is due to the constraint $\varrho \in ]\varrho_{\min}, \, \varrho_{\max}[$ which can by nature be enforced only locally for solutions to the continuity equation \eqref{linearT1}.

We shall rely on continuous estimates expressing the controlled growth of the solution in time. We will show that there is a parameter $a_0$ depending on the distance of the initial density to the singular values $\{\varrho_{\min}, \, \varrho_{\max}\}$ such that, whenever $t  >0$ satisfies $t^{1 - \frac{1}{p}} \, \|(q^*, \, v^*)\|_{\mathcal{Y}_t} < a_0$, the pair $(q, \, v) = \mathcal{T}(q^*, \, v^* )$ is well defined in $\mathcal{Y}_t$ and satisfies the estimate
\begin{align}\label{continuity1}
 \|(q, \, v)\|_{W^{2,1}_p(Q_t; \, \mathbb{R}^{N-2}) \times W^{2,1}_p(Q_t; \, \mathbb{R}^{3})} \leq \Psi(t, \, R_0, \, \|(q^*, \, v^*)\|_{W^{2,1}_p(Q_t; \, \mathbb{R}^{N-2}) \times W^{2,1}_p(Q_t; \, \mathbb{R}^{3})}) \, .
\end{align}
Here $R_0$ stands for the magnitude of the initial data $q^0$, $\varrho_0$ and $v^0$, and of the external forces $b$ in their respective norms. The function $\Psi$ is continuous, increasing in all arguments, and finite for $t^{1 - \frac{1}{p}} \, \|(q^*, \, v^*)\|_{\mathcal{Y}_t} < a_0$. Hence we obtain a self mapping property with the help of the following Lemma.
\begin{lemma}\label{estimfinale}
Suppose that $R_0$ is fixed. Suppose that there is $a_0 > 0$ such that the inequality \eqref{continuity1} is valid with a continuous function $\Psi = \Psi(t, \, R_0, \, \eta)$ satisfying the properties:
\begin{itemize}
 \item $\Psi(\cdot, \, R_0, \, \cdot)$ is finite for all $t \geq 0$ and $\eta \geq 0$ satisfying $t^{1 - \frac{1}{p}} \, \eta < a_0$;
\item $t \mapsto \Psi(t, \, R_0, \, \eta)$ is nondecreasing for all $0 \leq \eta$, and $\eta \mapsto \Psi(t, \, R_0, \, \eta)$ is nondecreasing for all $t$ as long as $t^{1 - \frac{1}{p}} \, \eta < a_0$;
\item The value of $\Psi(0, \, R_0, \, \eta) = \Psi^0(R_0) > 0$ is independent on $\eta$.
\end{itemize}
Then there is $t_0 = t_0(R_0) > 0$ such that the map $\mathcal{T} \, (q^*, \, v^*) := (q, \, v)$ maps a ball of $\mathcal{Y}_{t_0}$ into itself. 
\end{lemma}
\begin{proof}
In $\{(t, \, \eta) \in [\mathbb{R}_+]^2 \, : \, t^{1 - 1/p} \, \eta < a_0\}$,
the function $(t, \, \eta) \mapsto \Psi(t, \, R_0, \, \eta)$ is continuous and finite. Then, there is a first $t_0 > 0$ depending only on $R_0$ such that $$\{\eta > 0 \, : \, \Psi(t_0, \, R_0, \, \eta) \leq \eta \text{ and } \eta < a_0 \, t_0^{\frac{1}{p}-1}\} \neq \emptyset \, .$$
Otherwise, for all $t > 0$ and $\eta < a_0 \, t^{1/p - 1}$, we would have that $ \Psi(t, \, R_0, \, \eta) > \eta$. Thus, $\Psi(0, \, R_0, \, \eta) = \lim_{t\rightarrow 0} \Psi(t, \, R_0, \, \eta) \geq \eta$ for all $\eta >0$. Since $\Psi(0, \, R_0, \, \eta) = \Psi^0(R_0)$ is strictly positive, every choice of $\eta > \Psi^0(R_0)$ then yields a contradiction.
 
We can further show that $$0< \eta_0 := \inf  \{\eta > 0 \, : \, \Psi(t_0, \, R_0, \, \eta) \leq \eta \text{ and } \eta < a_0 \, t_0^{\frac{1}{p}-1}\} \, .$$
Otherwise, there are positive $\{\eta_k\}_{k\in\mathbb{N}}$, $\eta_k \searrow 0$, such that $\Psi(t_0, \, R_0, \, \eta_k) \leq \eta_k$ for all $k$. 
Then $0 \geq \lim_{k\rightarrow \infty} \Psi(t_0, \, R_0, \, \eta_k) = \Psi(t_0, \, R_0, \, 0)$. Since $\Psi(t_0, \, R_0, \, 0) \geq \Psi(0, \, R_0, \, 0) = \Psi^0(R_0) >0$, this is again a contradiction.

Consider $M := \{(q^*, \, v^*) \in \mathcal{Y}_{t_0} \, : \, \|(q^*, \, v^*)\|_{\mathcal{Y}_{t_0}} \leq \eta_0\}$. Since $\eta_0 < a_0\, t_0^{1/p-1}$, it follows that $t_0^{1-1/p} \, \|(q^*, \, v^*)\|_{\mathcal{Y}_{t_0}} < a_0$. The inequality \eqref{continuity1} is valid by assumption and it yields $\|(q, \, v)\|_{\mathcal{Y}_{t_0}} \leq \Psi(t_0, \, R_0, \, \eta_0) \leq \eta_0$, hence $(q, \, v) \in M$.
\end{proof}
In the case of the map $\mathcal{T}^1:\,(r^*, \,  w^*) \mapsto (r, \,  w)$ defined via solution to \eqref{equadiff0}, \eqref{equadiff1},\eqref{equadiff2}, \eqref{equadiff3}, \eqref{equadiff4}, we look for a fixed-point in the space $\phantom{}_{0}{\mathcal{Y}}_T$. The solution can only be defined globally on $[0, \, T]$ if the solution to \eqref{equadiff0} remains inside of $]\varrho_{\min}, \, \varrho_{\max}[$ on the entire time-interval. We will show that this can be ensured if the starting perturbation $w^*$ satisfies an inequality of type
\begin{align*}
\phi_0(T, \, \|w^*\|_{W^{2,1}_p(Q_T)}) \, \|w^*\|_{W^{2,1}_p(Q_T)} \leq a_0 \, , 
\end{align*}
in which $a_0 > 0$ is a fixed number depending on the distance of the initial data to the critical values $\{\varrho_{\min}, \, \varrho_{\max}\}$, and $\phi_0$ is a continuous function on $\overline{\mathbb{R}}_+^2$, which increases in both arguments.
We then prove a continuity estimate of the type
\begin{align}\label{continuity2}
\|(r, \,  w)\|_{\mathcal{Y}_T} \leq \Psi(T, \, R_0, \, \|(r^*, \,  w^*)\|_{\mathcal{Y}_T}) \, R_1 \, .
\end{align}
Here $R_0$ stands for the magnitude of initial data ($q^0$, $\varrho_0$ and $v^0$) and external forces $b$. The parameter $R_1$ expresses the distance of the initial data to a stationary/equilibrium solution (def. in \eqref{massstat}, \eqref{momentumsstat}). Defining $\eta_0$ to be the smallest positive solution to the equation $\phi_0(T, \, \eta_0) \, \eta_0 = a_0$, we will show that $\mathcal{T}^1$ maps the ball of radius $\eta_0$ in $\phantom{}_{0}{\mathcal{Y}}_T$ for initial data satisfying $R_1 \leq \eta_0/\Psi(T, \, R_0, \, \eta_0)$.
In order to apply the contraction principle and prove the theorems, we shall therefore prove the continuity estimate \eqref{continuity1}, \eqref{continuity2}. This is the main object of the next sections.

\section{Estimates of the linearised problems}\label{ESTI}

In this section, we present the estimates on which our main results in Theorem \ref{MAIN}, \ref{MAIN2} are footing. The preliminary work done in the paper \cite{bothedruet} shall, in many points, allow to abridge the calculations. The main novelty is the inversion of the parabolic--elliptic subsystem, which shall be dealt with in all details.

To achieve more simplicity in the notation, we introduce both for a function or vector field $f \in W^{2,1}_p(Q_T; \, \mathbb{R}^k)$ ($k\in \mathbb{N}$) and $t \leq T$ the notation
\begin{align}\label{Vfunctor}
 \mathscr{V}(t; \, f) :=  \|f\|_{W^{2,1}_p(Q_t; \, \mathbb{R}^k)} + \sup_{\tau \leq t} \|f(\cdot, \, \tau)\|_{W^{2-\frac{2}{p}}_p(\Omega; \, \mathbb{R}^k)} \, .
\end{align}
Recall that $W^{2-2/p}_p(\Omega)$ is the trace space for $f \in W^{2,1}_p(Q_T)$, $f \mapsto f(\cdot, t)$. Moreover we will need H\"older half-norms. For $\alpha, \, \beta \in [0,\, 1]$ and $f$ scalar valued, we denote
\begin{align*}
[f]_{C^{\alpha}(\Omega)} := \sup_{x \neq y \in \Omega} \frac{|f(x) - f(y)|}{|x-y|^{\alpha}}, \quad [f]_{C^{\alpha}(0,T)} := \sup_{t \neq s \in [0,T]} \frac{|f(t) - f(s)|}{|t-s|^{\alpha}}\\
[f]_{C^{\alpha,\beta}(Q_T)} := \sup_{t \in [0, \, T]} [f(\cdot, \, t)]_{C^{\alpha}(\Omega)} +  \sup_{x \in \Omega} [f(x, \cdot)]_{C^{\beta}(0,T)}  \, .
\end{align*}
The corresponding H\"older norms $\|f\|_{C^{\alpha}(\Omega)}$, $\|f\|_{C^{\alpha}(0,T)}$ and $f \in C^{\alpha,\beta}(Q_T)$ are defined by adding the corresponding $L^{\infty}-$norm to the half-norm.

\subsection{Estimates of a linearised problem for the variables $q$ and $\zeta$}

We first formulate some global assumptions and notations. Recall that $I = ]\varrho_{\min}, \, \varrho_{\max}[$. In this section, the maps $R_q, \, \widetilde{M}: \, I \times \mathbb{R}^{N-2} \rightarrow \mathbb{R}^{(N-2)\times (N-2)}$ are assumed to be of class $C^1(I \times \mathbb{R}^{N-2})$ into the set of symmetric, positive definite matrices. Furtheron, $A: \, I \times \mathbb{R}^{N-2} \rightarrow \mathbb{R}^{N-2}$, and $d: \, I \times \mathbb{R}^{N-2} \rightarrow \mathbb{R}_+$ are of class $C^1$ too. We fix $p > 3$, and we consider given $q^{*} \in W^{2,1}_p(Q_T; \, \mathbb{R}^{N-2})$ and $\varrho^* \in W^{1,1}_{p,\infty}(Q_T)$ such that $\varrho^*(x, \, t) \in ]\varrho_{\min}, \, \varrho_{\max}[$ for all $(x, \, t) \in \overline{Q_T}$. We then denote $R_q^* := R_q(\varrho^*, \, q^{*})$, $ \widetilde{M}^* := \widetilde{M}(\varrho^*, \, q^{*})$, $A^* := A(\varrho^*, \, q^*)$ and $d^* := d(\varrho^*, \, q^*)$. For $t \leq T$, we introduce the positive functions
\begin{align}\label{infimumrho}
m^*(t) := & m(\varrho^*, \, t) := \inf_{(x,\tau) \in Q_t} \min\left\{\frac{\varrho^*(x, \, \tau)}{ \varrho_{\min}} - 1, \, 1-\frac{\varrho^*(x, \, \tau)}{\varrho_{\max}}\right\} \, ,\\
\label{supremumrho}
M^*(t) := & M(\varrho^*, \, t) := \max\left\{\frac{1}{\inf_{(x,\tau) \in Q_t} (\frac{\varrho^*(x, \, \tau)}{ \varrho_{\min}} - 1)},\, \frac{1}{\inf_{(x,\tau) \in Q_t} (1-\frac{\varrho^*(x, \, \tau)}{\varrho_{\max}})  }\right\} \, .
\end{align}
We let $g \in L^p(Q_T; \, \mathbb{R}^{N-2})$, $q^0 \in W^{2-2/p}_p(\Omega; \, \mathbb{R}^{N-2})$ such that $\nu \cdot \nabla q^0(x) = 0$ on $\partial \Omega$ in the sense of traces, and $h \in W^{1,0}_p(Q_T; \, \mathbb{R}^3)$. 

For a pair $(q, \,  \zeta): \, Q_T \rightarrow \mathbb{R}^{N-2} \times \mathbb{R}$ we consider the linear parabolic--elliptic auxiliary problem
\begin{align}
\label{parabolicsystem} R_q^* \,  q_t - \divv (\widetilde{M}^{*} \, \nabla q + A^* \, \nabla \zeta) = & g \, \text{ in } Q_T, \,\, \nu \cdot \nabla q = 0 \text{ on } S_T, \, \, q(x, \, 0) = q^0(x) \text{ in } \Omega \, ,\\
\label{ellipticsystem} -\divv (d^* \, \nabla \zeta + A^* \, \nabla q) = & -\divv h \, \text{ in } Q_T, \,\, \nu \cdot \nabla \zeta = 0 \text{ on } S_T \, ,
\end{align}
and we want to obtain an estimate in the norm of $W^{2,1}_p(Q_T; \, \mathbb{R}^{N-2}) \times W^{2,0}_p(Q_T)$ for the solution. To this aim we first show that \eqref{parabolicsystem}, \eqref{ellipticsystem} can be equivalently reformulated as a system coupled only in the lower order.
\begin{lemma}
We adopt the general assumptions and notations formulated at the beginning of this section. A pair $(q, \,  \zeta) \in W^{2,1}_p(Q_T; \, \mathbb{R}^{N-2}) \times W^{2,0}_p(Q_T)$ is a solution to the problem \eqref{parabolicsystem}, \eqref{ellipticsystem} if the identity \eqref{ellipticsystem} and the initial and boundary condition are satisfied, and if instead of \eqref{parabolicsystem} we have
\begin{align}\label{parabolicsystemreduced}
&  R_{q}^* \, \partial_t q  - \divv ( [\widetilde{M}^*  -\frac{A^* \otimes A^*}{d^*} ] \, \nabla q) = \\
  & \qquad g + \nabla  \zeta \cdot [  \nabla A^* - \frac{A^*}{d^*} \, \nabla d^*] + \nabla (\frac{A^*}{d^*}) \cdot \nabla q \, A^* + \frac{A^*}{d^*} \, \divv h \nonumber \, .
  \end{align}
\end{lemma}
\begin{proof}
Computing the derivatives in the elliptic equation \eqref{ellipticsystem}, we obtain that
\begin{align}\label{umform}
- d^* \, \triangle \zeta = \nabla d^* \cdot \nabla \zeta + \divv(A^* \, \nabla q) - \divv h  \, .
\end{align}
Thus, under the side-condition \eqref{umform}, the parabolic equations \eqref{parabolicsystem} are equivalent to 
\begin{align}\begin{split}\label{a1equiv0}
&  R_{q}^* \, \partial_t q - \divv (\widetilde{M}^* \, \nabla q) = g+ A^* \, \triangle \zeta + \nabla A^* \cdot \nabla  \zeta \\
& \quad = g + \nabla A^* \cdot \nabla  \zeta - \frac{1}{d^*} \,  A^* \, [\nabla d^* \cdot \nabla \zeta + \divv(A^* \, \nabla q) - \divv h ] \, .
\end{split}
 \end{align}
Use of $\frac{A^*}{d^*} \,  \divv(A^* \, \nabla q) = \divv( \frac{A^* \otimes A^*}{d^*} \, \nabla q) - A^* \, \nabla q \cdot \nabla(\frac{A^*}{d^*})$ yields the claim.
\end{proof}
Using this lemma, we next prove an estimate for the solution to the linearised parabolic--elliptic problem.
\begin{prop}\label{qzetasystem}
Under the general assumptions and notations of this section, there is a unique pair $(q, \,  \zeta) \in W^{2,1}_p(Q_T; \, \mathbb{R}^{N-2})\times W^{2,0}_p(Q_T)$, solution to the problem \eqref{parabolicsystem}, \eqref{ellipticsystem}, such that $\int_{\Omega} \zeta(x, \, t) \, dx = 0$ for all $t \in ]0, \, T[$. Moreover, there are a constant $C$ depending only on $\Omega$, and continuous functions $\Psi_1 =  \Psi_1(t, \, a_1, \ldots, a_5)$ and $\Phi = \Phi(t, \, a_1, \ldots, a_5)$ defined for all $t \geq 0$ and all numbers $a_1, \ldots, a_5 \geq 0$, such that for all $t \leq T$ and for $0 < \beta \leq 1$ arbitrary:
\begin{align*} 
& \mathscr{V}(t; \, q) + \|\zeta\|_{W^{2,0}_p(Q_t)} \leq  C \, \Psi_{1,t}\, (1+[\varrho^*]_{C^{\beta,\frac{\beta}{2}}(Q_t)})^{\tfrac{2}{\beta}} \, (\|q^0\|_{W^{2-\frac{2}{p}}_p(\Omega)} + \|g\|_{L^p(Q_t)} +\|h\|_{W^{1,0}_{p}(Q_t)})\\
& \phantom{ \mathscr{V}(t; \, q) + \|\zeta\|_{W^{2,0}_p(Q_t)} \leq  C } + C \, \Phi_t \, \|h\|_{L^p(Q_t)} \, ,\\[0.1cm]
& \Psi_{1,t} = \Psi_1(t, \, M^*(t),  \, \|q^*(0)\|_{C^{\beta}(\Omega)}, \, 
\mathscr{V}(t; \, q^*), \, [\varrho^*]_{C^{\beta,\frac{\beta}{2}}(Q_t)}, \, \|\nabla \varrho^*\|_{L^{p,\infty}(Q_t)}) \, ,\\
& \Phi_t = \Phi(t, \, M^*(t),  \, \|q^*(0)\|_{C^{\beta}(\Omega)}, \, 
\mathscr{V}(t; \, q^*), \, [\varrho^*]_{C^{\beta,\frac{\beta}{2}}(Q_t)}, \, \|\nabla \varrho^*\|_{L^{p,\infty}(Q_t)}) \, .
\end{align*}
The function $\Psi_1$ possesses moreover the following two properties: It is increasing in all arguments, and the value of $\Psi_1(0, \, a_1,\ldots,a_5) = \Psi_1^0(a_1, \, a_2)$ is a function independent on the three last arguments. The function $\Phi$ is increasing in all arguments.
\end{prop}
\begin{proof}
The existence and uniqueness can be easily obtained by means of the uniform estimates. We thus suppose first that $(q, \, \zeta) \in W^{2,1}_p(Q_T; \, \mathbb{R}^{N-2})\times W^{2,0}_p(Q_T)$ is a given solution, and we prove the claimed estimate. In order to simplify the discussion, we adopt the following convention: When computing the derivative of a coefficient, like $\nabla d^* = d^*_{\varrho} \, \nabla \varrho^* + d^*_{q} \, \nabla q^*$, there occur different functions $d_{\varrho}^* := d_{\varrho}(\varrho^*, \, q^*)$ or $d_{q_j}^* = d_{q_j}(\varrho^*, \, q^*)$ of the variables $\varrho^*, \, q^*$. We denote $c_1^* = c_1(M^*(t), \,\|q^*\|_{L^{\infty}(Q_t)})$ a generic continuous function depending only on $M^*(t)$ and $\|q^*\|_{L^{\infty}(Q_t)}$, and increasing in these arguments. We then bound the $L^{\infty}(Q_T)$ norms of all non-linear functions depending on $\varrho^*$, $q^*$ by this generic $c_1^*$. 
\\[1ex]

\textbf{Step 1}: First estimate for the variable $\zeta$.

For almost all $s \leq t$, the function $\zeta$ satisfies the weak Neumann problem
\begin{align*}
 \int_{\Omega} d^* \, \nabla \zeta(x, \, s) \cdot \nabla \phi(x) \, dx = \int_{\Omega} (-A^*\, \nabla q + h)(x, \, s)  \cdot \nabla \phi(x) \, dx \, .
\end{align*}
By well-known weak elliptic theory, there is a unique solution $\zeta(s) \in W^{1,p}(\Omega)$ with $\int_{\Omega} \zeta(x, \,  s) \, dx = 0$. Moreover, for all $0<\beta < 1$, perturbation techniques shortly recalled in the Appendix, Lemma \ref{ellipticlemma} yield the estimate
\begin{align*}
 \|\nabla \zeta(s)\|_{L^p(\Omega)} \leq & c(\Omega,p,\, \inf_{\Omega} d^*(s), \, \sup_{\Omega} d^*(s)) \, (1+[d^*(s)]_{C^{\beta}(\Omega)})^{\frac{1}{\beta}} \times\\
 & \times (\| A^*\, \nabla q(s)\|_{L^p(\Omega)} + \|h(s)\|_{L^p(\Omega)})\\
\leq & c_1^* \, (1+[d^*(s)]_{C^{\beta}(\Omega)})^{\frac{1}{\beta}} \, ( \|A^*(s)\|_{L^{\infty}(\Omega)} \,  \|\nabla q(s)\|_{L^p(\Omega)} + \|h(s)\|_{L^{p}(\Omega)}) \, .
 \end{align*}
We define $ \phi^*_t := \sup_{s \leq t} (1+[d^*(s)]_{C^{\beta}(\Omega)})^{\frac{1}{\beta}}$. We bound $ \sup_{s\leq t} \|A^*(s)\|_{L^{\infty}(\Omega)}$ with a generic $c_1^*$, and it follows that
\begin{align}\label{zetalowerorder}
 \|\zeta\|_{W^{1,0}_p(Q_t)} \leq &  c_1^* \, \phi^*_t \, (\|\nabla q\|_{L^p(Q_t)} + \|h\|_{L^{p}(Q_t)}) \, . 
\end{align}

\textbf{Step 2}: First bound for the variable $q$.

We start from \eqref{parabolicsystemreduced}, and we define $K(\varrho, \, q) := \widetilde{M}(\varrho, \, q) -A(\varrho, \, q) \otimes A(\varrho, \, q)/d(\varrho, \, q)$. In view of the definitions \eqref{trafom}, \eqref{trafoa}, \eqref{trafod}, $K \in \mathbb{R}^{(N-2) \times (N-2)}$ is obviously symmetric, and obeys $$K = \Pi^{\sf T} \, M \, \Pi - \frac{\Pi^{\sf T} \, M \bar{V} \otimes \Pi^{\sf T} \, M \bar{V}}{M \bar{V} \cdot \bar{V}} \, .$$ For all $y \in \mathbb{R}^{N-2}$, $K \, y \cdot y = M \, \Pi y \cdot \Pi y - (M \bar{V} \cdot \Pi y)^2/M \bar{V} \cdot \bar{V} \geq 0$, because $M$ is positive semi--definite.
By the Cauchy-Schwarz inequality, $K y \cdot y = 0$ is possible only if either $\Pi y$ and $\bar{V}$ are parallel, or if $\Pi y$ and $1^N$ are parallel. Recall in this place that $\Pi y = \sum_{k=1}^{N-2} y_k \, \xi^k$. By the choice of the $\xi^k$s, we know that $\{\xi^1,\ldots,\xi^{N-2}, \, \bar{V}, \, 1^N\}$ is a basis of $\mathbb{R}^N$. Thus, $\Pi y = \lambda \, \bar{V}$ or $\Pi y = \lambda \, 1^N$ both would imply that $y = 0$. This shows that $K y \cdot y > 0$ unless $y = 0$, hence $K$ is positive definite.

Defining $K^* := K(\varrho^*, \, q^*)$, we rephrase \eqref{parabolicsystemreduced} as 
\begin{align}\label{parabolicsystemreducedpr}
  R_{q}^* \, \partial_t q - \divv ( K^* \, \nabla q) = & g + \tilde{g} \, ,
  \end{align}
in which $\tilde{g} :=  \nabla  \zeta \cdot [  \nabla A^* - \frac{A^*}{d^*} \, \nabla d^*] + \nabla (\frac{A^*}{d^*}) \cdot \nabla q \, A^* + \frac{A^*}{d^*} \, \divv h$ is bounded via 
\begin{align}\label{tildegnumber}
|\tilde{g}| \leq c_1^* \, ( |\nabla \zeta \cdot \nabla \varrho^*| + |\nabla \zeta \cdot \nabla q^*| + |\nabla q \cdot \nabla \varrho^*| + |\nabla q \cdot \nabla q^*| +  |\nabla h|) \, .
\end{align}
%
We now apply Appendix, Lemma \ref{qsystem}, which basically recalls the result of \cite{bothedruet}, Prop. 7.1 for a similar parabolic system. With $D_0(t) := (1+[\varrho^*]_{C^{\beta,\beta/2}(Q_t)})^{2/\beta} \, \|q^0\|_{W^{2-2/p}_p(\Omega)} + \|g\|_{L^p(Q_t)}$, and using \eqref{tildegnumber} to bound the norm of $\tilde{g}$, we obtain for the solution to \eqref{parabolicsystemreducedpr}
\begin{align}\label{vergessen}
 & \qquad \mathscr{V}(t; \, q)  \leq  C \, \bar{\Psi}_{1,t}\, \big[  D_0(t) + c_1^* \,   \|\nabla h\|_{L^p(Q_t)}  \\
 & + c_1^* \, (\|\nabla \zeta \cdot \nabla \varrho^*\|_{L^p(Q_t)} + \|\nabla \zeta \cdot \nabla q^*\|_{L^p(Q_t)} + \|\nabla q \cdot \nabla \varrho^*\|_{L^p(Q_t)} + \|\nabla q \cdot \nabla q^*\|_{L^p(Q_t)} )       \big] \, ,\nonumber
\end{align}
where $\bar{\Psi}_{1,t} = \bar{\Psi}_1( t, \, M^*(t),  \, \|q^*(0)\|_{C^{\beta}(\Omega)}, \, 
\mathscr{V}(t; \, q^*), \, [\varrho^*]_{C^{\beta,\beta/2}(Q_t)}, \, \|\nabla \varrho^*\|_{L^{p,\infty}(Q_t)}  )$, and the function $\bar{\Psi}_1$ fulfills all structural assumptions stated for $\Psi_{1}$. 

\textbf{Step 3}: Main estimate for the variable $\zeta$.

Since $\zeta \in W^{2,0}_p(Q_T)$, we can employ the pointwise identity \eqref{umform}. Since $\zeta$ has mean--value zero for all times, the full $W^{2,p}$ norm can be estimated by the Neumann-Laplacian, and we obtain that
\begin{align}\label{ESTIMzeta}
 \|\zeta\|_{W^{2,0}_p(Q_t)} \leq & c(\Omega,p) \, \|- \triangle \zeta\|_{L^p(Q_t)} \nonumber\\
= & c(\Omega,p) \, \|(d^*)^{-1} \, (\nabla d^*\cdot \nabla \zeta+  \divv (A^* \, \nabla  q-h))\|_{L^p(Q_t)} \nonumber\\
\leq & c \, \frac{1}{\inf_{(x,s) \in Q_t} d^*(x,s)} \, (\|\nabla d^*\cdot \nabla \zeta\|_{L^p(Q_t)} +   \|\divv (A^* \, \nabla q - h)\|_{L^p(Q_t)}) \, .
\end{align}
Computing the derivatives of the coefficients, and using the same conventions as above, we derive from \eqref{ESTIMzeta} the inequality
\begin{align*}
& \|\zeta\|_{W^{2,0}_p(Q_t)} \leq  c_1^* \, (\|\triangle q\|_{L^p(Q_t)} + \|\divv h\|_{L^p(Q_t)})\\
 & + c_1^* \, (\|\nabla \varrho^*\cdot \nabla \zeta\|_{L^p(Q_t)} + \|\nabla q^*\cdot \nabla \zeta\|_{L^p(Q_t)} + \|\nabla \varrho^*\cdot \nabla q\|_{L^p(Q_t)} + \|\nabla q^*\cdot \nabla q\|_{L^p(Q_t)}) \, .\nonumber
\end{align*}
We estimate $ \|\triangle q\|_{L^p(Q_t)} \leq \mathscr{V}(t; \, q)$, then we employ the inequality \eqref{vergessen} to see that\\[1ex]
\begin{align}\label{ESTIMzeta22}
\|\zeta\|_{W^{2,0}_p(Q_t)} \leq & C \, \bar{\Psi}_{1,t}\,  D_0(t) + c_1^* \, (1 + C\bar{\Psi}_{1,t}) \, \|\nabla h\|_{L^p(Q_t)}\nonumber\\
& + c_1^* \, (1 + C\bar{\Psi}_{1,t}) \,(\|\nabla \varrho^*\cdot \nabla \zeta\|_{L^p(Q_t)} + \|\nabla q^*\cdot \nabla \zeta\|_{L^p(Q_t)} ) \\
 & + c_1^* \, (1 + C\bar{\Psi}_{1,t}) \, ( \|\nabla \varrho^*\cdot \nabla q\|_{L^p(Q_t)} + \|\nabla q^*\cdot \nabla q\|_{L^p(Q_t)}) \, .\nonumber
\end{align}
\textbf{Step 4}: Combined estimates.

We add \eqref{vergessen} to \eqref{ESTIMzeta22} to obtain that
\begin{align}\label{ESTIMcombined}
\mathscr{V}(t; \, q) + \|\zeta\|_{W^{2,0}_p(Q_t)} \leq & 2C \, \bar{\Psi}_{1,t}\,  D_0(t) + c_1^* \, (1 + 2C\bar{\Psi}_{1,t}) \, \|\nabla h\|_{L^p(Q_t)}\nonumber\\
& + c_1^* \, (1 + 2C\bar{\Psi}_{1,t}) \,(\|\nabla \varrho^*\cdot \nabla \zeta\|_{L^p(Q_t)} + \|\nabla q^*\cdot \nabla \zeta\|_{L^p(Q_t)} ) \\
 & + c_1^* \, (1 + 2C\bar{\Psi}_{1,t}) \, ( \|\nabla \varrho^*\cdot \nabla q\|_{L^p(Q_t)} + \|\nabla q^*\cdot \nabla q\|_{L^p(Q_t)}) \, .\nonumber
\end{align}
In order to control the factors on the right-hand, we first apply \eqref{gagliardo} to find that $$\|\nabla \zeta(s)\|_{L^{\infty}(\Omega)} \leq C_1 \, \|D^2\zeta(s)\|_{L^p(\Omega)}^{\alpha} \, \|\zeta(s)\|_{L^p(\Omega)}^{1-\alpha} + C_2 \, \|\zeta(s)\|_{L^p(\Omega)} \, , \quad \alpha := \frac{1}{2}+\frac{3}{2p} \, ,$$
with $C_i = C_i(\Omega)$, $i=1,2$. We can bound $a \, b \leq \epsilon \, a^{1/\alpha} + c_{\alpha} \, \epsilon^{-\alpha/(1-\alpha)} \, b^{1/(1-\alpha)}$ (Young's inequality), for all $\epsilon > 0$ and $a, \, b > 0$. By these means, it follows that
\begin{align}\label{ESTIMzeta3} 
\|\nabla & \varrho^*  \cdot\nabla \zeta\|_{L^p(Q_t)}^p \leq \int_{0}^t |\nabla \varrho^*(s)|_{p}^p \, |\nabla \zeta(s)|_{\infty}^p \, ds\\
& \leq C_1 \, \int_{0}^t |\nabla \varrho^*(s)|_{p}^p \, |D^2\zeta(s)|_{p}^{p\alpha} \, |\zeta(s)|_{p}^{p(1-\alpha)} \, ds + C_2 \,\int_{0}^t |\nabla \varrho^*(s)|_{p}^p \, |\zeta(s)|_{p}^p \, ds\nonumber\\
& \leq  \epsilon \, \int_{0}^t  |D^2\zeta(s)|_{p}^{p} \, ds + c_{\alpha} \, \epsilon^{-\frac{\alpha}{1-\alpha}} \, \int_{0}^t |\nabla \varrho^*(s)|_{p}^{\frac{p}{1-\alpha}} \, |\zeta(s)|_{p}^{p} \, ds + C_2 \,\int_{0}^t |\nabla \varrho^*(s)|_{p}^p \, |\zeta(s)|_{p}^p \, ds\nonumber\\
& \leq \epsilon \, \int_{0}^t  |D^2\zeta(s)|_{p}^{p} \, ds + \int_{0}^t |\zeta(s)|_{p}^{p} \, (c_{\alpha} \, \epsilon^{-\frac{\alpha}{1-\alpha}}\, |\nabla \varrho^*(s)|_{p}^{\frac{p}{1-\alpha}} + C_2\,  |\nabla \varrho^*(s)|_{p}^p) \, ds  \, .\nonumber
\end{align}
Here we use the abbreviation $| \cdot |_r$ for $\|\cdot\|_{L^r(\Omega)}$. Just in the same way, we show that
\begin{align}\label{ESTIMzeta4}
 \|\nabla & q^*\cdot\nabla \zeta\|_{L^p(Q_t)}^p \nonumber\\
&  \leq \epsilon \, \int_{0}^t  |D^2\zeta(s)|_{p}^{p} \, ds  + \int_{0}^t |\zeta(s)|_{p}^{p} \, (c_{\alpha} \, \epsilon^{-\frac{\alpha}{1-\alpha}} \,  |\nabla q^*(s)|_{p}^{\frac{p}{1-\alpha}} + C_2\,  |\nabla q^*(s)|_{p}^p) \, ds \, .
\end{align}
We let $F^*(t) := \sup_{s\leq t}( \|\nabla q^*(s)\|_{L^p(\Omega)}^p+  \|\nabla \varrho^*(s)\|_{L^p(\Omega)}^p) $ and $X^*(t; \, \zeta) :=  \|\nabla \varrho^*\cdot\nabla \zeta\|^p_{L^p(Q_t)} + \|\nabla q^* \cdot\nabla \zeta\|^p_{L^p(Q_t)}$. With the help of \eqref{ESTIMzeta3} and of \eqref{ESTIMzeta4}, it follows that
 \begin{align*}
X^*(t; \, \zeta) \leq  2 \, \epsilon \, \|D^2\zeta(s)\|_{L^p(Q_t)}^{p} + [ c_{\alpha} \, \epsilon^{-\frac{\alpha}{1-\alpha}} \,  (F^*(t))^{\frac{1}{1-\alpha}} + C_2\,  F^*(t)]  \, \|\zeta\|_{L^{p}(Q_t)}^p \, . 
 \end{align*}
 We choose $\epsilon = 2^{-2-1/p}\,(c_1^*(1+2C\bar{\Psi}_{1,t}))^{-p}$, where $c_1^*$, $C$ and $\bar{\Psi}_{1,t}$ are the numbers occurring in the relation \eqref{ESTIMcombined}. Then\\[1ex]
 \begin{align}
 (c_1^* \, (1+2C\bar{\Psi}_{1,t}))^p \, X^*(t; \, \zeta) \leq  \frac{1}{2^{1+1/p}} \,  \|\zeta\|_{W^{2,0}_p(Q_t)}^p \label{ESTIMzeta5}\\
 + (c_1^* \,  (1+2C\bar{\Psi}_{1,t}))^p\, [ c_{\alpha} \, (2^{2+1/p}(c_1^*(1+2C\bar{\Psi}_{1,t}))^{p})^{\frac{\alpha}{1-\alpha}} \,  (F^*(t))^{\frac{1}{1-\alpha}} + C_2\,  F^*(t)] \, \|\zeta\|_{L^{p}(Q_t)}^p\, .\nonumber
 \end{align}
 Due to our conventions, we can bound every power of $c_1^*$ and the maximum of $1$ and $c_1^*$ again by another such function. Introducing a factor
 \begin{align*}
  (\Phi_{1,t}^*)^p := & c_1^* \, (1+2C\bar{\Psi}_{1,t})^p \, \max\{ c_{\alpha} \, (2^{2+1/p} \, (1+2C\bar{\Psi}_{1,t})^{p})^{\frac{\alpha}{1-\alpha}}, \, C_2\}\\
 & \times \{ (\mathscr{V}^p(t; \, q^*) + \|\nabla \varrho^*\|_{L^{p,\infty}(Q_t)}^p)^{\frac{1}{1-\alpha}} + (\mathscr{V}^p(t; \, q^*) + \|\nabla \varrho^*\|_{L^{p,\infty}(Q_t)}^p)\} \, ,
 \end{align*}
we can rephrase \eqref{ESTIMzeta5} as
\begin{align}\label{ESTIMzeta5neu}
 &  (c_1^* \, (1+2C\bar{\Psi}_{1,t}))^p \, X^*(t; \, \zeta) \leq  \frac{1}{2^{1+1/p}} \,  \|\zeta\|_{W^{2,0}_p(Q_t)}^p + (\Phi_{1,t}^*)^p \,  \|\zeta\|_{L^{p}(Q_t)}^p\, .
 \end{align}
By means of \eqref{zetalowerorder}, we bound $\|\zeta\|_{L^{p}(Q_t)}\leq c_1^* \, \phi^*_t \, (\|\nabla q\|_{L^p(Q_t)} + \|h\|_{L^p(Q_t)})$. 
Raising \eqref{ESTIMzeta5neu} to the power $1/p$, we show that
\begin{align*}
& c_1^* \, (1+2C\bar{\Psi}_{1,t}) \, (\|\nabla \varrho^*\cdot\nabla \zeta\|_{L^p(Q_t)} + \|\nabla q^*\cdot\nabla \zeta\|_{L^p(Q_t)} ) \nonumber\\
& \qquad \leq  \frac{1}{2} \,  \|\zeta\|_{W^{2,0}_p(Q_t)}+ C_p \, c_1^* \, \phi^*_t\, \Phi_{1,t}^* \, (\|\nabla q\|_{L^p(Q_t)} + \|h\|_{L^p(Q_t)})  \, .
\end{align*}
We insert the latter result into \eqref{ESTIMcombined}, obtaining
\begin{align}\label{ESTIMcombined2}
 & \qquad \mathscr{V}(t; \, q) +\frac{1}{2}\,  \|\zeta\|_{W^{2,0}_p(Q_t)}  \leq  2C \, \bar{\Psi}_{1,t}\,  D_0(t) \nonumber\\
 & + c_1^* \, (1 + 2C\bar{\Psi}_{1,t}) \, \|\nabla h\|_{L^p(Q_t)} + C_p \,  c_1^* \, \phi^*_t\, \Phi_{1,t}^* \, \|h\|_{L^p(Q_t)}\\
 & + c_1^* \, (1 + 2C\bar{\Psi}_{1,t}) \, ( \|\nabla \varrho^*\cdot \nabla q\|_{L^p(Q_t)} + \|\nabla q^*\cdot \nabla q\|_{L^p(Q_t)}) +  C_p \,c_1^* \, \phi^*_t\, \Phi_{1,t}^*  \,\|\nabla q\|_{L^p(Q_t)}  \, .\nonumber
\end{align}
In order to estimate $X^*(t, \, q)$, we apply the same steps as for $X^*(t, \, \zeta)$ (cf.\ \eqref{ESTIMzeta5neu}). Hence
\begin{align*}
 &  (c_1^* \, (1+2C\bar{\Psi}_{1,t}))^p \, X^*(t; \, q) \leq  \frac{1}{2^{1+1/p}} \,  \|q\|_{W^{2,0}_p(Q_t)}^p + (\Phi_{1,t}^*)^p \,  \|q\|_{L^{p}(Q_t)}^p\, ,
 \end{align*}
which, after raising to the power $1/p$, yields
\begin{align*}
 c_1^* \, (1+2C\bar{\Psi}_{1,t}) \, (\|\nabla \varrho^*\cdot\nabla q\|_{L^p(Q_t)} + \|\nabla q^*\cdot\nabla q\|_{L^p(Q_t)} ) 
\leq  \frac{1}{2} \,  \|q\|_{W^{2,0}_p(Q_t)} + C_p \,  \Phi_{1,t}^* \, \|q\|_{L^p(Q_t)}   \, .
\end{align*}
Since $\|q\|_{W^{2,0}_p(Q_t)} \leq \mathscr{V}(t; \, q)$, the latter and \eqref{ESTIMcombined2} imply that
\begin{align}\label{ESTIMcombined3}
&  \frac{1}{2}\,(\mathscr{V}(t; \, q) + \|\zeta\|_{W^{2,0}_p(Q_t)})  \leq C_p \,  \Phi_{1,t}^* \,( 1 +  c_1^* \, \phi^*_t)  \,\|q\|_{W_p^{1,0}(Q_t)}\\
  & \quad + 2C \, \bar{\Psi}_{1,t}\,  D_0(t)  + c_1^* \, (1 + 2C\bar{\Psi}_{1,t}) \,\|\nabla h\|_{L^p(Q_t)} + C_p\, c_1^* \, \phi^*_t\, \Phi_{1,t}^* \, \|h\|_{L^p(Q_t)}\nonumber\, .
\end{align}
In order to finally get rid of the factors with $q$ on the right-hand side, we introduce
\begin{align*}
 [A(t)]^{\frac{1}{p}} := & 2C \, \bar{\Psi}_{1,t}\,  D_0(t)  + c_1^* \, (1 + 2C\bar{\Psi}_{1,t}) \,\|\nabla h\|_{L^p(Q_t)} + C_p\, c_1^* \, \phi^*_t\, \Phi_{1,t}^* \, \|h\|_{L^p(Q_t)}\, ,
 \end{align*}
 and
 \begin{align*}
 [B(t)]^{\frac{1}{p}} := & C_p \,  \Phi_{1,t}^* \, (1+c_1^* \, \phi^*_t) \, , \qquad  f(t) := \sup_{\tau \leq t} \|q(\tau)\|_{W^{2-\frac{2}{p}}_p(\Omega)}^p\, .
\end{align*}
We raise \eqref{ESTIMcombined3} to the $p^{\rm th}$ power. We use $f(t) \leq\mathscr{V}^p(t; \, q)$ and $\|q\|_{W^{1,0}_p(Q_t)}^p \leq \int_0^t f(\tau) \, d\tau$. In this way, we obtain the inequality $f(t) \leq 2^p \, A(t) + 2^p\,  B(t) \, \int_{0}^t f(\tau) \, d\tau$. Using that $A$ and $B$ are monotone increasing by construction, the Gronwall Lemma yields $f(t) \leq 2^p \, A(t) \, \exp(2^p\, t \, B(t))$. In particular, we conclude that
\begin{align*}
 \|q\|_{W^{1,0}_p(Q_t)} \leq [f(t) \, t]^{\frac{1}{p}} \leq c_p \,  t^{\frac{1}{p}} \,  [A(t)]^{\frac{1}{p}} \, \exp(\frac{2^p}{p} \, t \, B(t)) \, .
\end{align*}
Combining the latter with \eqref{ESTIMcombined3}, it follows that
\begin{align}\label{ESTIMcombined4}
& \mathscr{V}(t; \, q) + \|\zeta\|_{W^{2,0}_p(Q_t)}  \leq  2 \, \{1 +  \tilde{c}_p \, t^{\frac{1}{p}} \, \exp(\frac{2^p}{p} \, t \, B(t)) \, [B(t)]^{\frac{1}{p}} \}\nonumber\\
 &\quad  \times   \, \{2C \, \bar{\Psi}_{1,t}\,  D_0(t)  + c_1^* \, (1 + 2C\bar{\Psi}_{1,t}) \, \|\nabla h\|_{L^p(Q_t)}+ C_p \, c_1^* \, \phi^*_t\, \Phi_{1,t}^* \, \|h\|_{L^p(Q_t)}\} \, .
\end{align}
In order to verify that the factors occurring in the latter inequality possess the structure as claimed in the statement, we note that occurrences of  $B(t)$ in \eqref{ESTIMcombined4} are multiplied by a power of $t$, so that they do not occur at $t = 0$. Moreover, the factor $\bar{\Psi}_{1,t}$ possesses the structure required for $\Psi_{1,t}$ in the statement. In order 
to estimate the dependence of $\|q^*\|_{L^{\infty}(Q_t)}$ on the coefficients $c_1^*$, we apply the same strategy as in the section 7 of \cite{bothedruet}: $\|q^*\|_{L^{\infty}(Q_t)} \leq \|q^0\|_{L^{\infty}(\Omega)} + t^{\gamma} \, \mathscr{V}(t; \, q^*)$ (Lemma \ref{HOELDER}). Setting $\Phi_t := C_p \, c_1^* \, \phi^*_t\, \Phi_{1,t}^*$, we are done.
\end{proof}

\subsection{Estimates for linearised problems for the variables $v$ and $\varrho$}

First we state the estimate for the linearised momentum equation. The proof follows the lines of the corresponding result in \cite{bothedruet}. (Since we can assume $\varrho^* \in [\varrho_{\min}, \, \varrho_{\max}]$, the proof is actually simpler.)
\begin{prop}\label{vsystem}
Assume that $\varrho^* \in C^{\alpha,0}(Q_t)$ ($0< \alpha \leq 1$) attains values in $]\varrho_{\min}, \, \varrho_{\max}[$,
that $f \in L^p(Q_T; \, \mathbb{R}^3)$, and that $v^0 \in W^{2-2/p}_p(\Omega; \,  \mathbb{R}^3)$ is such that $v^0 = 0$ on $\partial \Omega$. Then, there is a unique solution $v \in W^{2,1}_p(Q_T; \, \mathbb{R}^3)$ to $\varrho^* \, \partial_t v - \divv \mathbb{S}(\nabla v) = f$ in $Q_T$ with the boundary conditions $v = 0$ on $S_T$ and $v(x, \, 0) = v^0(x)$ in $\Omega$. Moreover, there is $C$ independent on $t$, $\varrho^*$, $v^0$, $f$ and $v$ such that
 \begin{align*}
\mathscr{V}(t; \, v) \leq & C \, \Psi_2(t, \, \sup_{\tau\leq t} [\varrho^*(\tau)]_{C^{\alpha}(\Omega)}) \, (1+ \sup_{\tau\leq t} [\varrho^*(\tau)]_{C^{\alpha}(\Omega)})^{\frac{2}{\alpha}} \, (\|f\|_{L^p(Q_t)} + \|v^0\|_{W^{2-2/p}_p(\Omega)})\, .
\end{align*}
The function $\Psi_2$ is continuous and increasing in both arguments, and it can be chosen such that $\Psi_2(0, \, a)  = (\min\{1, \, \varrho_{\min}\})^{-\frac{2}{\alpha}} \, (\varrho_{\max}/\varrho_{\min})^{\frac{p+1}{p}}$ is independent of $a$.
\end{prop}
For the linearised continuity equation, we must acknowledge the main difference with respect to the analysis of the compressible models.
\begin{prop}\label{contieq}
Assume that $v^* \in W^{2,1}_p(Q_T; \, \mathbb{R}^3)$ and that $\varrho_0 \in W^{1,p}(\Omega)$ satisfies $\varrho_{\min} < \varrho_0(x) < \varrho_{\max}$ in $\overline{\Omega}$. We define $M_0 = M(\varrho_0, \, 0) := [\inf_{x \in \Omega} \{\varrho_0(x)/\varrho_{\min} - 1, \, 1-\varrho_0(x)/\varrho_{\max}\}]^{-1}$.
Then the problem $\partial_t \varrho + \divv(\varrho \, v^*) = 0$ in $Q_T$ with $\varrho(x, \, 0) = \varrho_0(x)$ in $\Omega$ possesses a unique strictly positive solution of class $W^{1,1}_{p,\infty}(Q_T)$. Define also $M(t) := M(\varrho, \, t)$ (cf. \eqref{supremumrho}). Then, we can find a constant $c$ depending only on $\Omega$ and a function $\Psi_3 = \Psi_3(t, \, a_1, \, a_2)$ continuous and finite in the set 
\begin{align*}
\{t, \, a_1, \, a_2 \geq 0 \, : \, c \, a_1 \, t^{1-\frac{1}{p}} \, a_2 \, e^{c \, t^{1-\frac{1}{p}} \, a_2} \, < \,1\} \, ,
\end{align*}
such that $M(t) \leq  \Psi_3(t,  \, M_0, \, \mathscr{V}(t; \, v^*))$. Moreover, for $\beta = 1-3/p$, there are $\Psi_4, \, \Psi_5$ depending on $t$, $\|\nabla \varrho_0\|_{L^{p}(\Omega)}$ and $\mathscr{V}(t; \, v^*)$ such that
 \begin{align*}
 \|\nabla \varrho\|_{L^{p,\infty}(Q_t)} \leq & \Psi_4(t,\,  \|\nabla \varrho_0\|_{L^{p}(\Omega)}, \, \mathscr{V}(t; \, v^*)), \quad 
 [\varrho]_{C^{\beta,\frac{\beta}{2}}(Q_t)} \leq & \Psi_5(t, \, \|\nabla \varrho_0\|_{L^{p}(\Omega)}, \, \mathscr{V}(t; \, v^*)) \, .
\end{align*}
For $i=3,4,5$, $\Psi_i$ is continuous and increasing in all variables, and $\Psi_i(0, \, a_1, a_2) = \Psi^0_i(a_1)$ is independent on the last variable. The identity $\Psi_4(0, \, a_1, \, a_2) = a_1$, and the inequality $\Psi_5(0, \, a_1, \, a_2) \leq C \, a_1$, are also valid.
\end{prop}
\begin{proof}
The existence statement as well as the construction of the functions $\Psi_4$ and $\Psi_5$ is proved in \cite{bothedruet}, Corollary 7.8. The critical point is the construction of the function $\Psi_3$. We start from the well--known representation of the solution to the continuity equation (see a. o.\ \cite{solocompress}) 
\begin{align*}
 \varrho(x, \, t) := \varrho_0(y(0; \, x, \, t)) \, \exp\left(-\int_{0}^t \divv v^*(y(\tau; \, x,t), \, \tau) \, d\tau\right) \, ,
 \end{align*}
 where $y(\tau; \, x, \, t)$ is the characteristic curve with speed $v^*$ through $(x, \, t)$. Therefore,
 \begin{align*}
  \varrho_{\max} - \varrho = & \varrho_{\max} - \varrho_0(y(0; \, x, \, t)) + \varrho_0(y(0; \, x, \, t)) \, \left(1 - \exp\left(-\int_{0}^t \divv v^*(y(\tau; \, x,t), \, \tau) \, d\tau\right)\right)\\
  \geq&  \varrho_{\max}  \, \left(\frac{1}{M_0} - \left|1 - \exp\left(-\int_{0}^t \divv v^*(y(\tau; \, x,t), \, \tau) \, d\tau\right)\right|\right) \, .
 \end{align*}
Use of $|1 - e^b| \leq e^{|b|} \, |b|$ allows to bound
\begin{align*}
&   \left|1 - \exp\left(-\int_{0}^t \divv v^*(y(\tau; \, x,t), \, \tau) \, d\tau\right)\right| \\
& \quad  \leq \exp\left(\int_{0}^t \|\divv v^*(\tau)\|_{L^{\infty}(\Omega)} \, d\tau\right) \, \int_{0}^t \|\divv v^*(\tau)\|_{L^{\infty}(\Omega)} \, d\tau \, .
\end{align*}
Owing to the continuity of $W^{1,p}(\Omega) \subset L^{\infty}(\Omega)$ and H\"older's inequality
\begin{align*}
\|\divv v^*\|_{L^{\infty,1}(Q_t)} \leq c_{\Omega} \, \int_{0}^t \|\divv v^*(\tau)\|_{W^{1,p}(\Omega)} \, d\tau \leq c_{\Omega} \, t^{1-\frac{1}{p}} \, \|v^*\|_{W^{2,0}_p(Q_t)} \, .
\end{align*}
 Thus $1- \varrho/\varrho_{\max} \geq 1/M_0 -  c_{\Omega} \, t^{1-\frac{1}{p}} \,  \mathscr{V}(t; \, v^*) \,  \exp(c_{\Omega} \, t^{1-\frac{1}{p}} \,  \mathscr{V}(t; \, v^*))$.  Thanks to a similar argument applied to $\varrho_{\min} - \varrho$, we find that
\begin{align}\label{Psi3}
M(t) \leq \frac{M_0}{1 - c_{\Omega} \, M_0 \, t^{1-\frac{1}{p}} \,  \mathscr{V}(t; \, v^*) \, e^{c_{\Omega} \, t^{1-\frac{1}{p}} \,  \mathscr{V}(t; \, v^*)}}
\end{align}
and define the function $\Psi_3$ to be the right-hand of the latter relation. 
\end{proof}

\section{The continuity estimate for $\mathcal{T}$}\label{contiT}

We now want to combine the Propositions \ref{qzetasystem} and \ref{vsystem} with the linearisation of the continuity equation in Proposition \ref{contieq} to study the fixed point map $\mathcal{T}$ described at the beginning of Section \ref{twomaps} and defined by the equations \eqref{linearT1}, \eqref{linearT2}, \eqref{linearT3}, \eqref{linearT4} for given $v^* \in W^{2,1}_p(Q_T; \, \mathbb{R}^3)$ and $q^* \in W^{2,1}_p(Q_T; \, \mathbb{R}^{N-2})$. We define $\mathscr{V}^*(t) := \mathscr{V}(t; \, q^*) +  \mathscr{V}(t; \, v^*)$. At first we state estimates for the lower--order nonlinearities \eqref{gri}, \eqref{fri}. 
\begin{lemma}\label{fandg}
For $u^* = (q^*, \, \zeta^*, \, \varrho^*, \,  v^*) \in \mathcal{X}_{T,I}$, define $g^* := g(x, \, t, \, u^*,\, D^1_xu^*)$ and $f^* := f(x, \, t, \, u^*,\, D^1_xu^*)$ via \eqref{gri} and \eqref{fri}. There are continuous $\Psi_{g}, \, \Psi_{f} = \Psi(t, \, a_1,\ldots,a_4)$ defined for all $t \geq 0$ and $a_1,\ldots,a_4 \geq 0$ such that
\begin{align*}
 \|g^*\|_{L^p(Q_t)} \leq \Psi_{g}(t, \, M^*(t), \, \|(q^*(0), \, v^*(0))\|_{W^{2-2/p}_p(\Omega)}, \, \|\nabla \varrho^*\|_{L^{p,\infty}(Q_t)}, \, \mathscr{V}^*(t)) \, ,\\
 \|f^*\|_{L^p(Q_t)} \leq \Psi_{f}(t, \, M^*(t), \, \|(q^*(0), \, v^*(0))\|_{W^{2-2/p}_p(\Omega)}, \, \|\nabla \varrho^*\|_{L^{p,\infty}(Q_t)}, \, \mathscr{V}^*(t)) \, .
\end{align*}
$\Psi_g$ and $\Psi_f$ are increasing in all arguments with $\Psi_{g}(0, \, a_1,\ldots,a_4) = 0 = \Psi_{f}(0, \, a_1,\ldots,a_4)$.
\end{lemma}
These estimates were proved in \cite{bothedruet} for the case that the non-linear coefficients $R, \, \widetilde{M}$ are defined for $\varrho^*$ taking values in $]0, \, + \infty[$. The proof is exactly the same for $\varrho^* $ taking values in $I$, provided that we adapt the definition of $m^*(t), \, M^*(t)$ via \eqref{supremumrho}. Moreover, the arguments are very similar to the ones used to bound the right-hand vector field $h$. This statement, that we next prove in detail, might serve as an illustration.
\begin{lemma}\label{h}
Consider $u^* = (q^*, \, \zeta^*, \, \varrho^*, \,  v^*) \in \mathcal{X}_{T,I}$. Define $h^* := h(x, \, t, \, u^*)$ via \eqref{hri}. Then there is a continuous function $\Psi_{h} = \Psi_{h}(t, \, a_1,\ldots,a_4)$ defined for all $t \geq 0$ and $a_1,\ldots,a_4 \geq 0$ such that
\begin{align*}
\|h^*\|_{W^{1,0}_p(Q_t)} \leq & \Psi_{h}(t, \, M^*(t), \, \|q^*(0)\|_{W^{2-2/p}_p(\Omega)}, \, \|\nabla \varrho^*\|_{L^{p,\infty}(Q_t)}, \, \mathscr{V}^*(t)) \, .
\end{align*}
The function $\Psi_h$ is increasing in all arguments. Moreover $\Psi_{h}(0, \, a_1, \ldots ,a_4) = 0$.
\end{lemma}
\begin{proof}
 Recall that $h := A(\varrho^*, \, q^*) \, \tilde{b}(x, \,t) + d(\varrho^*, \, q^*) \, \hat{b}(x, \, t)$. With $c_1^*$ as in the proof of Prop. \ref{qzetasystem}, we bound $ |d(\varrho^*, \, q^*) \, \hat{b}| \leq c_1^* \, |\hat{b}|$ and $|A(\varrho^*, \, q^*) \, \tilde{b}| \leq  c_1^* \, |\tilde{b}|$. Hence $\|h\|_{L^{p}(Q_t)} \leq c_1^* \, (\|\tilde{b}\|_{L^{p}(Q_t)} +\|\hat{b}\|_{L^{p}(Q_t)})$. Lemma \ref{HOELDER} allows to bound $\|q^*\|_{L^{\infty}(Q_t)} \leq \|q^0\|_{L^{\infty}(\Omega)} + t^{\gamma} \, \mathscr{V}^*(t)$ and, evidently, $\|q^0\|_{L^{\infty}(\Omega)} \leq C \, \|q^0\|_{W^{2-2/p}_p(\Omega)}$. 
 
 We then define a function $\Psi^1_h(t, \, a_1, \, \ldots, \, a_4) :=  (\|\tilde{b}\|_{L^{p}(Q_t)} +\|\hat{b}\|_{L^{p}(Q_t)}) \, c_1(a_1, \, a_2 + t^{\gamma} \,a_4 )$. We see that $\Psi^{1}_h$ satisfies $\Psi_h^{1}(0, \, a_1,\ldots,a_4) = 0$, and $\|h\|_{L^p(Q_t)} \leq \Psi^1_{h,t}$.

We compute $h^*_x$, and readily show a bound $ |h^*_x| \leq c_1^* \, ((|\varrho^*_x| +|q^*_x|)  \, (|\tilde{b}| + |\hat{b}|) + |\tilde{b}_x| + |\hat{b}_x|)$. Hence
\begin{align*}
 & \|h^*_x\|_{L^p(Q_t)} \\
&  \leq  c_1^* \, (\|\varrho^*_x\|_{L^{p,\infty}(Q_t)} + \|q^*_x\|_{L^{p,\infty}(Q_t)}) \, (\|\tilde{b}\|_{L^{\infty,p}(Q_t)} + \|\hat{b}\|_{L^{\infty,p}(Q_t)}) + c_1^* \,  (\|\tilde{b}_x\|_{L^p(Q_t)} + \|\hat{b}_x\|_{L^p(Q_t)})\\
 & \leq  c_1^* \, [(\|\tilde{b}\|_{L^{\infty,p}(Q_t)} + \|\hat{b}\|_{L^{\infty,p}(Q_t)}) \, (\|\varrho^*_x\|_{L^{p,\infty}(Q_t)} + \mathscr{V}^*(t)) + \|\tilde{b}_x\|_{L^p(Q_t)} + \|\hat{b}_x\|_{L^p(Q_t)}] =:  \Psi^{2}_h \, .
\end{align*}
We use again Lemma \ref{HOELDER} to control $c_1^*$, seeing thus that the function $\Psi^{2}_h$ also possesses the desired structure ($\Psi^{2}_h = 0$).
\end{proof}
We are now ready to establish the final estimate that allows to obtain the self-mapping property.
\begin{prop}\label{estimateself}
For $(q^*, \, v^*) \in \mathcal{Y}_T$, the solution $(q, \, v) = \mathcal{T}(q^*, \, v^*)$ to the equations \eqref{linearT1}, \eqref{linearT2}, \eqref{linearT3}, \eqref{linearT4} exists and is unique in the class $\mathcal{Y}_t$ for all $t$ subject to
\begin{align}\label{existinterval}
 c \, M_0 \,  t^{1-\frac{1}{p}} \, \mathscr{V}^*(t) \, e^{ c \, t^{1-\frac{1}{p}} \, \mathscr{V}^*(t)} <1 \, ,
\end{align}
where $c= c(\Omega)$ and $M_0$ are the same as in Prop. \ref{contieq}. There is a continuous function $\Psi_6 = \Psi_6(t, \, a_1,\ldots,a_4)$ defined for all $t\geq 0$ and $a_1 \ldots a_4 \geq 0$ subject to the restriction
\begin{align}\label{RESTRIC}
 c \, a_1 \,  t^{1-\frac{1}{p}} \, a_4 \,  e^{ c \, t^{1-\frac{1}{p}} \, a_4} <1 \, ,
\end{align}
such that $\mathscr{V}(t; \, q) + \mathscr{V}(t; \, v) \leq \Psi_6(t, \, M_0, \,  \|(q^0, \, v^0)\|_{W^{2-2/p}_p(\Omega)}, \, \|\nabla \varrho_0\|_{L^p(\Omega)}, \, \mathscr{V}^*(t))$. The function $\Psi_6$ is increasing in all arguments and
\begin{align*}
& \Psi_6(0, \,  M_0,  \,  \|(q^0, \, v^0)\|_{W^{2-2/p}_p(\Omega)}, \, \|\nabla \varrho_0\|_{L^p(\Omega)}, \, \eta)  = \Psi_6^0(M_0, \,  \|(q^0, \, v^0)\|_{W^{2-2/p}_p(\Omega)}, \,   \|\nabla \varrho_0\|_{L^p(\Omega)}) \, 
\end{align*}
for all $\eta > 0$.
\end{prop}
\begin{proof}
Applying Prop. \ref{contieq}, we first find the global solution $\varrho$ to the continuity equation \eqref{linearT1} with data $v^*$ on $[0,T]$. The number $M(t)$ expressing the distance of the solution $\varrho$ to the thresholds $\{\varrho_{\min}, \, \varrho_{\max}\}$ remains finite for all $t$ subject to the restriction \eqref{existinterval} (see Prop. \ref{contieq}). On this time interval, we can therefore insert $(\varrho, \, q^*)$ into the coefficients of the system \eqref{linearT2}, \eqref{linearT3}. Applying Prop. \ref{qzetasystem}, we find a unique solution $(q, \, \zeta) \in W^{2,1}_p(Q_t; \, \mathbb{R}^{N-2}) \times W^{2,0}_p(Q_t)$. We then use $(\varrho, \, q^*) $ and $\zeta$ as data of the system \eqref{linearT4}. Applying Proposition \ref{vsystem}, we obtain a solution $v \in W^{2,1}_p(Q_t; \, \mathbb{R}^{3})$ for all $t$ subject to \eqref{existinterval}. This shows that $(q, \, v) := \mathcal{T}(q^*, \, v^*)$ is well defined in $\mathcal{Y}_t$ for all $t$ subject to \eqref{existinterval}.

In order to verify the estimates, we first recall the outcome of Proposition \ref{qzetasystem} applied with $\varrho^* = \varrho$. It follows that 
\begin{align}\label{ESTIMqzeta}
 & \mathscr{V}(t; \, q) + \|\zeta\|_{W^{2,0}_p(Q_t)} \nonumber\\
&  \leq C \, \Psi_1(t, \,  M(t),  \, \|q^*(0)\|_{ C^{\beta}(\Omega)}, \, 
\mathscr{V}(t; \, q^*), \, [\varrho]_{C^{\beta,\frac{\beta}{2}}(Q_t)}, \, \|\nabla \varrho\|_{L^{p,\infty}(Q_t)}) \,  \times\nonumber\\ 
& \qquad \times (1+[\varrho]_{C^{\beta,\frac{\beta}{2}}(Q_t)})^{\tfrac{2}{\beta}} \, (\|q^0\|_{W^{2-2/p}_p(\Omega)} + \|g^*\|_{L^p(Q_t)} +\|h^*+v^*\|_{W^{1,0}_p(Q_t)})\nonumber\\
& + C \, \Phi(t, \,  M(t),  \, \|q^*(0)\|_{ C^{\beta}(\Omega)}, \, 
\mathscr{V}(t; \, q^*), \, [\varrho]_{C^{\beta,\frac{\beta}{2}}(Q_t)}, \, \|\nabla \varrho\|_{L^{p,\infty}(Q_t)}) \, \|h^*+v^*\|_{L^p(Q_t)} \, .
\end{align}
Evidently $\|v^*\|_{W^{1,0}_p(Q_t)} \leq t^{\frac{1}{p}} \, \sup_{\tau \leq t} \|v(\tau)\|_{W^{1,p}(\Omega)} \leq t^{\frac{1}{p}} \,  \mathscr{V}(t; \, v^*)$. 
For the choices $\varrho^* = \varrho$ and $\beta := 1-3/p$, Proposition \ref{contieq} yields 
\begin{align*}
M(t) \leq & \Psi_3(t, \, M_0, \, \mathscr{V}(t; \, v^*)) =: \Psi_3(t, \ldots) \, , \\
  \|\nabla \varrho\|_{L^{p,\infty}(Q_t)} \leq & \Psi_4(t, \, \|\nabla \varrho_0\|_{L^p(\Omega)}, \, \mathscr{V}(t; \, v^*)) =: \Psi_4(t, \ldots)\, ,\\
 [\varrho]_{C^{\beta,\frac{\beta}{2}}(Q_t)}  \leq & \Psi_5(t, \, \, \|\nabla \varrho_0\|_{L^p(\Omega)}, \, \mathscr{V}(t; \, v^*)) =: \Psi_5(t, \ldots) \, .
 \end{align*}
Moreover, due to the Lemma \ref{fandg} and due to Lemma \ref{h},
\begin{align*}
 \|g^*\|_{L^p(Q_t)} \leq & \Psi_g\big( t, \,\Psi_3(t, \, \ldots), \, \|(q^0, \, v^0)\|_{W^{2-2/p}_p(\Omega)},\, \Psi_4(t, \, \ldots), \, \mathscr{V}^*(t)\big)  =: \Psi_g(t, \ldots) \, ,\\
 \|h^*\|_{W^{1,0}_p(Q_t)}\leq & \Psi_{h}(t, \, \Psi_3(t, \, \ldots), \, \|q^0\|_{W^{2-2/p}_p(\Omega)}, \, \Psi_4(t, \, \ldots), \, \mathscr{V}^*(t)) \, =: \Psi_h(t, \ldots) \, .
\end{align*}
Combining all these estimates we can bound the quantity $ \mathscr{V}(t; \, q)+ \|\zeta\|_{W^{2,0}_p(Q_t)} $ by some independent constant times the function
\begin{align*}
\Psi_6^{1} & :=  \Psi_1( t, \,  \Psi_3(t, \, \ldots)   ,  \, \|q^0\|_{C^{\beta}(\Omega)}, \, 
\mathscr{V}(t; \, q^*), \,  \Psi_5(t, \,\ldots), \, \Psi_4(t, \, \ldots)) \, \times \\
& \quad \times (1+\Psi_5(t, \, \ldots))^{\tfrac{2}{\beta}} \, ( \|q^0\|_{W^{2-2/p}_p(\Omega)}  + \Psi_g(t, \ldots) + \Psi_h(t, \ldots) + t^{\frac{1}{p}} \, \mathscr{V}(t; \, v^*))\\
& + \Phi( t, \,  \Psi_3(t, \, \ldots)   ,  \, \|q^0\|_{C^{\beta}(\Omega)}, \, 
\mathscr{V}(t; \, q^*), \,  \Psi_5(t, \,\ldots), \, \Psi_4(t, \, \ldots)) \, (\Psi_h(t, \ldots) + t^{\frac{1}{p}} \, \mathscr{V}(t; \, v^*)) \, .
\end{align*}
Applying the inequalities $\mathscr{V}(t; \, v^*), \, \mathscr{V}(t; \, q^*) \leq \mathscr{V}^*(t)$, and $\|q^0\|_{C^{\beta}(\Omega)} \leq c \,\|q^0\|_{W^{2-2/p}_p(\Omega)} $, we reinterpret the latter expression as a function $\Psi^{1}_{6}$ of the arguments $t$, $M_0$,  $\|(q^0, \, v^0)\|_{W^{2-2/p}_p(\Omega)}$, $\|\nabla \varrho_0\|_{L^p(\Omega)}$ and $\mathscr{V}^*(t)$.

At $t = 0$, we can use the estimates proved in the Propositions \ref{qzetasystem}, \ref{vsystem} and the Prop. \ref{contieq}. Recall in particular that $\Psi_1(0, \, a_1, \ldots, a_4) = \Psi_1^0(a_1, \, a_2)$. Moreover, $\Psi_3(0, \, M_0, \, a_4) = M_0$ (cf. \eqref{Psi3}). Thus, since $\Psi_5(0, \, a_1, \, a_4) \leq C \, a_1$ is bounded independently of $a_4$, since $\Psi_g(0, \ldots) = 0 = \Psi_h(0, \, \ldots)$ (see Lemma \ref{fandg}, \ref{h}), we can compute that 
\begin{align}\label{formulatpsi6null}
& \Psi^{1}_6(0, \, M_0,  \,  \|(q^0, \, v^0)\|_{W^{2-2/p}_p(\Omega)}, \, \|\nabla \varrho_0\|_{L^p(\Omega)},\, \mathscr{V}^*(t)) \\
& \quad= \Psi^0_1(M_0, \, \|q^0\|_{C^{1-3/p}(\Omega)}) \, (1+\|\nabla \varrho_0\|_{L^p(\Omega)})^{\tfrac{2p}{p-3}} \, \|q^0\|_{W^{2-2/p}_p(\Omega)} \nonumber\, .
\end{align}
We next apply Proposition \ref{vsystem} with $\varrho^* = \varrho$ and $f = f^*$ to obtain
\begin{align*}
  & \mathscr{V}(t; \, v) \\
   &  \leq  C \, \Psi_2(t, \, \sup_{\tau \leq t} [\varrho(\tau)]_{C^{\alpha}(\Omega)}) \, (1+\sup_{\tau \leq t} [\varrho(\tau)]_{C^{\alpha}(\Omega)})^{\frac{2}{\alpha}}\,  ( \|v^0\|_{W^{2-2/p}_p(\Omega)}+ \|f^*\|_{L^p(Q_t)} + \|\nabla \zeta\|_{L^p(Q_t)})   \, .
\end{align*}
For $\alpha = 1-3/p$, the norm $\mathscr{V}(t; \, v)$ is estimated above by the quantity
\begin{align*}
\Psi_2(t, \, \Psi_5(t, \, \ldots)) \, (1+ \Psi_5(t, \, \ldots))^{\frac{2}{\alpha}}\, (\|v^0\|_{W^{2-2/p}_p(\Omega)}+ \Psi_f(t,\ldots) + \|\nabla \zeta\|_{L^p(Q_t)}  ) \, .  
\end{align*}
Recalling that \eqref{ESTIMqzeta} and the subsequent arguments also provide an estimate for $\|\nabla \zeta\|_{L^p(Q_t)}$ by $\Psi^{1}_6$, we reinterpret the latter function as a $\Psi^{2}_{6}$ of the same arguments, and we note that
\begin{align*}
 & \Psi^{2}_6(0, \,M_0,  \,  \|(q^0, \, v^0)\|_{W^{2-2/p}_p(\Omega)},\, \|\nabla \varrho_0\|_{L^p(\Omega)}, \, \mathscr{V}^*(t)) \\
 & = \Psi^0_2 \times  (1+\|\nabla \varrho_0\|_{L^p(\Omega)})^{\tfrac{2p}{p-3}} \, (\|v^0\|_{W^{2-2/p}_p(\Omega)} + \Psi^{1}_6(0, \,\ldots)) \\
 & = \left( \frac{1}{\min\{1, \, \varrho_{\min}\}} \right)^{\tfrac{2p}{p-3}}\, \left( \frac{\varrho_{\max}}{\varrho_{\min}} \right)^{\tfrac{p+1}{p}} \, (1+\|\nabla \varrho_0\|_{L^p(\Omega)})^{\frac{2p}{p-3}} \, (\|v^0\|_{W^{2-2/p}_p(\Omega)}+ \Psi^{1}_6(0, \, \ldots)) \, .
\end{align*}
The value of $\Psi^{1}_6(0, \, \ldots)$ is given in \eqref{formulatpsi6null}. We define $\Psi_6 := \Psi_6^{1} + \Psi_6^{2}$. Due to Proposition \ref{contieq}, the function $\Psi_3(t, \, M_0, \, \mathscr{V}(t; \, v^*))$ is finite for all arguments satisfying \eqref{RESTRIC}, and therefore $\Psi_6$ is finite under the same condition. The claim follows.
\end{proof}
We sum up the continuity estimates in the following statement.
\begin{prop}\label{Tbounded}
We adopt the assumptions of Theorem \ref{MAIN2}.
Given $(q^*, \, v^*) \in \mathcal{Y}_T$, we define a map $\mathcal{T}(q^*, \, v^*) = (q, \, v)$ via solution to the equations \eqref{linearT1}, \eqref{linearT2}, \eqref{linearT3}, \eqref{linearT4} with homogeneous boundary conditions \eqref{lateralq}, \eqref{lateralv} and initial conditions $(q^0, \, \varrho_0, \, v^0)$. Then, there are $0 < T_0 \leq T$ and $\eta_0 > 0$ depending on the data $R_0 := (M_0, \, \|(q^0, \, v^0)\|_{W^{2-2/p}_p(\Omega)}, \, \|\nabla \varrho_0\|_{L^p(\Omega)})$ such that $\mathcal{T}$ maps the ball with radius $\eta_0$ in $\mathcal{Y}_{T_0}$ into itself.
\end{prop}
\begin{proof}
We define $a_0 > 0$ to be the solution to the equation $ c \, M_0 \,  x \, e^{ c \, x}  = 1$ associated with the numbers in \eqref{RESTRIC}. We apply the Lemma \ref{estimfinale} with $\Psi(t, \, R_0, \, \eta) := \Psi_6(t, \, R_0, \, \eta)$ from Prop. \ref{estimateself}, and the claim follows. 
\end{proof}
\section{Proof of the theorem on short-time well-posedness}\label{FixedPointIter}

\subsection{Existence and uniqueness}

We choose $T_0, \, \eta_0 > 0$ according to Proposition \ref{Tbounded}. 
Starting from $(q^1, \, v^1) = 0$, we consider a fixed point iteration $(q^{n+1}, \, v^{n+1}) := \mathcal{T} \, (q^{n}, \, v^{n})$ for $ n \in \mathbb{N}$. 
Recalling \eqref{Vfunctor}, we define $\mathscr{V}^{n+1}(t) := \mathscr{V}(t; \, q^{n+1}) + \mathscr{V}(t; \, v^{n+1})$. Since obviously $\mathscr{V}^1(t) \equiv 0$, Proposition \ref{Tbounded} guarantees that
\begin{align}\label{uniformestimates}
\sup_{n \in \mathbb{N}} \mathscr{V}^{n}(T_0) \leq \eta_0, \qquad \sup_{n \in \mathbb{N}} \|\varrho^{n} \|_{W^{1,1}_{p,\infty}(Q_{T_0})} + \|\zeta^n\|_{W^{2,0}_{p}(Q_{T_0})} < + \infty \, .
\end{align}
From Lemma \ref{iter} hereafter, we infer that the fixed-point iteration therefore yields strongly convergent subsequences in $L^2(Q_{T_0})$ for the components of $q^n$, $\zeta_n$, $\varrho_n$ and $v^n$ and for the gradients $q^n_x, \, \zeta_{x}^n$ and $v^n_x$. The passage to the limit in the approximation scheme is then a straightforward exercise, since we can rely on a uniform bound in $\mathcal{X}_{T_0}$.
The proofs are almost identical with the fixed-point iteration in \cite{bothedruet}. We leave the minor changes to the interested reader, and state without proof the following iteration lemma.
\begin{lemma}\label{iter}
For $n \in \mathbb{N}$, we define
\begin{align*}
 r^{n+1} & :=  q^{n+1} - q^n, \quad \chi^{n+1} := \zeta^{n+1} - \zeta^n, \,     \quad \sigma^{n+1} := \varrho^{n+1} - \varrho^{n}, \quad w^{n+1} := v^{n+1} - v^n \, .
 \end{align*}
 Then there are $k_0, \, p_0 > 0$ and $0 < t_1 \leq T_0$ such that for all $t \in [0, \, T_0 -t_1]$, the quantity
\begin{align*} 
 E^{n+1}(t) := & k_0 \, \sup_{\tau \in [t, \, t+t_1]}  (\|r^{n+1}(\tau)\|_{L^2(\Omega)}^2 + \|w^{n+1}(\tau)\|_{L^2(\Omega)}^2 + \|\sigma_{n+1}\|_{L^2(\Omega)}^2) \\
 &+ p_0 \, \int_{Q_{t,t+t_1}}  (|\nabla r^{n+1}|^2 +|\nabla \chi^{n+1}|^2 + |\nabla w^{n+1}|^2)  \, dxd\tau
 \end{align*}
satisfies $ E^{n+1}(t) \leq \frac{1}{2} \, E^{n}(t)$ for all $n \in \mathbb{N}$.
\end{lemma}

\subsection{Verification of continuation criteria}

In order to complete the proof of the Theorems \ref{MAIN}, \ref{MAINPr} it remains to investigate the claimed characterisations of the maximal existence interval.
\begin{lemma}\label{MAXEX}
 Suppose that $u =(q, \, \zeta,\, \varrho,\, v) \in \mathcal{X}_{t}$ is a solution to $\widetilde{\mathscr{A}}(u) = 0$ and $u(0) = u_0$ for all $0 < t < T^*$. Then the two following statements are valid:
 \begin{enumerate}[(1)]
\item \label{firstcrit} If $\mathscr{N}(t) := \|q\|_{C^{\alpha,\alpha/2}(Q_{t})} + \|\nabla q\|_{L^{\infty,p}(Q_{t})} + \|v\|_{L^{z\, p, \,p}(Q_{t})} + \int_{0}^{t} [\nabla v(\tau)]_{C^{\alpha}(\Omega)} \, d\tau$ with $\alpha > 0$ arbitrary and $z = z(p)$ defined in Theorem \ref{MAIN}, and $M(\varrho, \, t)$ (cf. \eqref{supremumrho}) are finite for $t \nearrow T^*$, then it is possible to extend the solution to a larger time interval. 
 
\item \label{secondcrit} If the tensor $M$ occurring in \eqref{DIFFUSFLUX} satisfies the additional conditions stated in Theorem \ref{MAINPr}, and if $\mathcal{K}(t) :=
  \|q\|_{W^{2,1}_p(Q_t; \, \mathbb{R}^{N-2})} + \|\zeta\|_{W^{2,0}_p(Q_t)} + \|v\|_{ W^{2,1}_p(Q_t; \, \mathbb{R}^{3})} $ remains finite for $t \nearrow T^*$, then the solution can be extended without additional condition concerning $M(\varrho, \, t)$.
\end{enumerate}
\end{lemma}
\begin{proof}
\textbf{First criterion \eqref{firstcrit}.} We must show that the quantity $ \mathscr{V}(t; \, q) +  \mathscr{V}(t; \, v)$ is bounded by a continuous function of $t, \, M(\varrho, \, t), \, \mathcal{N}(t)$. We will only sketch this point, which relies on going carefully through the proofs of the estimates in the Propositions \ref{qzetasystem}, \ref{vsystem} in the spirit of \cite{bothedruet}. 

To begin with, we notice that the components of $v_x$ have all spatial mean-value zero over $\Omega$ due to the boundary condition \eqref{lateralv}. Hence, for $\alpha >0$, inequalities $\|v_x(\tau)\|_{L^{\infty}(\Omega)} \leq c_{\Omega} \, [v_x(\tau)]_{C^{\alpha}(\Omega)}$ and $\|v_x\|_{L^{\infty,1}(Q_t)} \leq c_{\Omega} \, \int_0^t [ v_x(\tau)]_{C^{\alpha}(\Omega)} \, d\tau$ are available. For the solution to the continuity equation, Theorem 2 of \cite{solocompress} (see also Proposition 7.7 in \cite{bothedruet}) implies that $\sup_{\tau \leq t} [\varrho(\tau)]_{C^{\alpha}(\Omega)}$ is bounded by a function of $\int_0^t [ v_x(\tau)]_{C^{\alpha}(\Omega)} \, d\tau$, thus also by a function of $\mathcal{N}(t)$. Moreover, as in the same references, we show for all $t \geq 0$ that
\begin{align*}
\|\varrho_x(t)\|_{L^p(\Omega)} \leq & \phi(R_0, \, \|v_x\|_{L^{\infty,1}(Q_{t})}) \, (1+ \int_{0}^t\|v_{x,x}(\tau)\|_{L^{p}(\Omega)} \, d\tau) \\
\leq & \phi(R_0, \, \mathcal{N}(t)) \, (1+ \mathscr{V}(t; \, v)) \, .
\end{align*}
Here and throughout the proof, we denote by $\phi$ some generic continuous function increasing in its arguments, and $R_0$ stands for the initial data and the external forces.

We next exploit the momentum balance equation for $v$. We apply Proposition \ref{vsystem}, hence $\mathscr{V}(t; \, v) \leq  \phi(t, \, \mathcal{N}(t)) \, (\|f\|_{L^p(Q_t)} +\|\nabla \zeta\|_{L^p(Q_t)} + \|v^0\|_{W^{2-2/p}_p(\Omega)})$. The function $f$ obeys \eqref{fri} and therefore
\begin{align*}
 |f(x, \, t)| \leq & |\nabla \varrho(x, \, t)| \, \sup_{Q_t} |P_{\varrho}(\varrho, \, q)| +|\nabla q(x, \, t)| \,  \sup_{Q_t} |P_{q}(\varrho, \, q)| \\
 & + c \, (|v(x, \, t)| \, |\nabla v(x, \, t)| + |\bar{b}(x, \, t)| + |\tilde{b}(x, \, t)|) \,  \sup_{Q_t}  \varrho + |\hat{b}(x, \, t)| \, .
\end{align*}
Coefficients depending on $\varrho$ and $q$ can in general be bounded following the example of $$\sup_{Q_t} |P_{\varrho}(\varrho, \, q)| \leq \phi(M(\varrho, \, t), \, \|q\|_{L^{\infty}(Q_t)}) \leq \phi(M(\varrho, \, t), \, \mathcal{N}(t)) \, .$$ 

\pagebreak

Therefore, we show that
\begin{align*}
 \|f\|_{L^p(Q_t)}^p \leq \phi(M(\varrho, \, t), \, \mathcal{N}(t)) \, (& \|\nabla \varrho\|_{L^p(Q_t)}^p + \|\nabla q\|_{L^p(Q_t)}^p + \|v \, \nabla v\|_{L^p(Q_t)}^p \\
 & + \|\tilde{b}\|_{L^p(Q_t)}^p + \|\bar{b}\|_{L^p(Q_t)}^p + \|\hat{b}\|_{L^p(Q_t)}^p) \, . 
\end{align*}
We define $A_0(t) :=  \|\tilde{b}\|_{L^p(Q_t)}^p + \|\bar{b}\|_{L^p(Q_t)}^p + \|\hat{b}\|_{L^p(Q_t)}^p  + \|v^0\|_{W^{2-2/p}_p(\Omega)}$, hence
\begin{align*}
 &\mathscr{V}^p(t; \, v) \leq\\
 & \quad \phi(M(\varrho, \, t), \, \mathcal{N}(t)) \, ( \|\nabla \zeta\|_{L^p(Q_t)}^p + \|\nabla \varrho\|_{L^p(Q_t)}^p +  \|v \, \nabla v\|_{L^p(Q_t)}^p +   \|\nabla q\|_{L^p(Q_t)}^p + A_0(t))  \, .
\end{align*}
As shown, $\|\nabla \varrho\|_{L^p(Q_t)}^p \leq \phi(R_0, \, \|v_x\|_{L^{\infty,1}(Q_t)})  \, \int_0^t (1+ \mathscr{V}(\tau; \, v))^p \, d\tau$, and $\zeta$ satisfies the weak Neumann problem \eqref{linearT3}, hence 
\begin{align*}
 \|\nabla \zeta\|_{L^p(Q_t)} \leq \phi(M(\varrho, \, t), \, \mathcal{N}(t)) \, (\|\nabla q\|_{L^p(Q_t)} + \|v\|_{L^p(Q_t)} + \|\tilde{b}\|_{L^p(Q_t)} + \|\hat{b}\|_{L^p(Q_t)}) \, .
\end{align*}
We define $z = \frac{3}{p-2}$ if $3 < p < 5$, $z > 1$ arbitrary if $p = 5$ and $z = 1$ if $p > 5$. Recalling the continuity of the embedding $W^{1-2/p}_p \subset L^{3p/(5-p)^+}$, we show by means of H\"older's inequality that $\|v \,  v_x\|_{L^p(Q_t)}^p \leq c_{\Omega}\,  \int_0^t \|v(\tau)\|^p_{L^{z\, p}} \, \mathscr{V}^p(\tau; \, v) \, d\tau$. Therefore, combining the latter bounds yields
\begin{align*}
 \mathscr{V}^p(t; \, v) \leq  \phi(t, \, M(\varrho, \, t),\, \mathcal{N}(t)) \, \big(\int_0^t (1+\|v(\tau)\|^p_{L^{z\, p}}) \, \mathscr{V}^p(\tau; \, v) \, d\tau +   \|\nabla q\|_{L^p(Q_t)}^p + A_0(t)\big)  \, .
\end{align*}
We invoke the Gronwall Lemma, hence $\mathscr{V}^p(t; \, v) \leq \phi(t, \, M(\varrho, \, t),\, \mathcal{N}(t))  \,  (\|\nabla q\|_{L^p(Q_t)}^p + A_0(t))$. Since $\|\nabla q\|_{L^p(Q_t)}$ is also controlled by a function of $t$ and $\mathcal{N}(t)$, so does $\mathscr{V}^p(t; \, v)$. It follows that $\|\nabla \varrho\|_{L^{p,\infty}(Q_t)}^p \leq \phi(t, \, R_0, \,\mathcal{N}(t)) $. For $\beta = 1-3/p$, the Proposition \ref{contieq} yields that $\|\varrho\|_{C^{\beta, \beta/2}(Q_t)} \leq \phi(t, \, R_0, \,\mathcal{N}(t))$. Recalling that $q$ satisfies \eqref{parabolicsystemreducedpr}, we can now finish the proof as in \cite{bothedruet}, Lemma 9.2.

\textbf{Second criterion \eqref{secondcrit}.} The more interesting point is to get rid of the dependence on the distance $M(\varrho, \, t)$ to the density thresholds in the estimates. First we note that the relation \eqref{momentumprime} implies for the gradient of the pressure
\begin{align}\label{spacegradient}
\nabla P(\varrho, \, q) = F := &  -\nabla \zeta -\varrho \, (\partial_t v + (v \cdot \nabla)v) + \divv \mathbb{S}(\nabla v) \nonumber\\
&+ R(\varrho, \, q) \cdot \tilde{b} + \hat{b} + \varrho \, \bar{b} \, . 
\end{align}
Clearly, $\|F\|_{L^p(Q_t)}$ is bounded by a function of $b$ and the norms of $\zeta$ and $v$ occurring in the quantity $\mathcal{K}(t)$. We notice in particular that this function is independent on $M(\varrho, \, t)$.

In order to obtain a bound on the entire pressure gradient, we employ the continuity equation \eqref{massprime3}. We compute that
\begin{align*}
 \partial_t P(\varrho, \, q) = &  P_{\varrho}(\varrho, \, q) \,  \partial_t \varrho  + P_{q}(\varrho, \, q) \, \partial_t q\\
 = & P_{\varrho}(\varrho, \, q) \, (-v\cdot \nabla \varrho - \varrho \, \divv v) + P_{q}(\varrho, \, q) \, \partial_t q \, .
\end{align*}
Define $m(\varrho, \, t)  := \min_{Q_t}\{1-\varrho/\varrho_{\max}, \, \varrho/\varrho_{\min}-1\}$. Thanks to Lemma \ref{special}, the properties of the pressure function guarantee that $|P_{\varrho}(\varrho, \, q)| \leq c_4 \, m(\varrho, \, t)^{-1}$ in $Q_t$. Since $|P_{\varrho}(\varrho, \, q)| \, |\nabla \varrho| = | \nabla P(\varrho, \, q) - P_q(\varrho, \, q) \, \nabla q|$, the same Lemma \ref{special} also implies that $$c_3 \, m(\varrho, \, t)^{-1} \, |\nabla \varrho| \leq |\nabla P(\varrho, \, q)| + c_5 \, |\nabla q|) \, .$$ By these means, the time derivative of pressure is bounded via
\begin{align}\label{timegrad}
 |\partial_t P(\varrho, \, q) | \leq & c_4 \,  m^{-1} \, ( |v| \, |\nabla \varrho| + |\varrho| \, |\divv v|) + |P_q| \, |\partial_t q| \nonumber\\
  \leq & c_4 \, \frac{1+c_5}{c_3} \,  |v| \,(|\nabla P(\varrho, \, q)| + |\nabla q|) +c_4 \,   \varrho_{\max} \, m^{-1} \, |\divv v| + c_5 \, |\partial_t q| \, . 
\end{align}
We want to obtain a control on $m(\varrho, \, t)^{-1} \, |\divv v| $. To this aim, we recall the relation \eqref{massprime2}, which allows us to compute
\begin{align}\label{divergence}
& \qquad \divv v = \divv( d(\varrho, \, q) \, (\nabla \zeta-\hat{b}) + A(\varrho, \, q) \, (\nabla q - \tilde{b}))\\
= &  d(\varrho, \, q) \, \divv  (\nabla \zeta-\hat{b}) + A(\varrho, \, q) \, \divv(\nabla q - \tilde{b})  + (\nabla \zeta-\hat{b}) \cdot \nabla d(\varrho, \, q)  + \nabla A(\varrho, \, q)  \cdot (\nabla q - \tilde{b}) \, \nonumber\\
= & d \, \divv  (\nabla \zeta-\hat{b}) + A \, \divv(\nabla q - \tilde{b}) + [(\nabla \zeta-\hat{b}) \, d_q + (\nabla q - \tilde{b}) \, A_q] \, \nabla q\nonumber\\
& + \nabla \varrho \, [d_{\varrho} \, (\nabla \zeta-\hat{b}) + A_{\varrho} \, (\nabla q - \tilde{b})]\nonumber \, .
\end{align}
We recall Lemma \ref{special}, which shows for a constant $c_1$, depending only on the kinetic matrix $M$ and the free energy function $k$, that
\begin{align*}
 |d(\varrho, \, q)| + |A(\varrho, \, q)| + |d_q(\varrho, \, q)| + |A_q(\varrho, \, q)| & \leq c_1 \, m(\varrho) \, ,
\end{align*}
and, moreover, that $| d_{\varrho}(\varrho, \, q)| + |A_{\varrho}(\varrho, \, q)| \leq c_2$. Applying these estimates to \eqref{divergence}, we obtain that
\begin{align*}
 \frac{1}{m(\varrho, \, t)} \, |\divv v| \leq& c_1 \, \underbrace{[|D^2\zeta| + |D^2q| + |\tilde{b}_x| +|\hat{b}_x| + |\nabla q| \, (|\nabla \zeta| + |\nabla q| + |\tilde{b}|+|\hat{b}|)]}_{=:G}\\
 & + c_2 \, \frac{|\nabla \varrho| }{m(\varrho, \, t)} \, (|\nabla \zeta| + |\nabla q| + |\tilde{b}|+|\hat{b}|) \, .
\end{align*}
Recalling again that $m(\varrho)^{-1} \, |\nabla \varrho| \leq c \, (|\nabla P(\varrho, \, q)| + |\nabla q|)$, we get the bound
\begin{align*}
  \frac{1}{m(\varrho, \, t)} \, |\divv v| \leq |G| + c \, (|\nabla P(\varrho, \, q)| + |\nabla q|) \, (|\nabla \zeta| + |\nabla q| + |\tilde{b}|+|\hat{b}|) \,.
\end{align*}
It is readily verified that $G$ is continuously bounded in $L^p(Q_t)$ by the quantity $\mathcal{K}(t)$, independently of $M(\varrho, \, t)$. Since $|\zeta_x| + |q_x| + |\tilde{b}|+|\hat{b}|$ is bounded in $L^{\infty,p}(Q_t)$, we recall \eqref{spacegradient} to finally obtain
\begin{align}\label{weighteddiv}
 \|m(\varrho, \, t)^{-1} \, \divv v\|_{L^{p,\frac{p}{2}}(Q_t)} \leq \Psi(t, \, \mathcal{K}(t)) \, .
\end{align}
By means of \eqref{weighteddiv}, \eqref{spacegradient}, \eqref{timegrad} we see that also $\|\partial_t P(\varrho, \, q)\|_{L^{p,p/2}(Q)}$ is bounded by a function of $ t$ and $\mathcal{K}(t)$, independently of $M(\varrho, \, t)$. Overall we have $\|P_x\|_{L^p(Q_t)} + \|P_t\|_{L^{p,p/2}(Q_t)} \leq \Psi$. For $p > 5$, we can show that this implies a bound $\|P\|_{L^{\infty}(Q_t)} \leq C(t) \, \Psi$, where $C(t)$ is the embedding constant of an anisotropic Sobolev space into $L^{\infty}(Q_t)$.
It remains to recall that for the choice \eqref{kfunktion}, the function $P$ satisfies (cf. \cite{druetmixtureincompweak}, Proposition 5.3)
\begin{align*}
| P(\varrho, \, q)| \geq c \, \ln \max\{\frac{1}{\varrho_{\max}-\varrho}, \, \frac{1}{\varrho-\varrho_{\min}}\} - C \, (1 +|q|) \, .
\end{align*}
This implies that $M(\varrho, \, t) \leq C_1 \, e^{C_2 \, (\|P(\varrho, \, q)\|_{L^{\infty}(Q_t)} + \|q\|_{L^{\infty}(Q_t)})}$, and the claim follows.
\end{proof}

\section{Global well-posedness}\label{contiT1}

\subsection{The map $\mathcal{T}^1$ is well defined}\label{welldef}

We consider the equations \eqref{equadiff0}, \eqref{equadiff1}, \eqref{equadiff2}, \eqref{equadiff3}, \eqref{equadiff4} characteristic of the definition of the map $\mathcal{T}^1$. We recall that these equations are obtained by comparing a solution to some suitable extension $(\hat{q}^0, \, \hat{v}^0) \in \mathcal{Y}_T$, to be constructed here below, of the initial data. The initial density $\varrho_0$ is extended by a function $\hat{\varrho}^0$ obtained via the solution of \eqref{Extendrho0}. We moreover introduce the function $\hat{\zeta}^0$, solution to \eqref{Extendzeta0}.

In order to define $\mathcal{T}^1$ we must make sense of the linear operators $(g^1)^{\prime}(u^*, \, \hat{u}^0)$, $(h^1)^{\prime}(u^*, \, \hat{u}^0)$ and $(f^1)^{\prime}(u^*, \, \hat{u}^0)$. The density components in the vectors $\hat{u}^0 = (\hat{q}^0, \, 1, \, \hat{\varrho}^0, \, \hat{v}^0)$ and $u^*$ (def. in \eqref{ustar}) must therefore assume values in $I$ \emph{up to} time $T > 0$!
This property is to be expected if the initial data are close enough to an equilibrium solution $(\rho^{\text{eq}}, \, p^{\text{eq}}, \, v^{\text{eq}})$ defined by the relations \eqref{massstat}, \eqref{momentumsstat}. The distance of the initial data to this solution is expressed by the number
\begin{align}\label{R1}
R_1 := \|q^0 - q^{\text{eq}}\|_{W^{2-2/p}_p(\Omega; \, \mathbb{R}^{N-2})} + \|v^0 - v^{\text{eq}}\|_{W^{2-2/p}_p(\Omega; \, \mathbb{R}^3)} + \|\varrho^0-\varrho^{\text{eq}}\|_{W^{1,p}(\Omega)}\, ,
\end{align}
in which $\varrho^{\text{eq}} := \sum_{i=1}^N \rho^{\text{eq}}_i$ and $q^{\text{eq}}_{\ell} = \eta^{\ell} \cdot \nabla_{\rho}k(\rho^{\text{eq}})$ for $\ell =1,\ldots,N-2$. Throughout this section, we moreover employ the abbreviation
\begin{align}\label{R0}
 R_0 := & \|\hat{u}^0\|_{\mathcal{X}_T} + \|\hat{\varrho}^0\|_{W^{2,0}_p(Q_T)} + \|\tilde{b}\|_{W^{1,0}_p(Q_T)} + \|\hat{b}\|_{W^{1,0}_p(Q_T)} +\|\bar{b}\|_{L^p(Q_T)}) \, .
 \end{align}
Observe the occurrence of the higher--order $W^{2,0}_p$--norm of $\hat{\varrho}^0$ in the definition of $R_0$. 

To commence with, we recall a result of \cite{bothedruet} for estimating the gradient of solutions to a perturbed continuity equation. 
The proof in \cite{bothedruet} is given for zero initial conditions, but the extension to the nonzero case is completely straightforward.
\begin{lemma}\label{PerturbConti}
Assume that $\sigma \in W^{1,1}_{p,\infty}(Q_T)$ satisfies $\partial_t \sigma +\divv(\sigma \, v) = - \divv(\hat{\varrho}^0 \, w)$ with $\hat{\varrho}^0 \in W^{1,1}_{p,\infty}(Q_T) \cap W^{2,0}_p(Q_T)$ and $v, \,w \in W^{2,1}_p(Q_T; \, \mathbb{R}^3)$. Then there are constants $C, \, c$, depending only on $\Omega$, such that 
\begin{align*}
 \|\sigma(t)\|_{W^{1,p}(\Omega)}^p \leq & C \, \exp\big(c \, \int_0^t [\|v_x(\tau)\|_{L^{\infty}(\Omega)} + \|v_{x,x}(\tau)\|_{L^p(\Omega)} + 1]  d\tau\big) \, \times \\
 & \times  (\|\sigma(0)\|_{W^{1,p}(\Omega)}^p + \|\hat{\varrho}^0\|_{W^{2,0}_p(Q_t)}^p \, \|w\|_{L^{\infty}(Q_t)}^p + \|\hat{\varrho}^0\|_{W^{1,1}_{p,\infty}(Q_t)}^p \,  \|w\|_{W^{2,0}_p(Q_t)}^p) \, 
\end{align*}
for all $t \leq T$.
\end{lemma}
\textbf{Construction of global extensions.}
Under the assumptions of Theorem \ref{MAIN3}, the trivial extensions $q^{\text{eq}}(x, \, t) := q^{\text{eq}}(x)$ and $v^{\text{eq}}(x, \, t) := v^{\text{eq}}(x)$ are such that $q^{\text{eq}} \in W^{2,\infty}_{p,\infty}$ and $v^{\text{eq}} \in W^{3,\infty}_{p,\infty}(Q_T)$. Introduce on $\Omega$ the differences $q^1(x) := q^0(x) - q^{\text{eq}}(x)$ and $v^1(x) := v^0(x) - v^{\text{eq}}(x)$. We extend $q^0$ and $v^0$ via
\begin{align}\label{dextend}
 \hat{v}^0(x, \, t) := v^{\text{eq}}(x) + \underbrace{\mathcal{E}(v^1)(x, \, t)}_{=: \hat{v}^1(x, \, t)}, \quad  \hat{q}^0(x, \, t) := q^{\text{eq}}(x) +\underbrace{ \mathcal{E}(q^1)(x, \, t)}_{=: \hat{q}^1(x, \, t)} \, , 
\end{align}
in which $\mathcal{E}: \, W^{2-2/p}_p(\Omega) \rightarrow W^{2,1}_p(Q_T)$ is a linear, bounded extension operator. Typically, the components of $q^1, \, v^1$ defined in $\Omega$ are first extended to elements of $W^{2-2/p}_{p}(\mathbb{R}^3)$ with bounded support. Then, we solve Cauchy-problems for the heat equation to extend the functions into $\mathbb{R}^3 \times [0, \, T]$ or even $\mathbb{R}^4$. As the assumptions in Theorem \ref{MAIN3} moreover guarantee that $v^1 \in W^{2,p}(\Omega)$, this procedure even yields the additional regularity $\hat{v}^1 \in W^{4,2}_p(Q_T; \, \mathbb{R}^3)$ (cf. \cite{ladu}, Ch. 4, Par. 3, inequality (3.3)). 

Then, the extensions defined in \eqref{dextend} satisfy 
\begin{align}\label{distance}
 \|\hat{q}^0 - q^{\text{eq}}\|_{W^{2,1}_p(Q_T)} + \|\hat{v}^0 - v^{\text{eq}}\|_{W^{2,1}_p(Q_T)} \leq & C_{\mathcal{E}} \, (\|q^1\|_{W^{2-2/p}_p(\Omega)} +  \|v^1\|_{W^{2-2/p}_p(\Omega)})\nonumber\\
\leq & C_{\mathcal{E}} \, R_1 \, ,\\
 \label{extendbetter}
 \|\hat{v}^0\|_{W^{3,0}_p(Q_T)} \leq & C \, (\|v^{\text{eq}}\|_{W^{3,p}(\Omega)} + \|v^0\|_{W^{2,p}(\Omega)}) \, .
\end{align}
In order to extend $\varrho^0$, we solve $\partial_t \hat{\varrho}^0 + \divv( \hat{\varrho}^0 \, \hat{v}^0) = 0$ with initial condition $\hat{\varrho}^0 = \varrho^0$. By these means, $\hat{\varrho}^0 \in W^{1,1}_{p,\infty}(Q_T)$. Due to \eqref{extendbetter}, we can even show that $\hat{\varrho}^0 \in W^{2,0}_p(Q_T)$. We next extend the equilibrium solution via $\varrho^{\text{eq}}(x, \, t) := \varrho^{\text{eq}}(x) \in W^{2,\infty}_{p,\infty}(Q_T)$. Then, by definition, $\divv (\hat{\varrho}^{\text{eq}} \, \hat{v}^{\text{eq}}) = 0$ in $Q_T$ (cp. \eqref{massstat}), and $\partial_t \hat{\varrho}^{\text{eq}} = 0$. Thus, the difference $\hat{\varrho}^1 := \hat{\varrho}^0 - \varrho^{\text{eq}}$ is a solution to
\begin{align*}
 \partial_t \hat{\varrho}^1 + \divv (\hat{\varrho}^1 \, \hat{v}^0) = - \divv(\hat{\varrho}^0 \, \hat{v}^{1}) \, , \quad \hat{\varrho}^1(x, \, 0) = \varrho^1(x) := \varrho^0(x) - \varrho^{\text{eq}}(x) \, .
\end{align*}
Since $\hat{\varrho}^0 \in W^{1,1}_{p,\infty}(Q_T) \cap W^{2,0}_p(Q_T)$ by construction, the estimate of Lemma \ref{PerturbConti} applies (with the choices $\sigma = \hat{\varrho}^1$, $v = \hat{v}^0$ and $w := \hat{v}^1$). Hence, invoking also \eqref{distance},
\begin{align*}
 \|\hat{\varrho}^1\|_{W^{1,1}_{p,\infty}(Q_T)}^p \leq & C \, \exp\big(c \, \int_0^{T} [\|\hat{v}^0_x(\tau)\|_{L^{\infty}(\Omega)} + \|\hat{v}^0_{x,x}(\tau)\|_{L^p(\Omega)} + 1]  d\tau\big) \times\\
 & \times
  [\|\varrho^1\|_{W^{1,p}(\Omega)}^p + \|\hat{\varrho}^0\|_{W^{2,0}_p(Q_T)}^p \, \|\hat{v}^1\|_{L^{\infty}(Q_T)}^p + \|\hat{\varrho}^0\|_{W^{1,1}_{p,\infty}(Q_T)}^p \,  \|\hat{v}^1\|_{W^{2,0}_p(Q_T)}^p]\\
   \leq & C_T \, (\|\varrho^1\|_{W^{1,p}(\Omega)}^p +\|\hat{v}^1\|_{W^{2,1}_p(Q_T)}^p) \leq C(R_0,T) \, \, R_1^p \, .
\end{align*}
The latter and \eqref{distance} now entail
\begin{align}\label{distanceglobal}
 \|\hat{q}^0 - q^{\text{eq}}\|_{W^{2,1}_p(Q_T)} + \|\hat{v}^0 - v^{\text{eq}}\|_{W^{2,1}_p(Q_T)} + \|\hat{\varrho}^0-\varrho^{\text{eq}}\|_{W^{1,1}_{p,\infty}(Q_T)} \leq C \, R_1 \, .
\end{align}
Thus, it also follows that $\|\hat{\varrho}^0-\varrho^{\text{eq}}\|_{L^{\infty}(Q_T)} \leq C \, R_1$. Therefore
\begin{align*}
 \varrho_{\max} - \hat{\varrho}^0(x, \, t) \geq  \varrho_{\max} - \varrho^{\text{eq}}(x) - C \, R_1, \quad \hat{\varrho}^0(x, \, t) -\varrho_{\min} \geq \varrho^{\text{eq}}(x) - \varrho_{\min} - C \, R_1 \, .
\end{align*}
By definition, the equilibrium density remains in the thresholds, that is, $M(\varrho^{\text{eq}}, \, 0) < + \infty$ (see \eqref{supremumrho}). If $R_1$ is small enough, for instance if it satisfies the condition
\begin{align}\label{smalldata}
R_1 \leq \frac{1}{2C} \, \min_{x \in \overline{\Omega}}\{\varrho_{\max} - \varrho^{\text{eq}}(x), \, \varrho^{\text{eq}}(x)-\varrho_{\min}\} \, , 
\end{align}
we can show that
\begin{align}\label{globaldrin}
 M(\hat{\varrho}^0, \, T)  = \text{esssup}_{Q_T} \max\{\frac{1}{\varrho_{\max} - \hat{\varrho}^0}, \, \frac{1}{ \hat{\varrho}^0 -\varrho_{\min}}\} \leq 2 \, M(\varrho^{\text{eq}}, \, 0) < + \infty \, .
\end{align}
We define $\zeta^{\text{eq}}(x) := \eta^{N-1} \cdot \nabla_{\rho}k(\rho^{\text{eq}}(x))$. Multiplying \eqref{massstat} with $\bar{V}$, we see that $\zeta^{\text{eq}}$ satisfies
\begin{align*}
 \divv(v^{\text{eq}} - d(\varrho^{\text{eq}}, \, q^{\text{eq}}) \, (\nabla \zeta^{\text{eq}} - \hat{b}(x)) + A(\varrho^{\text{eq}}, \, q^{\text{eq}}) \, (\nabla q^{\text{eq}}- \tilde{b}(x))) = 0 \, .
\end{align*}
Since $\hat{\zeta}^0$ is constructed solving \eqref{Extendzeta0}, the difference $y := \hat{\zeta}^0 - \zeta^{\text{eq}}$ satisfies
\begin{align*}
- \divv(d^0 \, \nabla y) = & - \divv( \hat{v}^{0} - v^{\text{eq}} + (d^0 - d^{\text{eq}}) \, (\hat{b}(x) - \nabla \zeta^{\text{eq}}))\\
 & - \divv( A^{\text{eq}} \cdot (\nabla q^{\text{eq}}-\tilde{b}(x)) -  A^{0} \, (\nabla \hat{q}^{0}-\tilde{b}(x))) \, 
\end{align*}
where zero superscript of a coefficient means evaluation at $(\hat{\varrho}^0, \, \hat{q}^0)$, while \emph{eq} superscript means evaluation at $(\hat{\varrho}^{\text{eq}}, \, \hat{q}^{\text{eq}})$. Thus, elementary calculations show that also
\begin{align}\label{distanceglobal2}
\| \hat{\zeta}^0- \zeta^{\text{eq}}\|_{W^{2,0}_p(Q_T)} & \leq  C \, (\|\hat{q}^0 - q^{\text{eq}}\|_{W^{2,1}_p(Q_T)} + \|\hat{v}^0 - v^{\text{eq}}\|_{W^{2,1}_p(Q_T)} + \|\hat{\varrho}^0-\varrho^{\text{eq}}\|_{W^{1,1}_{p,\infty}(Q_T)})\nonumber \\
& \leq C \, R_1 \, .
\end{align}
\textbf{The nonlinear map.} Consider now $(r^*, \, w^*)$ given in $\phantom{}_0{\mathcal{Y}}_T$. We define $q^* := \hat{q}^0 + r^*$ and $v^* = \hat{v}^0+ w^*$. Following \eqref{equadiff0}, we introduce $\varrho^* := \mathscr{C}(v^*)$. Then, the difference $\sigma^* := \varrho^* - \hat{\varrho}^0$ is a solution to
\begin{align*}
 \partial_t \sigma^* + \divv(\sigma^* \, v^*) = - \divv(\hat{\varrho}^0 \, w^*) \, , \quad \sigma^*(x, \, 0) = 0 \, .
\end{align*}
Making use of Lemma \ref{PerturbConti} ($\sigma^* = \sigma$ and $v^* = v$, $w^* = w$ therein), we get
\begin{align*}
  \|\varrho^* - \hat{\varrho}^0\|_{W^{1,1}_{p,\infty}(Q_T)}^p \leq & C \, \exp\big(c \, \int_0^{ T} [\|v^*_x(\tau)\|_{L^{\infty}(\Omega)} + \|v^*_{x,x}(\tau)\|_{L^p(\Omega)} + 1]  d\tau\big) \, \times \\
 & \times  (\|\hat{\varrho}^0\|_{W^{2,0}_p(Q_T)}^p \, \|w^*\|_{L^{\infty}(Q_T)}^p + \|\hat{\varrho}^0\|_{W^{1,1}_{p,\infty}(Q_T)}^p \,  \|w^*\|_{W^{2,0}_p(Q_T)}^p)\\
 \leq & \phi_0(T, \, R_0, \, \|w^*\|_{W^{2,1}_p(Q_T)}) \, \|w^*\|_{W^{2,1}_p(Q_T)}^p \, ,
\end{align*}
with a certain continuous function $\phi_0$ increasing of its arguments. Hence, use of the continuous embedding $W^{1,1}_{p,\infty} \subset L^{\infty}$ yields $  \|\varrho^* - \hat{\varrho}^0\|_{L^{\infty}(Q_T)} \leq \phi_0(T, \, R_0, \, \|w^*\|_{W^{2,1}_p(Q_T)}) \, \|w^*\|_{W^{2,1}_p(Q_T)}$. We recall \eqref{globaldrin} to show that, under the condition
\begin{align}\label{smalldata2}
 \phi_0(T, \, R_0, \, \mathscr{V}(T; \, w^*)) \,\mathscr{V}(T; \, w^*)  \leq \frac{1}{4} \, \min_{x \in \overline{\Omega}}\{\varrho_{\max} - \varrho^{\text{eq}}(x), \, \varrho^{\text{eq}}(x)-\varrho_{\min}\} =: a_0 \, ,
\end{align}
we can guarantee that $M(\varrho^*, \, T) < + \infty$ globally. 

The vector $\hat{u}_0 := (\hat{q}^0, \, 1, \, \hat{\varrho}^0, \, \hat{v}^0)$ is in $\mathcal{X}_{T,I}$ under the condition \eqref{smalldata}. Given $(r^*, \, w^*) \in \phantom{}_{0}{\mathcal{Y}}_T$ satisfying \eqref{smalldata2}, we define $u^* := (\hat{q}^0 + r^*, \, 1, \,  \mathscr{C}(\hat{v}^0+w^*), \, \hat{v}^0+w^*)$ (cp. \eqref{ustar}), and we see by the latter arguments that $u^* \in \mathcal{X}_{T,I}$ too. Thus, we can make sense of the operators $(g^1)^{\prime}(u^*, \, \hat{u}^0)$, $(h^1)^{\prime}(u^*, \, \hat{u}^0)$ and $(f^1)^{\prime}(u^*, \, \hat{u}^0)$ in the right-hand of the equations \eqref{equadiff1}, \eqref{equadiff2}, \eqref{equadiff4} on the entire interval $[0, \, T]$.

If we can solve the linear system \eqref{equadiff1}, \eqref{equadiff2}, \eqref{equadiff3}, \eqref{equadiff4} for $(r, \, \chi, \, \zeta,\, w)$, we obtain a \emph{globally defined} solution in $\phantom{}_{0}{\mathcal{X}}_T$, and we can meaningfully define $\mathcal{T}^1(r^*, \, w^*) := (r, \, w)$. We shall prove the solvability by linear continuation on the base of the continuity estimates, that we are in the position to prove next.

\subsection{Continuity estimates}

We need at first an estimate for the operators $(g^1)^{\prime}$, $(h^1)^{\prime}$ and $(f^1)^{\prime}$. We shall prove it for general body forces $b = b(x, \, t)$, even if the statement of Theorem \ref{MAIN3} requires only $b = b(x)$.
\begin{lemma}\label{RightHandControlsecond}
Assume that the initial data satisfy \eqref{smalldata}. Consider $\hat{u}^0 := (\hat{q}^0, \, 1, \, \hat{\varrho}^0, \, \hat{v}^0) \in \mathcal{X}_{T,I}$ with $\hat{\varrho}^0 \in W^{2,0}_p(Q_T)$ constructed in the Section \ref{welldef}. For a given $(r^*, \, w^*) \in \phantom{}_{0}{\mathcal{Y}}_T$ satisfying \eqref{smalldata2}, we define $u^* := (\hat{q}^0 + r^*, \, 1, \,  \mathscr{C}(\hat{v}^0+w^*), \, \hat{v}^0+w^*) \in \mathcal{X}_{T,I}$ (cf. \eqref{ustar}, Section \ref{welldef}). We further consider $(r, \, w) \in \phantom{}_{0}{\mathcal{Y}}_T$, and we denote by $\sigma $ the function obtained via solution of \eqref{equadiff3}. We define $\bar{u} := (r, \, 1, \, \sigma, \, w) \in \phantom{}_{0}{\mathcal{X}}_T$. Then the operators $(g^1)^{\prime}$, $(h^1)^{\prime}$ and $(f^1)^{\prime}$ on the right-hand side of \eqref{equadiff1}, \eqref{equadiff2}, \eqref{equadiff3} satisfy
\begin{align*}
& \|(g^1)^{\prime}(u^*,\, \hat{u}^0) \, \bar{u}\|_{L^p(Q_t)}^p  + \|(h^1)^{\prime}(u^*,\, \hat{u}^0) \, \bar{u}\|_{W^{1,0}_p(Q_t)}^p + \|(f^1)^{\prime}(u^*, \, \hat{u}^0) \, \bar{u}\|_{L^p(Q_t)}^p \\
& \quad \leq K_2^*(t) \,  \int_{0}^t \mathscr{V}^p(s) \, K^*_1(s) \, ds 
\end{align*}
with functions $K^*_1 \in L^1(0,T)$ and $K_2^* \in L^{\infty}(0,T)$. There is a continuous function $\Phi^*(t, \, a_1,a_2)$ defined for all $t, \, a_1, a_2 \geq 0$, such that
\begin{align*}
 \|K^*_1\|_{L^1(0,t)}, \,   \|K^*_2\|_{L^{\infty}(0,t)} \leq & \Phi^*(t, \, \mathscr{V}^*(t),\,  R_{0})
\end{align*}
 for all $t \leq T$, where $\mathscr{V}(t) := \mathscr{V}(t; \, r) + \mathscr{V}(t; \, w)$, $\mathscr{V}^*(t) := \mathscr{V}(t; \, r^*) + \mathscr{V}(t; \, w^*)$ and $R_0$ is defined in \eqref{R0}.
\end{lemma}
\begin{proof}
The estimates of $(g^1)^{\prime}, \, (f^1)^{\prime}$ were performed in \cite{bothedruet} for the corresponding norms. They can be translated one to one to the present context. In adapting the proof, recall also that the numbers $M(\varrho^*, \, T)$ and $M(\hat{\varrho}^0, \, T)$ are finite by construction. We consider here the factor $(h^1)^{\prime}$, which is treated with similar arguments. We recall that
\begin{align*}
h^1 =& h^1(x, \, t,\, q, \, \varrho)\\
= &d(\varrho,\, q) \, (\hat{b}(x, \, t) - \nabla \hat{\zeta}^0(x, \, t)) + A(\varrho, \, q) \, (\tilde{b}(x, \, t) - \nabla \hat{q}^0(x, \, t)) \, ,
\end{align*}
where $\hat{\zeta}^0$ is constructed solving \eqref{Extendzeta0}. The derivatives of $h^1$ are given by the following expressions:
\begin{align*}
h^1_q = d_q \, (\hat{b} - \nabla \hat{\zeta}^0) + A_q \, (\tilde{b} - \nabla \hat{q}^0), \quad h^1_{\varrho} = d_{\varrho} \, (\hat{b} - \nabla \hat{\zeta}^0) + A_{\varrho} \, (\tilde{b} - \nabla \hat{q}^0) \, ,
\end{align*}
while the gradients in $x$ obey
\begin{align*}
 \nabla h^1_q = & (d_{q,q} \, \nabla q  + d_{q,\varrho} \, \nabla \varrho)\, (\hat{b} - \nabla \hat{\zeta}^0)\\
& + (A_{q,q} \, \nabla q + A_{q,\varrho} \, \nabla \varrho) \, (\tilde{b} - \nabla \hat{q}^0) + d_q\, (\nabla \hat{b} - D^2 \hat{\zeta}^0) + A_q \, (\nabla \tilde{b} - D^2 \hat{q}^0) \, , \\
\nabla h^1_{\varrho} =& (d_{\varrho,q} \, \nabla q  + d_{\varrho,\varrho} \, \nabla \varrho)\, (\hat{b} - \nabla \hat{\zeta}^0)\\
& + (A_{\varrho,q} \, \nabla q + A_{\varrho,\varrho} \, \nabla \varrho) \, (\tilde{b} - \nabla \hat{q}^0) + d_{\varrho}\, (\nabla \hat{b} - D^2 \hat{\zeta}^0) + A_{\varrho} \, (\nabla \tilde{b} - D^2 \hat{q}^0) \, .
\end{align*}
Denote by $c_1^*$ a generic function depending on $M(\varrho, \, T)$ and $\|q\|_{L^{\infty}(Q_T)}$. Then the following estimates are obviously valid:
\begin{align*}
|h^1_q| +|h^1_{\varrho}| \leq & c_1^* \, (|\hat{b}| + |\tilde{b}| + |\nabla \hat{\zeta}^0| + |\nabla \hat{q}^0|) \, ,\\
|\nabla h^1_q| + |\nabla h^1_{\varrho}| \leq & c_1^* \, (|\nabla q| + |\nabla  \varrho|) \, (|\hat{b}| + |\tilde{b}| + |\nabla \hat{\zeta}^0| + |\nabla \hat{q}^0|)\\
& + c_1^* \, (|\hat{b}_x| + |\tilde{b}_x| + |D^2_{x,x}\hat{\zeta}^0| + |D^2_{x,x}\hat{q}^0|) \, .
\end{align*}
Using that $W^{1,p}(\Omega) \subset L^{\infty}(\Omega)$, it follows that
\begin{align}\label{K1}
\|h^1_q\|_{L^{\infty,p}(Q_t)} + \|h^1_{\varrho}\|_{L^{\infty,p}(Q_t)} \leq & c_1^* \, (\|\hat{b}\|_{W^{1,0}_p(Q_t)} + \|\tilde{b}\|_{W^{1,0}_p(Q_t)} + \|\hat{\zeta}^0\|_{W^{2,0}_p(Q_t)} + \|\hat{q}^0\|_{W^{2,0}_p(Q_t)}) \, , \\
\label{K2}
\|\nabla h^1_q\|_{L^p(Q_t)} + \|\nabla h^1_{\varrho}\|_{L^p(Q_t)} \leq & c_1^* \, (\|\nabla q\|_{L^{p,\infty}(Q_t)} + \|\nabla \varrho\|_{L^{p,\infty}(Q_t)}+1) \times \nonumber \\ 
& \times (\|\hat{b}\|_{W^{1,0}_p(Q_t)} + \|\tilde{b}\|_{W^{1,0}_p(Q_t)} + \|\hat{\zeta}^0\|_{W^{2,0}_p(Q_t)} + \|\hat{q}^0\|_{W^{2,0}_p(Q_t)}) \, .
\end{align}
Next we turn to estimate $(h^1)^{\prime}(u^*, \, \hat{u}^0) \, \bar{u}$ in $W^{1,0}_p(Q_t)$. At first we notice that
$(h^1)^{\prime}(u^*, \, \hat{u}^0) \, \bar{u} = h^1_q(u^*,\, \hat{u}^0) \, r + h^1_{\varrho}(u^*,\, \hat{u}^0) \, \sigma$. Thus
\begin{align}\label{prl1}
\|(h^1)^{\prime}(u^*, \, \hat{u}^0) \, \bar{u} \|_{L^p(Q_t)}^p \leq &\int_0^t (|h^1_q(u^*,\, \hat{u}^0)|_{L^{\infty}}^p +| h^1_{\varrho}(u^*,\, \hat{u}^0)|_{L^{\infty}}^p) \, (|r|_{L^p}^p + |\sigma|_{L^p}^p) \, d\tau \,\nonumber \\
\leq & c_1^* \, \int_0^t K_1^*(\tau) \, (|r|_{L^p}^p + |\sigma|_{L^p}^p) \, d\tau \,
\end{align}
with $c_1^* = c_1(M(\varrho^*, \, T), \, M(\hat{\varrho}^0, \, T), \, \|q^*\|_{L^{\infty}(Q_T)}, \, \|\hat{q}^0\|_{L^{\infty}(Q_T)})$, and
\begin{align*}
K_1^*(\tau) := \|\hat{b}(\tau)\|_{W^{1,p}(\Omega)} + \|\tilde{b}(\tau)\|_{W^{1,p}(\Omega)} + \|\hat{\zeta}^0(\tau)\|_{W^{2,p}(\Omega)} + \|\hat{q}^0(\tau)\|_{W^{2,p}(\Omega)} \, .
\end{align*}
The function $K^*_1$ is integrable on $(0, \, t)$ with norm bounded by a function $\Phi^*_t$ of the required structure. Estimating $\|q^*\|_{L^{\infty}(Q_t)} \leq \|\hat{q}^0\|_{L^{\infty}(Q_t)} + C \, t^{\gamma} \, \mathscr{V}(r^*; \, t)$, we see that
\begin{align*}
\|(h^1)^{\prime}(u^*, \, \hat{u}^0) \, \bar{u} \|_{L^p(Q_t)}^p \leq & \phi(t, \, \mathscr{V}^*(t), \, R_0) \, \int_{0}^t K_1^*(\tau) \, (\|r(\tau)\|_{L^p(\Omega)}^p + \|\sigma(\tau)\|_{L^p(\Omega)}^p) \, d\tau \\
\leq & \phi(t, \, \mathscr{V}^*(t), \, R_0) \, \int_{0}^t K_1^*(\tau) \, (\|r(\tau)\|_{L^p(\Omega)}^p + c_0 \, \|\sigma_x(\tau)\|_{L^p(\Omega)}^p) \, d\tau \, .
\end{align*}
For the terms containing $\sigma_x$, we use the result of Lemma \ref{PerturbConti}. It yields for $\tau \leq t$ that, in particular,
\begin{align}\label{prl2}
 & &  \|\sigma(\tau)\|_{W^{1,p}(\Omega)} \leq &  K_3(\tau) \, \|w\|_{L^{\infty}(Q_{\tau})}^p + K_4(\tau) \,  \|w\|_{W^{2,0}_p(Q_{\tau})}^p\\
\text{ with } & &  K_3(\tau) := & C \, e^{c \, \int_0^{\tau} [\|v^*_x\|_{L^{\infty}(\Omega)} + \|v^*_{x,x}\|_{L^p(\Omega)} + 1] \, ds} \,  \|\hat{\varrho}^0\|_{W^{2,0}_p(Q_{\tau})}^p \, ,\nonumber\\
 & & K_4(\tau) := & C \, e^{c \, \int_0^{\tau} [\|v^*_x\|_{L^{\infty}(\Omega)} + \|v^*_{x,x}\|_{L^p(\Omega)} + 1] \, ds} \,  \|\hat{\varrho}^0\|_{W^{1,1}_{p,\infty}(Q_{\tau})}^p\nonumber \, .
\end{align}
Since $\|w\|_{L^{\infty}(Q_{\tau})} \leq \bar{c} \, \sup_{s \leq \tau} \|w(s)\|_{W^{2-2/p}_p(\Omega)}$, we obtain that
\begin{align*}
\int_{0}^t K_1^*(\tau) \, \|\sigma_x(\tau)\|_{L^p(\Omega)}^p \, d\tau \leq & \max\{K_3(t), \,K_4(t)\} \, \int_{0}^t K_1^*(\tau) \, [\|w\|_{L^{\infty}(Q_{\tau})}^p + \|w\|_{W^{2,0}_p(Q_{\tau})}^p] \, d\tau \, \\
\leq & \max\{K_3(t), \,K_4(t)\} \, (1+\bar{c})^p \, \int_{0}^t K_1^*(\tau) \, \mathscr{V}^p(w; \, \tau) \, d\tau \, .
\end{align*}
Thus for $K_2^*(t) := C \, \phi(t, \, \mathscr{V}^*(t), \, R_0) \, \max\{K_3(t), \,K_4(t), \, 1\}$, it follows that
\begin{align*}
& \|(h^1)^{\prime}(u^*, \, \hat{u}^0) \, \bar{u} \|_{L^p(Q_t)}^p \leq  \phi(t, \, \mathscr{V}^*(t), \, R_0) \, \int_{0}^t K_1^*(\tau) \, \|r(\tau)\|_{L^p(\Omega)}^p \, d\tau \\
& \quad +\phi(t, \, \mathscr{V}^*(t), \, R_0) \, c_0 \, \max\{K_3(t), \,K_4(t)\} \, (1+\bar{c})^p \, \int_{0}^t K_1^*(\tau)   \, \mathscr{V}^p(w; \, \tau) \, d\tau \,\\
& \phantom{\|(h^1)^{\prime}(u^*, \, \hat{u}^0) \, \bar{u} \|_{L^p(Q_t)}^p }   \leq  K_2^*(t) \, \int_{0}^t K_1^*(\tau)   \, \mathscr{V}^p(\tau) \, d\tau\, .
\end{align*}
We can prove a similar estimate for $\|\nabla ((h^1)^{\prime}(u^*, \, \hat{u}^0) \, \bar{u}) \|_{L^p(Q_t)}^p$. First we notice that
\begin{align*}
\nabla ((h^1)^{\prime}(u^*, \, \hat{u}^0) \, \bar{u}) =& \nabla (h^1_q(u^*,\, \hat{u}^0) \, r + h^1_{\varrho}(u^*,\, \hat{u}^0) \, \sigma)\\
= &\nabla h^1_q(u^*,\, \hat{u}^0) \, r + \nabla h^1_{\varrho}(u^*,\, \hat{u}^0) \, \sigma + h^1_q(u^*,\, \hat{u}^0) \, \nabla r + h^1_{\varrho}(u^*,\, \hat{u}^0) \, \nabla \sigma \, .
\end{align*}
As before (see \eqref{prl1}), $\|h^1_q \, \nabla r + h^1_{\varrho} \, \nabla \sigma \|_{L^p(Q_t)}^p \leq c_1^* \, \int_0^t K_1^*(\tau) \, (\|r_x\|_{L^p(\Omega)}^p + \|\sigma_x\|_{L^p(\Omega)}^p) \, d\tau$. Treating $\sigma_x$ as in \eqref{prl2}, we obtain that $\|h^1_q \, \nabla r + h^1_{\varrho} \, \nabla \sigma \|_{L^p(Q_t)}^p \leq K_2^*(t) \, \int_{0}^t K_1^*(\tau)   \, \mathscr{V}^p(\tau) \, d\tau$.

On the other hand, \eqref{K2} yields
\begin{align*}
&\|\nabla h^1_q(u^*,\, \hat{u}^0) \, r + \nabla h^1_{\varrho}(u^*,\, \hat{u}^0) \, \sigma\|_{L^p(Q_t)}^p \\
\leq &c_1^* \,(\|\nabla q^*\|_{L^{p,\infty}(Q_t)} + \|\nabla \hat{q}^0\|_{L^{p,\infty}(Q_t)} + \|\nabla \varrho^*\|_{L^{p,\infty}(Q_t)} + \|\nabla \hat{\varrho}^0\|_{L^{p,\infty}(Q_t)}+1)^p \, \times\\
&\times \int_0^t (\|\hat{b}\|_{W^{1,p}(\Omega)} + \|\tilde{b}\|_{W^{1,p}(\Omega)} + \|\hat{\zeta}^0\|_{W^{2,p}(\Omega)} + \|\hat{q}^0\|_{W^{2,p}(\Omega)})^p \, (\|r\|_{L^{\infty}(\Omega)}^p + \|\sigma\|_{L^{\infty}(\Omega)}^p) \, d\tau\, .
\end{align*}
This implies that
\begin{align*}
\|\nabla h^1_q \, r + \nabla h^1_{\varrho} \, \sigma\|_{L^p(Q_t)}^p \leq \tilde{K}^*_2(t) \,  \int_0^t \tilde{K}^*_1(\tau) \, (\|r(\tau)\|_{L^{\infty}(\Omega)}^p + c_0 \, \|\sigma_x(\tau)\|_{L^{p}(\Omega)}^p) \, d\tau \, .
\end{align*}
Again, we treat $\sigma_x$ by means of \eqref{prl2}. The claim follows.
 \end{proof}
 Next we prove the main continuity estimate.
We apply Proposition \ref{qzetasystem} to \eqref{equadiff1}, \eqref{equadiff2}. Making use of the fact that $r(0, \, x) =0$ in $\Omega$, we get the estimate
\begin{align}\label{ESTIMrchi}
 \mathscr{V}(t; \, r) + \|\chi\|_{W^{2,0}_p(Q_t)} \leq & C \, \tilde{\Psi}_{1,T}  \, (\|g^1\|_{L^p(Q_t)} + \|(h^1)^{\prime}(u^*,\hat{u}^0) \, \bar{u}\|_{W^{1,0}_p(Q_t)} +\|w\|_{W^{1,0}_p(Q_t)})\\ 
 \leq &  C \, \tilde{\Psi}_{1,T} \, (\|\hat{g}^0\|_{L^p(Q_t)}  + \|w\|_{W^{1,0}_{p}(Q_t)}) \nonumber\\
 & + C \, \tilde{\Psi}_{1,T} \, (\|(g^1)^{\prime}(u^*,\hat{u}^0) \, \bar{u}\|_{L^p(Q_t)} +  \|(h^1)^{\prime}(u^*,\hat{u}^0) \, \bar{u}\|_{W^{1,0}_p(Q_t)}) \nonumber  \, .
\end{align}
Here $\tilde{\Psi}_{1,T} = \max\{\Psi_{1,T}, \, \Phi_T\}$ depends continuously on the data. We then apply Proposition \ref{vsystem} to \eqref{equadiff4} and obtain
\begin{align}\label{ESTIMw}
 \mathscr{V}(T; \, w) \leq & C \, \tilde{\Psi}_{2,T} \, (\|f^1\|_{L^p(Q_t)} + \|\nabla\chi\|_{L^p(Q_t)}) \nonumber\\
 \leq & C \, \tilde{\Psi}_{2,T} \, (\|\hat{f}^0\|_{L^p(Q_t)} + \|\nabla\chi\|_{L^p(Q_t)} + \|(f^1)^{\prime}(u^*,\hat{u}^0) \, \bar{u}\|_{L^p(Q_t)}) \, ,
\end{align}
again with some $\tilde{\Psi}_{2,T}$ depending on $T$ and $\sup_{s\leq t} [\varrho^*(s)]_{C^{\alpha}(\Omega)})$. We estimate $\|\nabla\chi\|_{L^p(Q_t)}$ by means of \eqref{ESTIMrchi}. We next raise both \eqref{ESTIMrchi} and \eqref{ESTIMw} to the $p^{\text{th}}$ power, add both inequalities, and get for the function $\mathscr{V}(t) := \mathscr{V}(t; \, r) + \mathscr{V}(t; \, w) +  \|\chi\|_{W^{2,0}_p(Q_t)} $ the inequality 
\begin{align*}
 \mathscr{V}^p(t) \leq& C \, (\tilde{\Psi}_{1,T}^p + \tilde{\Psi}_{2,T}^p) \, ( \|\hat{g}^0\|^p_{L^p(Q_T)} + \|\hat{f}^0\|^p_{L^p(Q_T)} + \|w\|_{W^{1,0}_p(Q_t)}^p \\
 & +  \|(g^1)^{\prime}(u^*, \, \hat{u}^0) \, \bar{u}\|_{L^p(Q_t)}^p + \|(f^1)^{\prime}(u^*, \, \hat{u}^0) \, \bar{u}\|_{L^p(Q_t)}^p+  \|(h^1)^{\prime}(u^*,\hat{u}^0) \, \bar{u}\|_{W^{1,0}_p(Q_t)}^p)  \, .
\end{align*}
Then we make use of $\|w\|_{W^{1,0}_p(Q_t)}^p \leq \int_0^t \mathscr{V}^p(s) \, ds$, and we apply the Lemma \ref{RightHandControlsecond} to find that
\begin{align}\label{candidateforgronwall}
  \mathscr{V}^p(t) \leq & C \, (\tilde{\Psi}_{1,T}^p + \tilde{\Psi}_{2,T}^p) \, \left(\|\hat{g}^0\|^p_{L^p(Q_T)} + \|\hat{f}^0\|^p_{L^p(Q_T)} + K_2^*(t) \, \int_0^t K^*_1(s) \,   \mathscr{V}^p(s) \, ds\right) \, .
\end{align}
The Gronwall inequality implies that
\begin{align*}
\mathscr{V}^p(t) \leq & C \, (\tilde{\Psi}_{1,T}^p + \tilde{\Psi}_{2,T}^p) \, \exp\left[C \, (\tilde{\Psi}_{1,T}^p + \tilde{\Psi}_{2,T}^p) \, K^*_{2}(t) \, \int_0^t K^*_1(s)\, ds\right] \times\\
& \times (\|\hat{g}^0\|^p_{L^p(Q_T)} + \|\hat{f}^0\|^p_{L^p(Q_T)}) \, .
\end{align*}
We thus have proved the following continuity estimate.
\begin{prop}\label{fixednew}
We define $R_0$ via \eqref{R0}.
 Suppose that $(r^*, \, w^*) \in \phantom{}_{0}{\mathcal{Y}}_T$ satisfy the condition \eqref{smalldata2}. Then $(r, \, w) = \mathcal{T}^1(r^*, \, w^*)$ is well defined in $\phantom{}_{0}{\mathcal{Y}}_T$. Moreover, there is a continuous function $\Psi_7 = \Psi_7(T, \, R_0, \, \eta)$, increasing of all arguments and finite for all $\phi_0(T, \, R_0,\, \eta) \, \eta < \frac{1}{4M(\varrho^{\text{eq}}, \, 0)}$ such that
 \begin{align*}
  \mathscr{V}(T) \leq & \Psi_7(T, \, R_0, \, \mathscr{V}^*(T)) \,  (\|\hat{g}^0\|_{L^p(Q_T)} + \|\hat{f}^0\|_{L^p(Q_T)}) \, .
 \end{align*}
\end{prop}

\subsection{Existence of a unique fixed-point of $\mathcal{T}^1$}\label{FixedPointT1}

We are now in the position to prove a self-mapping property for sufficiently 'small data'. We recall the definitions \eqref{R1}, \eqref{R0} of the critical norms $R_0, \, R_1$.
We denote $u^{\text{eq}} = (q^{\text{eq}}, \, \zeta^{\text{eq}}, \, \varrho^{\text{eq}}, \, v^{\text{eq}})$ and let $\hat{u}^0 := (\hat{q}^{0}, \, \hat{\zeta}^{0}, \, \hat{\varrho}^{0}, \, \hat{v}^{0})$ and $\hat{u}^1 :=  u^{\text{eq}}-\hat{u}^0$. 
In \eqref{distanceglobal}, \eqref{distanceglobal2}, we just proved that $\|\hat{u}^1\|_{\mathcal{X}_T} \leq C \, R_1$.
Recalling that the operator $\widetilde{\mathscr{A}}$ is continuously differentiable into the space $\mathcal{Z}_T$ defined in \eqref{IMAGESPACE}, and that $\widetilde{A}(u^{\text{eq}}) = 0$ by the definition of an equilibrium solution, we can verify that
\begin{align*}
\widetilde{ \mathscr{A}}(\hat{u}^0) =  & \widetilde{ \mathscr{A}}(\hat{u}^{\text{eq}} + \hat{u}^1) = \widetilde{ \mathscr{A}}(\hat{u}^{\text{eq}} + \hat{u}^1) - \widetilde{ \mathscr{A}}(\hat{u}^{\text{eq}})= \int_{0}^1  \widetilde{ \mathscr{A}}^{\prime}(\hat{u}^{\text{eq}} + \theta \, \hat{u}^1) \, d\theta \, \hat{u}^1 \, .
\end{align*}
Thus $ \|\widetilde{ \mathscr{A}}(\hat{u}^0)\|_{Z_T} \leq C \, R_1$. The definitions of $\hat{g}^0$ and $\hat{f}^0$ in \eqref{glinear2} show that
\begin{align}\label{finalepertub}
 & \|\hat{g}^0\|_{L^p(Q_T)}  +\|\hat{f}^0\|_{L^p(Q_T)}  =  \|\widetilde{ \mathscr{A}}^1(\hat{u}^0)\|_{L^p(Q_T)}  + \|\widetilde{ \mathscr{A}}^4(\hat{u}^0)\|_{L^p(Q_T)} \leq \bar{C} \, R_1 \, . 
\end{align}
These considerations allow to state and prove the main properties of $\mathcal{T}^1$.
\begin{lemma}
We define $R_0$ via \eqref{R0} and $R_1$ via \eqref{R1}. For $\phi_0$ and $a_0$ defined in \eqref{smalldata2}, we define $\eta_0 >0$ as the smallest positive number such that $\phi_0(T, \, R_0, \, \eta_0) \, \eta_0 = a_0$. We define $\bar{R}_1 = \min\{1/(2CM(\varrho^{\text{eq}},0)), \, \eta_0/(\bar{C} \, \Psi_7(T, \, R_0, \, \eta_0))\}$ with $\Psi_7$ from Proposition \ref{fixednew}, $C$ from \eqref{smalldata}, and $\bar{C}$ from \eqref{finalepertub}. If $R_1 \leq \bar{R}_1$, the map $\mathcal{T}^1$ is well defined and possesses a unique fixed-point.
\end{lemma}
\begin{proof}
If $w^*$ satisfies \eqref{smalldata2} and if $R_1$ satisfies \eqref{smalldata}, $\mathcal{T}^1(r^*, \, w^*)$ is well defined in $\mathcal{Y}_T$. We apply Proposition \ref{fixednew}, use \eqref{finalepertub} and obtain
\begin{align*}
\|\mathcal{T}^1(r^*, \, w^*)\|_{\mathcal{Y}_T} \leq & \Psi_7(T, \,R_0 , \, \|(r^*, \, w^*)\|_{\mathcal{Y}_T} ) \, (\|\hat{g}^0\|_{L^p(Q_T)} + \|\hat{f}^0\|_{L^p(Q_T)}) \leq  \eta_0 \,.
\end{align*}
We consider the iteration $\bar{u}^{n+1} := \mathcal{T}^1(\bar{u}^n)$, starting at zero. The sequence $(q^{n}, \, \zeta^n, \, \varrho^n, \, v^n)$ is then uniformly bounded in $\mathcal{X}_T$. We show the contraction property with respect to the same lower--order norm than in Lemma \ref{iter}. There are $k_0, \, p_0 > 0$ such that the quantities 
 \begin{align*}
 &  E^n(t) :=   p_0 \, \int_{t}^{t+t_1} \{|\nabla (r^{n}-r^{n-1})|^2 +|\nabla (\chi^{n} - \chi^{n-1})|^2 + |\nabla (w^{n}-w^{n-1})|^2\} \, dxds\\
  &+  k_0 \, \sup_{\tau \in [t, \, t + t_1]} \{\|(r^n-r^{n-1})(\tau)\|_{L^2(\Omega)}^2 + \|(\sigma^n-\sigma^{n-1})(\tau)\|_{L^2(\Omega)}^2 +\|(w^n-w^{n-1})(\tau)\|_{L^2(\Omega)}^2\}
 \end{align*}
satisfy $E^{n+1}(t) \leq\tfrac{1}{2} \, E^{n}(t)$ for some fixed $t_1 > 0$ and every $t \in [0, \, T-t_1]$. 
\end{proof}

\appendix

\section{Properties of the free energy density}\label{FE}

In this section we prove the statements of the section \ref{PrelimFE} devoted to the convex conjugate of the free energy density: Lemma \ref{smoothness}, Lemma \ref{Hessianandgradient} and Lemma \ref{functionf2}. 

We assume that $k$ satisfies the assumptions of Lemma \ref{smoothness}. Notice that requiring $k$ essentially smooth on $S_1$, while positive homogeneous, induces that $k$ is also essentially smooth on $S_{\bar{V}}$. To see this, we consider any sequence $\{r^m\} \subset S_{\bar{V}}$ such that $r^m \rightarrow \bar{r}$ for $m \rightarrow \infty$, and $\bar{r}$ belongs to the relative boundary of $S_{\bar{V}}$, which means that there is $i \in \{1,\ldots,N\}$ such that $\bar{r}_i = 0$. Then we define $y^m := \bar{r}_m/\sum_{i=1}^N r^m_i$ which belongs to $S_1$ for all $m$, and satisfies $y^m_i \rightarrow 0$ for $m \rightarrow \infty$. Since $k$ is positively homogeneous, we have $\nabla_{\rho} k(r^m) = \nabla_{\rho} k(y^m)$. Thus, by the assumptions of Lemma \ref{smoothness}, we see that $|\nabla_{\rho}k(r^m)| \rightarrow + \infty$, which is the essential smoothness on $S_{\bar{V}}$.

Consider now $\mu \in \mathbb{R}^N$ arbitrary. Then we claim first that there exists a unique $\bar{r} \in S_{\bar{V}}$ such that
\begin{align*}
f(\mu) = \sup_{r\in S_{\bar{V}}} \{\mu \cdot r - k(r)\} =  \mu \cdot \bar{r} - k(\bar{r}) \, .
\end{align*}
Since $S_{\bar{V}}$ is bounded, we first notice that $ \sup_{r\in S_{\bar{V}}} \{\mu \cdot r - k(r)\} = \max_{r \in \overline{S_{\bar{V}}}}  \{\mu \cdot r - k(r)\}$. Thus, there is $\bar{r} \in \overline{S_{\bar{V}}}$ such that $\sup_{r\in S_{\bar{V}}} \{\mu \cdot r - k(r)\} = \mu \cdot \bar{r} - k(\bar{r})$. We want to show that $\bar{r}$ is an interior point. Since $S_{\bar{V}}$ is a convex set, we can find for every $a \in S_{\bar{V}}$ a $h > 0$ such that $\bar{r} + h \, (a - \bar{r}) \in S_{\bar{V}}$. Due to the choice of $\bar{r}$
\begin{align*}
 \mu \cdot (\bar{r} + h \, (a - \bar{r})) - k(\bar{r} + h \, (a - \bar{r})) \leq \mu \cdot \bar{r} - k(\bar{r}) \, ,
\end{align*}
which yields $k(\bar{r} + h \, (a - \bar{r})) - k(\bar{r}) \geq h \, \mu \cdot (a-\bar{r})$ and $ \lim_{h \searrow 0} \frac{k(\bar{r} + h \, (a - \bar{r})) - k(\bar{r})}{h} > -\infty$. The latter however contradicts the fact that $k$ is \emph{essentially smooth} on $S_{\bar{V}}$ (cf. \cite{rockafellar}, Lemma 26.2). Thus, $\bar{r} \in S_{\bar{V}}$ is an interior point.

The uniqueness of $\bar{r}$ follows from the strict convexity of $k$ on $S_{\bar{V}}$.

Since $k$ is differentiable, and since $r \mapsto \mu \cdot r - k(r)$ attains its maximum in $\bar{r}$, we must have $(\nabla k(\bar{r}) - \mu) \cdot \xi = 0$ for every tangential vector $\xi \in \mathbb{R}^N$ such that $\xi \cdot \bar{V} = 0$. Thus, there is $p \in \mathbb{R}$ such that $\mu = \nabla_{\rho} k(\bar{r}) + p \, \bar{V}$. Multiplying with $\bar{r}$, use of the homogeneity of degree one implies that $\bar{r} \cdot \nabla_{\rho} k(\bar{r}) = k(\bar{r})$, hence
\begin{align}\label{mch}
\sup_{r\in S_{\bar{V}}} \{\mu \cdot r - k(r)\}  =  \mu \cdot \bar{r} - k(\bar{r}) = p \, \bar{r} \cdot \bar{V} = p \, , 
\end{align}
showing that $p = f(\mu)$. Due to the structure $f(\mu) = \mu \cdot \bar{r} - k(\bar{r}) = \max_{r\in S_{\bar{V}}} \{\mu \cdot r - k(r)\}$, we easily show that $f$ is differentiable in $\mu$ with $\nabla_{\mu}f(\mu) = \bar{r}$. In order to show the differentiability of higher order, we can exploit the identities
\begin{align*}
 \mu - f(\mu) \, \bar{V} = \nabla_{\rho}k(\nabla_{\mu}f(\mu)), \quad \quad \bar{V} \cdot \nabla_{\mu} f(\mu) = 1 \, .
\end{align*}
For a system of orthonormal vectors $\xi^1,\ldots,\xi^{N-1}$ for $\{\bar{V}\}^{\perp}$, and $\xi^N := \bar{V}/|\bar{V}|$, we then have
\begin{align*}
 \mu\cdot \xi^j = & \xi^j \cdot \nabla_{\rho}k( \sum_{i=1}^{N-1} \xi^i \cdot \nabla_{\mu}f(\mu) \, \xi^i+ \frac{\xi^N}{|\bar{V}|}) \text{ for } j =1,\ldots,N-1, \\
\frac{1}{|\bar{V}|} =& \xi^N \cdot   \nabla_{\mu} f(\mu)  \, .
\end{align*}
The latter can be viewed as an algebraic system of the form $F(X) = (\mu \cdot \xi^1, \ldots,\mu\cdot \xi^{N-1},\, \frac{1}{|\bar{V}|})$ for the unknowns $X := (\xi^1 \cdot \nabla_{\mu}f(\mu), \ldots, \xi^N  \cdot \nabla_{\mu}f(\mu)) \in \mathbb{R}^N$. The Jacobian of this system obeys
\begin{align*}
 \frac{\partial F_j}{\partial X_{i}} = \begin{cases}
 D^2k \xi^i \cdot \xi^j  & \text{ for } i=1,\ldots,N-1, \, j = 1,\ldots,N-1\, ,\\
 0 & \text{ for } i = N, \, j = 1,\ldots,N-1 \, , \\
 \delta_{i, \,N} & \text{ for }i = 1,\ldots,N, \, j = N \, ,\\
  \end{cases}
\end{align*}
where $D^2k$ is evaluated at $\nabla_{\mu} f(\mu)$. We can easily verify that $\{D^{2}k(\nabla_{\mu} f(\mu)) \xi^i \cdot \xi^j\}_{i,j=1,\ldots,N-1}$ is strictly positive definite: A vector of the form $\sum_{j=1}^{N-1} \xi^j \, a_j$, $a \neq 0$ can never be parallel to $\nabla_{\mu}f(\mu)$, since multiplying with $\bar{V}$ yields a contradiction. On the other hand, the properties of $k$ guarantee that the kernel of $D^{2}k(\nabla_{\mu} f(\mu))$ is the one-dimensional span of $\nabla_{\mu} f(\mu)$.

Thus, the equations $F(X(\mu)) = (\mu \cdot \xi^1, \ldots,\mu\cdot \xi^{N-1},\, \frac{1}{|\bar{V}|})$ define implicitly a map $\mu \mapsto X(\mu)$ of class $C^1(\mathbb{R}^N)$. This clearly implies that $f \in C^2(\mathbb{R}^N)$, and we obtain the formula
\begin{align}\label{mch2}
D^2_{\mu_k,\mu_i} f(\mu) = \sum_{j=1}^N \frac{\partial X_j}{\partial \mu_k} \, \xi^j_i \, .
\end{align}
If $k \in C^3(\mathbb{R}^N_+)$, we then differentiate again to obtain that $f$ is $C^3(\mathbb{R}^N)$. This proves the claims of Lemma \ref{smoothness}. The claims of Lemma \ref{Hessianandgradient} and \ref{functionf2} are also readily established (use \eqref{mch2} and \eqref{mch}).

\section{Auxiliary statements}

For the proof of the following Lemma, we need the variable transformation in the section \ref{changevariables}.
\begin{lemma}\label{special}
We adopt the assumptions of Theorem \ref{MAIN} for the tensor $M: \, \mathbb{R}^N_+ \rightarrow \mathbb{R}^{N\times N}$, and we assume that $k: \, \mathbb{R}^N_+ \rightarrow \mathbb{R}$ is given by \eqref{kfunktion}. We assume moreover that there is a continuous function $C = C(|\rho|)$, bounded on compact subsets of $\overline{\mathbb{R}}^N \setminus \{0\}$, such that $B_{i,j}(\rho) := M_{i,j}(\rho) /\rho_j$, with entries belonging to $C^1(\mathbb{R}^N_+)$, satisfies for all $\rho \in \mathbb{R}^N_+$ the conditions
\begin{align*}
|B_{i,j}(\rho)| + \rho_k \, |B_{i,j,\rho_k}(\rho)| \leq C(|\varrho|) \text{ for all } i, \, j, \, k \in \{1,\ldots,N\} \,.
\end{align*}
For $\varrho \in I$ and $q \in \mathbb{R}^{N-2}$, we denote $M(\varrho, \, q) := M(\sum_{\ell=1}^{N-2} R_{\ell}(\varrho, \, q) \, \eta^{\ell} + \varrho \, \eta^N)$, and recall the definitions \eqref{trafom}, \eqref{trafoa}, \eqref{trafod} of the objects $\widetilde{M}(\varrho, \, q)$, $A(\varrho, \, q)$, $d(\varrho, \, q)$ and the definition \eqref{pressuredef} of the non-linear part $P(\varrho, \, q)$ of the pressure. For $\varrho \in I$, we define $m(\varrho) := \min\{1 - \varrho/\varrho_{\max}, \, \varrho/\varrho_{\min}-1\}$. Then the following statements are valid: For all $\varrho \in I$ and $q \in \mathbb{R}^{N-2}$
\begin{itemize}
 \item $ |d(\varrho, \, q)| + |A(\varrho, \, q)| + |d_q(\varrho, \, q)| + |A_q(\varrho, \, q)| \leq c_1 \, m(\varrho)$;
\item $| d_{\varrho}(\varrho, \, q)| + |A_{\varrho}(\varrho, \, q)| \leq c_2$;
\item The function $P_{\varrho}$ is positive and $c_3 \, (m(\varrho))^{-1} \leq P_{\varrho}(\varrho, \, q) \leq c_4 \,(m(\varrho))^{-1}  $ with $c_3 > 0$; Moreover $|P_q(\varrho, \, q)| \leq c_5$.
\end{itemize}
\end{lemma}
\begin{proof}
For $\rho \in S_{\bar{V}}$ we consider the vector $ u_j = - \rho_j \, (\bar{V}_j -1/\varrho_{\min})$ for $j = 1,\ldots,N$. By the definition of $\varrho_{\min}$, all components of $u$ are positive. Moreover $\sum_{j=1}^N u_j = \varrho/\varrho_{\min} - 1$ by the definition of $S_{\bar{V}}$. Since $M \, 1^N = 0$, the identity $ M(\rho) \, \bar{V} = M(\rho) \, (\bar{V} - 1^N/\varrho_{\min})= - B(\rho) \, u $ holds. By assumption $|B(\rho)| \leq C(|\rho|)$, and therefore
\begin{align*}
 | M(\rho) \, \bar{V}| \leq C(|\rho|) \, |u| \leq C_0 \, (\frac{\varrho}{\varrho_{\min}} - 1) \, .
\end{align*}
Analogously, considering next $u_j := \rho_j \, (\bar{V}_j - 1/\varrho_{\max})$, we obtain that $| M(\rho) \, \bar{V}| \leq C_1 \, (1-\varrho/\varrho_{\max})$, and overall that $| M(\rho) \, \bar{V}|  \leq C \, m(\varrho)$.

We next investigate the derivatives. To do so, we recall two properties of the map $\mathscr{R}(\varrho, \, q)$ (Section \ref{changevariables}, \eqref{rhomap}). For $\ell = 1,\ldots,N-2$ direct computations yield for $i=1,\ldots,N$
\begin{align*}
 \partial_{q_{\ell}} \mathscr{R}_i(\varrho, \, q) =  &  D^2f e^i \cdot \xi^{\ell} - \frac{D^2 f e^i \cdot 1^N \, D^2f \xi^{\ell} \cdot 1^N}{D^2f 1^N \cdot 1^N} \quad \text{ for } \ell = 1,\ldots, N-2\,  ,\\
 \partial_{\varrho} \mathscr{R}_i(\varrho, \, q) = & \frac{D^2 f e^i \cdot 1^N}{D^2f 1^N \cdot 1^N} \, .
 \end{align*}
In these formula, we evaluate $D^2f$ at $\mu = \sum_{\ell = 1}^{N-2} q_{\ell} \, \xi^{\ell} + \mathscr{M}(\varrho, \, q) \, 1^N$. In the Section 4 of \cite{druetmixtureincompweak}, we prove that $D^2 f e^i \cdot 1^N \leq C_0 \, D^2f 1^N \cdot 1^N$ for $i = 1,\ldots,N$ (Lemma 4.3 (e)). Moreover, $|D^2f e^i \cdot a| \leq c_a \, \rho_i$ for any vector $a$ (cf. Lemma 4.3 (a)). From these properties, we infer that
\begin{align*}
\frac{1}{\rho_i} \, |\partial_{q} \mathscr{R}_i(\varrho, \, q)| \leq \frac{|D^2f e^i|}{\rho_i} \, (1 + C_0 \, \max_{\ell = 1,\ldots, N-2} |\xi^{\ell}|) \leq c \, , \qquad |\partial_{\varrho} \mathscr{R}_i(\varrho, \, q) | \leq c \, .
\end{align*}
We again express $M_{i,j}(\rho) \, \bar{V}_j = - B_{i,j}(\rho) \, \rho_j \, (\bar{V}_j - 1/\varrho_{\min})$, hence
\begin{align*}
 \partial_{\rho_k} M_{i,j}(\rho) \, \bar{V}_j =  - B_{i,j,\rho_k}(\rho) \, \rho_j \, (\bar{V}_j - \frac{1}{\varrho_{\min}}) - B_{i,k} \, (\bar{V}_k - \frac{1}{\varrho_{\min}}) \, ,
\end{align*}
and therefore, it follows for $\ell = 1,\ldots,N-2$ that
\begin{align*}
 \partial_{q_{\ell}} M_{i,j}(\mathscr{R}(\varrho, \, q)) \, \bar{V}_j = \sum_{k=1}^N (- B_{i,j,\rho_k}(\rho) \, \rho_j \, (\bar{V}_j - \frac{1}{\varrho_{\min}})- B_{i,k} \, (\bar{V}_k - \frac{1}{\varrho_{\min}})) \, \mathscr{R}_{k,q_{\ell}} \, .
\end{align*}
Since $ - B_{i,j,\rho_k}(\rho) \, \rho_j \, (\bar{V}_j - 1/\varrho_{\min}) = - B_{i,j,\rho_k} \, u_j$, and, by assumption, $|B_{i,j,\rho_k}| \leq C(|\rho|)/\rho_k$, we invoke that $|\mathscr{R}_{k,q_{\ell}}| \leq C \, \rho_{k}$ to show that
\begin{align*}
 | B_{i,j,\rho_k}(\rho) \, u_j \, \mathscr{R}_{k,q_{\ell}}| \leq C_0 \, |u| \leq C_0 \, (\frac{\varrho}{\varrho_{\min}} - 1) \, .
\end{align*}
Moreover, by the same means
\begin{align*}
|B_{i,k} \, (\bar{V}_k - \frac{1}{\varrho_{\min}})) \, \mathscr{R}_{k,q_{\ell}}| \leq |B_{i,k} \, u_k \, (\mathscr{R}_{k,q_{\ell}}/\rho_k)| \leq C_1 \, |u| \leq  C_1 \, (\frac{\varrho}{\varrho_{\min}} - 1) \, .
\end{align*}
Arguing the same for the other choice of $u$, it follows that $|  \partial_{q_{\ell}} M_{i,j}(\mathscr{R}(\varrho, \, q)) \, \bar{V}_j| \leq C \, m(\varrho)$.
The other estimates claimed have been verified for this special case of the function $k$ in the Section 4 of \cite{druetmixtureincompweak}.
\end{proof}
\begin{rem}
In the case that the matrix $M$ results from inversion of the Maxwell-Stefan equations, we notice that the matrix $B$ of Lemma \ref{special} is nothing else but the pseudo-inverse of the Maxwell-Stefan matrix. It is shown in the paper \cite{bothedruetMS} that natural assumptions on the binary diffusivities are sufficient for proving that the entries of $B$ consist of regular functions of the state variables. In particular, they satisfy the assumptions of Lemma \ref{special}.
\end{rem}

The following statement is directly taken from our paper \cite{bothedruet}. There we must only adapt the definition of the parameters $m^*$ and $M^*$ according to \eqref{supremumrho}, \eqref{infimumrho} in order to account for different density thresholds in the incompressible model.
 \begin{prop}\label{qsystem}
Assume that $R_q, \, \widetilde{M}: \, I \times \mathbb{R}^{N-2} \rightarrow \mathbb{R}^{(N-2)\times (N-2)}$ are maps of class $C^1$ into the set of symmetric, positive definite matrices. 
Consider given $q^{*} \in W^{2,1}_p(Q_T; \, \mathbb{R}^{N-2})$ and $\varrho^* \in W^{1,1}_{p,\infty}(Q_T)$ ($p > 3$) such that the values of $\varrho^*$ are strictly contained in $I$ in $\overline{Q_T}$. Let $g \in L^p(Q_T; \, \mathbb{R}^{N-2})$ and $q^0 \in W^{2-2/p}(\Omega)$ such that $\nu \cdot \nabla q^0(x) = 0$ in the sense of traces on $\partial \Omega$. Then, there is a unique $q \in W^{2,1}_p(Q_T; \, \mathbb{R}^{N-2})$ solution to the problem
\begin{align*}
R_q(\varrho^*, \, q^{*}) \,  q_t - \divv (\widetilde{M}(\varrho^*, \, q^{*}) \, \nabla q) = g \, \text{ in } Q_T, \quad \nu \cdot \nabla q = 0 \text{ on } S_T, \, \quad
q(x, \, 0) = q^0(x) \text{ in } \Omega \, ,
\end{align*}
Moreover there is a constant $C$ independent on $T$, $q$, $\varrho^*$ and $q^*$ such that for all $t \leq T$ and $0 < \beta \leq 1$:
\begin{align*}
& \mathscr{V}(t; \, q) \leq C \, \bar{\Psi}_{1,t}\, \left[(1+[\varrho^*]_{C^{\beta,\frac{\beta}{2}}(Q_t)})^{\tfrac{2}{\beta}} \, \|q^0\|_{W^{2-\frac{2}{p}}_p(\Omega)} + \|g\|_{L^p(Q_t)}\right] \, , \\[0.1cm]
& \bar{\Psi}_{1,t} = \bar{\Psi}_1(t, \, M^*(t),  \, \|q^*(0)\|_{C^{\beta}(\Omega)}, \, 
\mathscr{V}(t; \, q^*), \, [\varrho^*]_{C^{\beta,\frac{\beta}{2}}(Q_t)}, \, \|\nabla \varrho^*\|_{L^{p,\infty}(Q_t)}) \, , 
\end{align*}
with a continuous function $\bar{\Psi}_1$ defined for all $t \geq 0$ and all numbers $a_1, \ldots, a_5 \geq 0$. The function $\bar{\Psi}_1$ is increasing in all arguments and moreover $\bar{\Psi}_1(0, \, a_1,\ldots,a_5) = \bar{\Psi}_1^0(a_1, \, a_2, \, a_3)$ is a function independent on the two last arguments.
%
\end{prop}
We also recall some estimates of H\"older norms. This is also proved in \cite{bothedruet}.
\begin{lemma}\label{HOELDER}
 For $0 \leq \beta < \min\{1, \, 2-\frac{5}{p}\}$ we define
\begin{align*}
\gamma := \begin{cases}
            \frac{1}{2} \, (2 - \frac{5}{p} -\beta) & \text{ for } 3<p<5 \, ,\\
            (1-\beta) \,  \frac{p-1}{3+p} & \text{ for } 5 \leq p \, . 
            \end{cases}
            \end{align*}
Then, there is $C = C(t)$ bounded on finite time intervals such that $C(0) = C_0$ depends only on $\Omega$ and for all $q^* \in W^{2,1}_{p}(Q_t)$
 \begin{align*}
 \|q^*\|_{C^{\beta,\frac{\beta}{2}}(Q_t)} \leq \|q^*(0)\|_{C^{\beta}(\Omega)} +  C(t) \,  t^{\gamma} \, [\|q^*\|_{W^{2,1}_{p}(Q_t)} + \|q^*\|_{C([0,t]; \, W^{2-\frac{2}{p}}_p(\Omega)})]  \, .
\end{align*}
\end{lemma}
Finally we have a perturbation Lemma for elliptic problems. This property ought to be well known, and we only mention details for more convenience on reading.
\begin{lemma}\label{ellipticlemma}
Let $a \in C^{\beta}(\overline{\Omega})$ ($\beta > 0$) satisfy $0 < a_0 \leq a(x) \leq a_1 < + \infty$ for all $x \in \Omega$. Suppose that $F \in L^p(\Omega)$ with $p > 3$. Then, there is a unique $u \in W^{1,p}(\Omega)$ satisfying $\int_{\Omega} (a(x) \, \nabla u-F(x)) \cdot \nabla \phi \, dx  = 0$ for all $\phi \in C^1(\overline{\Omega})$ and $\int_{\Omega} u \, dx = 0$. Moreover, there is $c = c(\Omega,p,a_0,a_1)$ such that
\begin{align*}
\|\nabla u\|_{L^p(\Omega)} \leq c \, (1 + [a]_{C^{\beta}})^{\frac{1}{\beta}} \, \|F\|_{L^p(\Omega)} \, .
\end{align*}
\end{lemma}
\begin{proof}
Existence of a unique weak solution is well-known. In order to prove the estimate, we start recalling a few standard inequalities. First, the bound $\sqrt{a_0} \, \|\nabla u\|_{L^2} \leq \|F\|_{L^2}$ is valid. Since we choose the mean-value of $u$ to be zero, then also $\|u\|_{L^2} \leq c_{\Omega} \, a_0^{-\frac{1}{2}} \, \|F\|_{L^2}$. Moreover, for $s > 3$ arbitrary, we find that $\|u\|_{L^{\infty}(\Omega)} \leq c(\Omega,s) \, (a_0^{-1} \, \|F\|_{L^s(\Omega)} + \frac{a_1}{a_0} \, \|u\|_{L^{s/2}(\Omega)})$. Thus, choosing $ s \leq \min\{p, \, 4\}$ and employing the Hoelder inequality, we easily show that $\|u\|_{L^{\infty}(\Omega)} \leq \tilde{c}(\Omega, \, a_0,\, a_1) \, \|F\|_{L^p(\Omega)}$.

We now come to the main argument. We consider $x^0$ in $\Omega$ and $r > 0$. We choose a nonnegative cut-off function $\eta \in C^1_c(B_r(x^0))$ satisfying $|\nabla \eta| \leq c_0 \, r^{-1}$. Choosing in the weak formulation a testfunction of the form $\phi \, \eta$, we obtain, after some obvious shifting, for $w := u \, \eta$ the identity
\begin{align*}
a(x^0) \, \int_{\Omega} \nabla w \cdot \nabla \phi \, dx =& \int_{\Omega} (a(x^0) -a)\, \nabla w \cdot \nabla \phi  \, dx \\
&+\int_{\Omega} (F \, \eta + a \, u \, \nabla \eta) \cdot \nabla \phi  \, dx + \int_{\Omega} (F\cdot \nabla \eta + a \, \nabla \eta \cdot \nabla u) \, \phi  \, dx \, .
\end{align*}
This is a weak Neumann problem for the Laplacian of $w$. By standard results, we obtain an estimate
\begin{align*}
\|\nabla w\|_{L^p} \leq & c(\Omega, \, p) \, \Big( \|(1-\frac{a}{a(x^0)}) \, \nabla w\|_{L^p}\\
 & + \frac{1}{a(x^0)} \, (\|F \, \eta + a \, u \, \nabla \eta\|_{L^p} + \|F\cdot \nabla \eta + a \, \nabla \eta \cdot \nabla u +a(x^0) \,  w\|_{L^{\frac{p^*}{p^*-1}}(\Omega)})\Big) \, .
 \end{align*}
Here $p^*$ is the Sobolev embedding exponent of $W^{1,p}(\Omega)$. For $p >3$, we have $p^* = +\infty$ and $p^*/(p^*-1) = 1$. Next, since $w$ is supported in $B_r(x^0)$, and since $a$ is Hoelderian, it follows that
\begin{align*}
\|(1-\frac{a}{a(x^0)}) \, \nabla w\|_{L^p} \leq \frac{[a]_{C^{\beta}}}{a(x^0)} \, r^{\beta} \,  \|\nabla w\|_{L^p} \, .
\end{align*}
Thus, fixing $r^{\beta} := \frac{a_0}{2c(\Omega, \, p)[a]_{C^{\beta}}}$, we obtain that
\begin{align*}
\|\nabla w\|_{L^p} \leq  \frac{2\, c(\Omega, \, p)}{a(x^0)} \, (\|F \, \eta + a \, u \, \nabla \eta\|_{L^p} + \|F\cdot \nabla \eta + a \, \nabla \eta \cdot \nabla u +a(x^0) \,  w\|_{L^{1}(\Omega)}) \, .
 \end{align*}
With the notation $\Omega_r(x^0) = B_r(x^0) \cap \Omega$, we notice that
\begin{align*}
\|\nabla w\|_{L^p} \geq & \|\nabla u\, \eta\|_{L^p} - \frac{c_0}{r} \, \|u\|_{L^p(\Omega_r(x^0))} \, ,\\
\|F \, \eta + a \,  u \, \nabla \eta\|_{L^p} \leq & \|F \, \eta\|_{L^p} + \frac{a_1 \, c_0}{r} \,  \|u\|_{L^p(\Omega_r(x^0))} \, ,\\
\|F\cdot \nabla \eta + a \, \nabla \eta \cdot \nabla u +a(x^0) \,  w\|_{L^{1}} \leq & \frac{c_0}{r} \, (\|F\|_{L^1(\Omega_r(x^0))}+ a_1 \, \|\nabla u\|_{L^1(\Omega_r(x^0))}) \, .
\end{align*}
 Thus, we have shown that
 \begin{align*}
 \|\nabla u\, \eta\|_{L^p} \leq & \frac{2 \, c(\Omega, \, p)}{a_0} \,  (\|F \, \eta\|_{L^p} +  \frac{c_0}{r} \, (\|F\|_{L^1(\Omega_r(x^0))}+ a_1 \, \|\nabla u\|_{L^1(\Omega_r(x^0))})   ) \\
 & +    \, \frac{c_0}{r} \, (1+2 \, c(\Omega, \, p) \, \frac{a_1}{a_0}) \, \|u\|_{L^p(\Omega_r(x^0))} \, .
 \end{align*}
By appropriate covering of $\Omega$ with partition of unity, we obtain the inequality
\begin{align*}
\|\nabla u\|_{L^p} \leq & \frac{2 \, m_0 \, c(\Omega, \, p)}{a_0} \, \|F\|_{L^p}+ \frac{2\, c_0 \, m_0 \, c(\Omega, \, p)}{a_0 \, r} \, \|F\|_{L^1} \\
&+ \frac{a_1 \, c_0}{a_0\, r} \, m_0 \, (2 \, c(\Omega, \, p) \, \|\nabla u\|_{L^1} +(1+2 \, c(\Omega, \, p) ) \, \|u\|_{L^p}) \, .
\end{align*}
Here $m_0$ is some geometric constant associated with the covering of $\Omega$. It remains to estimate
\begin{align*}
\|\nabla u\|_{L^1} \leq |\Omega|^{\frac{1}{2}} \, \|\nabla u\|_{L^2} \leq  |\Omega|^{\frac{1}{2}}\, a_0^{-\frac{1}{2}} \, \|F\|_{L^2}\\
\|u\|_{L^p} \leq |\Omega|^{\frac{1}{p}} \, \|u\|_{L^{\infty}(\Omega)} \leq c(\Omega,p)  \, \|F\|_{L^p} \, ,
\end{align*}
where we employ the preliminary consideration at the beginning of this proof to show that $\|u\|_{L^{\infty}} \leq c \, \|F\|_{L^p}$. Recalling the choice of $r$, we are done.
\end{proof}
\newcommand{\etalchar}[1]{$^{#1}$}

\end{document}